\documentclass[11pt]{article}
\usepackage[utf8x]{inputenc}
\usepackage[margin=1.5in]{geometry}
\usepackage{microtype}
\usepackage[T1]{fontenc}
\usepackage{ae,aecompl}
\usepackage{times}
\usepackage{tikz-cd}
\usepackage{mathrsfs} 
\usepackage{harmony}
\usepackage{verbatim,amsmath,amsthm,amsfonts,amssymb,latexsym,graphicx,mathtools,extpfeil,color,mathabx}
\usepackage{tikz}
\usepackage{epstopdf,pinlabel}
\usepackage[all]{xy}
\usepackage{graphicx}
\usepackage{booktabs}
\usepackage{caption}
\usepackage[framemethod=TikZ]{mdframed}
\usepackage{tabularx}
\usepackage{subcaption}
\usepackage{slashed}
\DeclareMathAlphabet{\mathpzc}{OT1}{pzc}{m}{it}
\usepackage[colorlinks,pagebackref,hypertexnames=false]{hyperref} \usepackage[alphabetic,backrefs,msc-links]{amsrefs}

\usepackage{aliascnt}
\numberwithin{equation}{section}

\newcommand{\bv}{{\bf v}}

\newcommand{\bz}{{\bf z}}
\newcommand{\bA}{{\bf A}}

\newcommand{\bC}{{\bf C}}

\newcommand{\bP}{{\bf P}}
\newcommand{\bQ}{{\bf Q}}
\newcommand{\bR}{{\bf R}}

\newcommand{\bZ}{{\bf Z}}

\newcommand{\cE}{\mathcal{E}}

\newcommand{\cN}{\mathcal{N}}

\newcommand{\fC}{{\mathfrak C}}

\newcommand{\Z}{\bZ}
\newcommand{\Q}{\bQ}
\newcommand{\R}{\bR}
\newcommand{\C}{\bC}

\renewcommand{\det}{\operatorname{det}}

\renewcommand{\epsilon}{\varepsilon}

\def\({\mathopen{}\left(}
\def\){\right)\mathclose{}}
\def\<{\mathopen{}\left<}
\def\>{\right>\mathclose{}}

\usepackage{multicol, color}

\definecolor{gold}{rgb}{0.85,.66,0}
\definecolor{cherry}{rgb}{0.9,.1,.2}
\definecolor{burgundy}{rgb}{0.8,.2,.2}
\definecolor{orangered}{rgb}{0.85,.3,0}
\definecolor{orange}{rgb}{0.85,.4,0}
\definecolor{olive}{rgb}{.45,.4,0}
\definecolor{lime}{rgb}{.6,.9,0}
\definecolor{green}{rgb}{.2,.7,0}
\definecolor{grey}{rgb}{.4,.4,.2}
\definecolor{brown}{rgb}{.4,.3,.1}

\def\makeautorefname#1#2{\AtBeginDocument{\expandafter\def\csname#1autorefname\endcsname{#2}}}

\newcommand{\mynewtheorem}[2]{
  \newaliascnt{#1}{equation}          
  \newtheorem{#1}[#1]{#2}
  \aliascntresetthe{#1}
  \makeautorefname{#1}{#2}
}
\mynewtheorem{theorem}{Theorem}
\mynewtheorem{prop}{Proposition}
\mynewtheorem{cor}{Corollary}
\mynewtheorem{construction}{Construction}
\mynewtheorem{lemma}{Lemma}
\mynewtheorem{conjecture}{Conjecture}

\numberwithin{substep}{step}
\makeautorefname{step}{Step}
\makeautorefname{substep}{Step}

\numberwithin{subcase}{case}
\makeautorefname{case}{Case}
\makeautorefname{subcase}{case}

\theoremstyle{remark}
\mynewtheorem{remark}{Remark}

\theoremstyle{definition}
\mynewtheorem{definition}{Definition}
\mynewtheorem{example}{Example}
\mynewtheorem{exercise}{Exercise}
\mynewtheorem{convention}{Convention}
\newtheorem*{convention*}{Convention}
\newtheorem*{conventions*}{Conventions}
\mynewtheorem{question}{Question}
\mynewtheorem{assumption}{Assumption}
        
\makeautorefname{chapter}{Chapter}
\makeautorefname{section}{Section}
\makeautorefname{subsection}{Section}
\makeautorefname{subsubsection}{Section}

\theoremstyle{introthm}

\theoremstyle{introcor}

\theoremstyle{introprop}

\theoremstyle{introquestion}

\global\mdfdefinestyle{exampledefault}{%
linecolor=black,backgroundcolor=gray!1,linewidth=1pt,leftmargin=0cm,rightmargin=0cm,topline=false,bottomline=false,skipabove=12pt}

\title{Rank three instantons, representations and sutures}
\author{Aliakbar Daemi, Nobuo Iida and Christopher Scaduto}
\date{}

\newcommand{\Addresses}{{
  \bigskip
  \footnotesize
  Aliakbar Daemi, \textsc{Washington University in St. Louis, MO, USA}\par\nopagebreak
  \textit{E-mail address}: \texttt{adaemi@wustl.edu}
  \vspace{.2cm}
  
  Nobuo Iida, \textsc{Tokyo Institute of Technology, Tokyo, Japan
}\par\nopagebreak
  \textit{E-mail address}: \texttt{iida.n.ad@m.titech.ac.jp}
    \vspace{.2cm}
    
Christopher Scaduto, \textsc{University of Miami, Coral Gables, FL, USA}\par\nopagebreak
  \textit{E-mail address}: \texttt{cscaduto@miami.edu}
}}

\begin{document}
\maketitle

\begin{abstract}
	We show that the knot group of any knot in any integer homology sphere admits a non-abelian representation 
	into $SU(3)$ such that meridians are mapped to matrices whose eigenvalues are the three distinct third roots of unity.
	This answers the $N=3$ case of a question posed by Xie and the first author. We also characterize when a $PU(3)$-bundle 
	admits a flat connection. The key ingredient in the proofs is a study of the 
	ring structure of $U(3)$ instanton Floer homology of $S^1\times \Sigma_g$.
	In an earlier paper, Xie and the first author stated the so-called eigenvalue conjecture about this ring, and in this paper 
	we partially resolve this conjecture. This allows us to establish a surface decomposition theorem for $U(3)$ instanton 
	Floer homology of sutured manifolds, and then obtain the mentioned topological applications.
	Along the way, we prove a structure theorem for $U(3)$ Donaldson invariants, which is the counterpart of Kronheimer and 
	Mrowka's structure theorem for $U(2)$ Donaldson invariants. We also prove a non-vanishing theorem for the 
	$U(3)$ Donaldson invariants of symplectic manifolds.
\end{abstract}

\hypersetup{linkcolor=black}
\setcounter{tocdepth}{1}
\tableofcontents

\newpage


\section{Introduction}

This paper studies invariants in low-dimensional topology derived from $U(N)$ instanton gauge theory, with an emphasis on the case $N=3$. Before describing the particular invariants and the general strategy, we begin with the central topological applications, which regard the existence of certain non-abelian representations of fundamental groups of $3$-manifolds. 

\subsection*{$U(3)$ representations of $3$-manifold groups}

Let $N\geq 2$ be an integer. The following was posed by the first author and Xie \cite{DX}:

\begin{question}\label{question:main}
	If $K$ is a non-trivial knot in an integer homology $3$-sphere $Y$, does there exist a homomorphism $\phi:\pi_1(Y\setminus K)\to SU(N)$ with non-abelian image, such that
	\[
		\phi(\mu) = c\left[ \begin{array}{cccc} 1 & 0 & \cdots & 0\\
											0 & \zeta & \cdots & 0\\
											\vdots  & \vdots & \ddots & \vdots \\
											0 & 0 & \cdots & \zeta^{N-1} \end{array}\right]
	\]
	where $\zeta = e^{2\pi  i /N}$ and $c=e^{\pi i/N}$ or $c=1$ depending on whether $N$ is even or odd?
\end{question}

The notation $\mu$ refers to the class of a fixed meridian in the knot group. Note that if Question \ref{question:main} has an affirmative answer for $N$, then it does so too for all $lN$, where $l\in \Z_{>0}$. Kronheimer and Mrowka proved that Question \ref{question:main} has an affirmative answer in the case $N=2$ \cite{km-sutures}. In this paper we answer it affirmatively in the case $N=3$:

\begin{theorem}\label{thm:intro1}
	If $K$ is a non-trivial knot in an integer homology $3$-sphere $Y$, then there exists a homomorphism  $\phi:\pi_1(Y\setminus K)\to SU(3)$ with non-abelian image, such that 
	\[
		\phi(\mu) = \left[ \begin{array}{ccc} 1 & 0 &  0\\
											0 & \zeta & 0\\
											0 & 0 & \zeta^{2} \end{array}\right], \qquad \zeta=e^{2\pi i/3}.
	\]
\end{theorem}

We also address the existence of $3$-dimensional representations for fundamental groups of closed $3$-manifolds. The following is an $N=3$ analogue of a result of Kronheimer and Mrowka \cite[Thm. 7.21]{km-sutures} (see also \cite[Thm. 1.6]{km-icm}). 

\begin{theorem}\label{thm:intro2}
	Let $Y$ be a closed, oriented $3$-manifold, and $\omega\in H^2(Y;\Z/3)$. Suppose $\omega[S]\equiv 0\pmod{3}$ for every embedded $2$-sphere $S\subset Y$. Then there exists a homomorphism $\pi_1(Y)\to PU(3)$ with the associated characteristic class in $H^2(Y;\Z/3)$ equal to $\omega$.
\end{theorem}

The {\emph{Covering Conjecture}} states that the $N$-fold branched cover $\Sigma_N(K)$ of a non-trivial knot $K$ in a homotopy sphere $Y$ is not a homotopy sphere \cite[Problem 3.38]{kirbylist}. Theorem \ref{thm:intro1} provides a homomorphism $\phi$ which descends to a non-trivial homomorphism of $\pi_1(\Sigma_3(K))$, and thus proves the Covering Conjecture for $N=3$. This also proves the Smith Conjecture, for $N=3$: a non-trivial knot is not the fixed point set of an order $N$ orientation preserving diffeomorphism of $S^3$. Both the Covering Conjecture and Smith Conjecture for general $N$ are theorems, proved by ideas and techniques from diverse areas of mathematics including hyperbolic geometry, minimal surface theory, $SL(2,\C)$ character varieties and classical 3-manifold topology \cite{morganbass}. The proof indicated for $N=3$ (modelled on the proof of Kronheimer and Mrowka for $N=2$) is based on instanton Floer theory.

The Floer homology we utilize is in the setting of $U(3)$ instanton gauge theory. Donaldson-type invariants for closed $4$-manifolds can be defined for any of the groups $U(N)$, see \cite{kronheimer-higher, culler}. (More precisely, the relevant gauge symmetry group is $PU(N)$.) Such invariants were first studied in the physics literature by Mari\~{n}o and Moore \cite{marino-moore}. There, a generalization of Witten's conjecture is provided, which predicts that no new topological information can be derived for $4$-manifolds of simple type when $N\geq 3$. In contrast, Theorems \ref{thm:intro1} and \ref{thm:intro2} are derived from higher rank instanton Floer theory and do not follow from the $U(2)$ theory. To the best of the authors' knowledge, these theorems are the first genuine topological applications of higher rank instanton gauge theory.

\subsection*{$U(3)$ sutured instanton homology}

Kronheimer and Mrowka proved analogues of the above results in the case $N=2$ using $U(2)$ sutured instanton Floer homology \cite{km-sutures}. The strategy to address Question \ref{question:main} in general is to develop sutured instanton Floer theory for $U(N)$ so that the proofs for the $N=2$ case may be adapted. This was initiated in \cite{DX}, where $U(3)$ sutured instanton homology was constructed and some basic properties were established. To a balanced sutured manifold $(M,\alpha)$, the construction outputs a $\Z/2$-graded complex vector space 
\begin{equation}\label{eq:shi-intro}
	SHI_*^3(M,\alpha).
\end{equation}
This is done by first gluing $[-1,1]\times F$ to $M$, where $F$ is a connected surface of genus $g$ with boundary; the gluing is such that $[-1,1]\times \partial F$ is identified with annuli neighborhoods of the sutures $\alpha \subset \partial M$. As $(M,\alpha)$ is balanced, the resulting $3$-manifold has two boundary components which are diffeomorphic, and gluing these up by a diffeomorphism forms a closed $3$-manifold. Then \eqref{eq:shi-intro} is defined by taking a certain subspace of the $U(3)$ instanton homology of the closed $3$-manifold with some choice of admissible bundle.

In \cite{DX}, it is shown that $SHI_*^3(M,\alpha)$ is independent of the gluing maps involved in the construction. Here we establish that the construction is also independent of the choice of $F$ (in particular, the genus $g$). We obtain the following. 

\begin{theorem}\label{thm:intro-SHI-invariance}
	The $U(3)$ sutured instanton homology $SHI_*^3(M,\alpha)$ is independent of all auxiliary choices and is an invariant of the balanced sutured manifold $(M,\alpha)$.
\end{theorem}

We also establish a {\it surface decomposition result}, which describes the behavior of $U(3)$ sutured instanton homology under surface decompositions. This is the counterpart of analogous results in sutured Heegaard Floer homology \cite{juhasz} and $U(2)$ sutured instanton homology \cite{km-sutures}. The surface decomposition result, given in Proposition \ref{sur-decom}, leads to the following non-vanishing result, which (together with a modest generalization given in Corollary \ref{non-vanishing-sutured}) is used to prove Theorems \ref{thm:intro1} and \ref{thm:intro2}.

\begin{theorem}\label{thm:intro-SHI}
	For any balanced taut sutured manifold $(M,\alpha)$, the $U(3)$ sutured instanton homology 
	group $SHI_*^3(M,\alpha)$ is non-trivial. 
\end{theorem}

There are two important special cases of $U(3)$ sutured instanton homology. Both can be defined more generally in the setting of $U(N)$ instanton homology. The first is the $U(N)$ framed instanton homology for closed $3$-manifold $Y$, denoted $I^{\#,N}(Y)$. Versions of these groups were first studied by Kronheimer and Mrowka in \cite{KM:YAFT}. Here we study some further basic properties. We compute in Theorem \ref{thm:eulercharframed} that the Euler characteristic is
\begin{equation}\label{eq:framedeulercharthmintro}
			\chi\left( I^{\#,N}(Y) \right) = |H_1(Y;\Z)|^{N-1}
\end{equation}
when $b_1(Y)=0$, and is otherwise zero. This generalizes the $N=2$ computation from \cite{scadutothesis}. Moreover, in the $N=3$ case, we give a decomposition result for cobordism maps in framed instanton homology analogous to the $N=2$ result in \cite[Theorem 1.16]{bs-lspace}, and which relies on an adaptation of the $U(3)$ Structure Theorem given below.

The other Floer homology group of interest is the $U(N)$ knot instanton homology for a knot in an integer homology $3$-sphere. In the case $N=2$, the graded Euler characteristic of this homology is a multiple of the Alexander polynomial \cite{KM:alexander,lim}. For $N=3$, we provide in Section \ref{sec:alexander} a conjectural relationship between the bi-graded Euler characteristic of the $U(3)$ knot homology and the Alexander polynomial, relying in part on the generalized version of Witten's conjecture from \cite{marino-moore}.

\subsection*{$U(3)$ Donaldson-type invariants for $4$-manifolds}

We also study the structure of $U(3)$ polynomial invariants for closed $4$-manifolds. For any closed connected oriented smooth $4$-manifold $X$ define
\begin{equation}\label{eq:axintro}
	\bA^3(X) := \left( \text{Sym}^\ast(H_0(X)\otimes H_2(X))\otimes \Lambda^\ast H_1(X)  \right)^{\otimes 2}
\end{equation}
where complex coefficients are assumed. For $\alpha\in H_i(X)$ where $i\in \{0,1,2\}$, and $r\in \{2,3\}$, we write $\alpha_{(r)}$ when regarding $\alpha$ as an element of the $(r-1)^{\text{st}}$ tensor power of \eqref{eq:axintro}. The degree of $\alpha_{(r)}$ in this case is defined to be $2r-i$. If $b^+(X)>1$, then for $w\in H^2(X;\Z)$ there is an associated $U(3)$ Donaldson-type invariant
\[
	D^3_{X,w}:\bA^3(X)\to \C.
\]
Let $x\in X$, viewed as a generator of $H_0(X)$. We say that $X$ is $U(3)$ {\emph{simple type}} if
\begin{equation}\label{eq:simpletypeintro}
	D_{X,w}^3(x_{(2)}^3 z) = 27 D_{X,w}^3(z), \qquad D_{X,w}^3(x_{(3)}z)=0, \qquad D_{X,w}^3(\delta z)=0
\end{equation}
for all $z\in \bA^3(X)$ and $\delta\in \Lambda^\ast H_1(X)\otimes \Lambda^\ast H_1(X)$. When $b_1(X)=0$, this terminology is introduced in \cite{DX}, and it is an analogue of Kronheimer and Mrowka's simple type condition in the $U(2)$ case; without the constraint on $b_1(X)$, it is an analogue of Mu\~{n}oz's {\emph{strong simple type}} condition \cite{munoz-basic}. Define for all $z\in\bA^3(X)$:
\[
	\widehat{D}^3_{X,w}(z) := D^3_{X,w}((1+\frac{1}{3}x_{(2)}+\frac{1}{9}x_{(2)}^2)z).
\]
It is also convenient to introduce the following formal power series in $\C[\![H_2(X)\oplus H_2(X)]\!]$:
\begin{equation*}
	\mathbb{D}^3_{X,w}(z) = \widehat{D}^3_{X,w}(e^z).
\end{equation*}
Our main result regarding the structure of these invariants is the following analogue of Kronheimer and Mrowka's structure theorem in the $U(2)$ case \cite{km-structure}. Let $\zeta=e^{2\pi i  /3}$.

\begin{theorem}\label{thm-intro:structure}
Suppose $b^+(X) > 1$, and $X$ is $U(3)$ simple type. Then there is a finite set $\{K_i\} \subset  H^2(X; \Z)$ and $c_{i,j}\in\Q[\sqrt{3}]$ such that for any $w\in H^2(X;\Z)$, and $\Gamma,\Lambda\in H_2(X)$:
\begin{equation}
\mathbb{D}^3_{X,w}(\Gamma_{(2)}+\Lambda_{(3)}) = e^{\frac{Q(\Gamma)}{2}-Q(\Lambda)}\sum_{i, j} c_{i, j} \zeta^{w\cdot \left(\frac{K_i-K_j}{2}\right)}e^{\frac{\sqrt{3}}{2}(K_i+K_j)\cdot \Gamma+\frac{\sqrt{-3}}{2}(K_i-K_j)\cdot \Lambda} \label{eq:structurethmformulaintro}
\end{equation}
Each class $K_i$ is an integral lift of $w_2(X)$, and satisfies the following: if $\Sigma\subset X$ is a smoothly embedded surface of genus $g$ with $\Sigma\cdot \Sigma\geq 0$ and $[\Sigma]$ non-torsion, then 
\begin{equation}\label{eq:structureadjunctionintro}
  2g-2 \geq
|\langle K_i, \Sigma\rangle |+[\Sigma]^2.
\end{equation}
\end{theorem}

\noindent This result partially resolves Conjecture 7.2 from \cite{DX}. We note that the expression \eqref{eq:structurethmformulaintro} differs slightly from what appears in that reference, due to a minor difference in convention; see Remark \ref{rmk:conventions}. As explained in \cite{DX}, it is predicted by Mari\~{n}o and Moore \cite{marino-moore} that the classes $K_i$ appearing in Theorem \ref{thm-intro:structure} are equal to the Kronheimer and Mrowka basic classes in $U(2)$ Donaldson theory, as well as the Seiberg--Witten basic classes; furthermore, the constants $c_{i,j}$ are expressible in terms of the data from these other theories. 

\begin{remark}
	In Theorem \ref{thm:structure}, if $[\Sigma]$ is torsion and $g\geq 1$, then \eqref{eq:structureadjunctionintro} trivially holds. Note that if $\Sigma$ is as in Theorem \ref{thm-intro:structure} and has genus zero, then \eqref{eq:structureadjunctionintro} never holds, and hence there are no such classes $K_i$, in which case the invariants $D_{X,w}$ all vanish. 
\end{remark}

We also prove a non-vanishing result for symplectic $4$-manifolds.

 \begin{theorem}\label{thm-intronon-vanishing-symplectic}
	Let $X$ be a closed symplectic 4-manifold with $b^+(X) > 1$. Then the invariant $D_{X,w}^3$ is non-trivial for all $w\in H^2(X;\Z)$.
\end{theorem}

\noindent Our strategy to prove this non-vanishing result is similar to Ozsv\'ath and Szab\'o's proof in \cite{os-symplectic} for the corresponding result in the context of Heegaard Floer homology and closely follows a strategy suggested by Kronheimer and Mrowka in the $U(2)$ case. (The non-triviality of $U(2)$ Donaldson invariants for symplectic $4$-manifolds was also proved by Sivek in a different way \cite{sivek}.) Theorem \ref{thm-intronon-vanishing-symplectic} can be used to prove that the $U(3)$ instanton homology of an irreducible $3$-manifold with $3$-admissible bundle is non-zero, and leads to an alternative proof of Theorem \ref{thm:intro2}.

\subsection*{Eigenvalues and the $U(3)$ instanton homology of $S^1\times \Sigma_g$}

The main technical ingredient that paves the way for most of the above results is Theorem \ref{thm:mainev} below, which concerns the $U(3)$ instanton homology of a circle times a surface $\Sigma_g$ of genus $g$ with an admissible bundle. We restrict our attention to the {\emph{simple type ideal}}, a subspace of this Floer homology, and compute the eigenvalues of certain operators acting on it. The simple type ideal is generated by relative invariants coming from $4$-manifolds of simple type. Our eigenvalue result is a partial analogue to one used by Kronheimer and Mrowka in the $U(2)$ case, due to Mu\~{n}oz \cite{munoz}. 

The $U(N)$ instanton Floer homology of $S^1\times \Sigma_g$ with an admissible bundle is isomorphic to $H^\ast(\cN_g)$, the cohomology of the moduli space of rank $N$ stable holomorphic bundles over $\Sigma_g$ with some fixed determinant. In fact, this instanton Floer group admits a multiplication which is a deformation of the cup product on $H^\ast(\cN_g)$, and is expected to be isomorphic to its quantum multiplication. Mu\~{n}oz's computation of eigenvalues in the $N=2$ case relies on the fact that $H^\ast(\cN_g)$ has a simple ring presentation which is recursive in the genus \cite{baranovskii, king-newstead, siebert-tian, zagier}. Such a concise description is not currently available in the $N=3$ case, but a complete set of relations for the ring is known, due to Earl \cite{earl}. Our restriction to the simple type ideal (which suffices for the purposes of the above results) simplifies the algebra considerably, and allows us to use Earl's description of the ring $H^\ast(\cN_g)$ to prove, together with results from \cite{DX}, the desired eigenvalue result.

The authors expect that the method of proof for Theorem \ref{thm:mainev} may also be employed for $N >3$. There are two essential ingredients that are required. One is a generalization to $N\geq 4$ of \cite[Prop. 5.7]{DX}, which gives the existence of certain eigenvalues in the $U(N)$ instanton Floer homology of a circle times a surface. The other is a computation, for $N\geq 4$, of the vector space dimension of the ring $H^\ast(\cN_g)$ modulo the ``undeformed simple type relations,'' analogous to what is done below for $N=3$. Relevant to this second ingredient is the work of Earl and Kirwan \cite{earl-kirwan}. Given an appropriate generalization of Theorem \ref{thm:mainev} for $N>3$, the authors expect that analogues for all of the results stated in this introduction, for general $N$, can also be proved, following similar methods. The authors hope to return to these matters in future work.\\

\textbf{Outline}\;\; In Section \ref{sec:strategy}, we state and outline the proof of the main technical result, Theorem \ref{thm:mainev}. In Sections \ref{sec:undeformed} and \ref{sec:simpletypequotient}, the cohomology ring $H^\ast(\cN_g)$ is studied, and the proof of Theorem \ref{thm:mainev} is completed. In Section \ref{sec:sutured}, we study $U(3)$ sutured instanton homology and prove Theorems \ref{thm:intro-SHI-invariance}, \ref{thm:intro-SHI}, followed by the proofs of Theorems \ref{thm:intro1} and \ref{thm:intro2}. In Section \ref{sec:structure}, we prove Theorem \ref{thm-intro:structure}, and in Section \ref{non-vanishing-symp} we prove Theorem \ref{thm-intronon-vanishing-symplectic}. In Section \ref{sec:framed}, $U(N)$ framed instanton homology is studied. Finally, in Section \ref{sec:alexander} we discuss $U(3)$ instanton knot homology and the Alexander polynomial. \\

\textbf{Acknowledgments}\;\; We thank Hisaaki Endo, Peter Kronheimer, Jake Rasmussen, Arash Rastegar, Danny Ruberman and Steven Sivek for helpful discussions. AD was supported by NSF Grant DMS-2208181 and NSF FRG Grant DMS-1952762.  NI was supported by JSPS KAKENHI Grant Number 22J00407. CS was supported by NSF FRG Grant DMS-1952762.


\section{Background and general strategy}\label{sec:strategy}

As mentioned in the introduction, the strategy to prove Theorems \ref{thm:intro1} and \ref{thm:intro2} is to develop sutured instanton Floer homology for the gauge group $U(N)$, and adapt arguments from the $N=2$ case due to Kronheimer and Mrowka \cite{km-sutures}. This strategy was initiated in the $N=3$ case by the first author and Xie \cite{DX}. The raw material for the construction of sutured instanton homology for general $N$ is the $U(N)$ instanton Floer homology
\begin{equation}\label{eq:instantonfloerhomologyintro}
	I_\ast^N(Y,\gamma)
\end{equation}
which is defined for a closed, oriented, connected $3$-manifold $Y$ and an oriented $1$-cycle $\gamma$ satisfying the {\emph{$N$-admissibility}} condition: there exists some oriented surface $\Sigma\subset Y$ such that $\gamma\cdot \Sigma$ is coprime to $N$. The group $I_\ast^N(Y,\gamma)$ is constructed by applying Morse homological methods to a Chern--Simons functional on $\mathscr{B}$, the configuration space of connections on the $PU(N)$-bundle on $Y$ determined by $\gamma$. These Floer homology groups were constructed by Kronheimer and Mrowka \cite{KM:YAFT}, generalizing the work of Floer in the $N=2$ case \cite{floer-dehn}. In this paper, we work with Floer homology over the coefficient field $\C$, in which case \eqref{eq:instantonfloerhomologyintro} is a $\Z/4N$-graded complex vector space. 

Given a homology class $a\in H_i(Y;\C)$ there are associated linear operators
\begin{equation}
	\mu_r(a):I_\ast^N(Y,\gamma)\to I_\ast^N(Y,\gamma), \qquad 2\leq r \leq N. \label{eq:muoperators}
\end{equation}
The degree of $\mu_r(a)$ is $2r-i\pmod{4N}$. There is a universal $PU(N)$-bundle $\mathbf{P}$ over $\mathscr{B}\times Y$, and $\mu_r(a)$ is roughly the cap product on moduli spaces with $c_r(\mathbf{P})/a$. 

\begin{remark}\label{rmk:conventions}
	Our convention for $\mu_r(a)$ differs from that of \cite{DX} by the sign $(-1)^r$. Furthermore, the grading we use on instanton homology is the negative of the convention in that paper (and is in fact a cohomological grading convention).
\end{remark}

The construction of sutured instanton Floer homology relies on taking certain simultaneous generalized eigenspaces of the operators in \eqref{eq:muoperators} acting on the Floer groups \eqref{eq:instantonfloerhomologyintro}. The crucial case to understand is when $Y=S^1\times \Sigma_g$ where $\Sigma_g$ is a surface of genus $g$, and $\gamma=\gamma_{d}=S^1\times \{x_1,\ldots,x_d\}$, with $d$ coprime to $N$. We write
\[
	V^N_{g,d} := I_\ast^N(S^1\times \Sigma_g,\gamma_d).
\]
The relevant operators acting on $V_{g,d}^N$ are denoted as follows:
\begin{equation}\label{eq:s1sigmamaps}
	\alpha_r :=  \mu_r(\Sigma_g), \qquad \beta_r := \mu_r(x), \qquad \psi_r^i := \mu_r(\eta_i) \qquad (2\leq r \leq N)
\end{equation}
where $x\in S^1\times \Sigma$ and the $\eta_i$ $(1\leq i \leq 2g)$ range over a symplectic basis of closed oriented curves on $\Sigma_g$. These operators are graded-commutative. In particular, since the $\psi_r^i$ are of odd degree, they square to zero, and each one has zero as its only eigenvalue. Consider the simultaneous eigenvalues with respect to the classes $\alpha_r$, $\beta_r$:
\[
	\Xi_{g,d}^N := \left\{ \lambda=(\lambda_1,\dots,\lambda_{2N-2})\in \C^{2N-2} \; \mid \; \exists v\in V_{g,d}^N :\; \alpha_rv = \lambda_{r-1}v, \;\; \beta_rv = \lambda_{r+N-2} v \right\}
\]
where $r$ ranges over $2,\ldots, N$. For $\lambda\in \Xi_{g,d}^N$ we denote by 
\[
	V_{g,d}^N(\lambda) = \bigcap_{r=2}^N \bigcup_{k=1}^\infty \text{ker}\left((\alpha_r-\lambda_{r-1} )^k\right)\cap \text{ker}\left((\beta_r-\lambda_{r+N-2})^k\right) \subset V_{g,d}^N
\]
the associated generalized eigenspace. Then we have
\[	
	 V_{g,d}^N= \bigoplus_{\lambda\in \Xi_{g,d}^N}  V_{g,d}^N(\lambda).
\]
\begin{lemma}\label{lemma:evaction}
	Let $\zeta$ be a $2N^{\rm{th}}$ root of unity. If $\lambda=(\lambda_1,\ldots,\lambda_{2N-2})\in \Xi_{g,d}^N$, then also
	\[
		\lambda':=(\zeta \lambda_1,\zeta^2 \lambda_2,\ldots,\zeta^{N-1}\lambda_{N-1}, \zeta^2\lambda_{N}, \zeta^3\lambda_{N+1},\ldots,\zeta^{N-1}\lambda_{2N-3}, \zeta^{N}\lambda_{2N-2}) \in \Xi_{g,d}^N.
	\]
	Furthermore, $\dim V_{g,d}^N(\lambda) = \dim V_{g,d}^N(\lambda')$.
\end{lemma}

\begin{proof}
Fix $\zeta$ as in the statement, and define $f:V_{g,d}^N\to V_{g,d}^N$ as follows. Let $v\in V_{g,d}^N$ and write $v_{i}$ for the component of $v$ in grading $i\pmod{4N}$. Then
\[
	f(v) := \sum_{i=0}^{2N-1} \zeta^{-i}v_{2i} + \sum_{i=0}^{2N-1} \zeta^{-i}v_{2i+1}.
\]
Let $v\in V_{g,d}^N(\lambda)$. In particular, for each $2\leq r\leq N$ we have $(\alpha_r-\lambda_{r-1})^{N}v=0$ for some positive integer $N$. This is equivalent to the collection of identities
\[
	\sum_{i=0}^N {N\choose i} \alpha_r^{N-i} (-\lambda_{r-1})^i v_{l-(N-i)(2r-2)} =0
\]
where $2\leq r\leq N$ and $0\leq l\leq 4N-1$. It is straightforward to check that this identity is preserved upon replacing $\lambda_{r-1}$ with $\zeta^{r-1}\lambda_{r-1}$ and replacing $v_{l-(N-i)(2r-2)}$ with $f(v)_{l-(N-i)(2r-2)}$. The conditions involving the $\beta_r$ operators is similar. Thus $f$ induces a vector space isomorphism from $V_{g,d}^N(\lambda)$ to $V_{g,d}^N(\lambda')$.
\end{proof}

There is also an operator of degree $-4d \pmod{4N}$ denoted
\begin{equation}\label{epsilon-product}
	\varepsilon :V^N_{g,d}\to V^N_{g,d}
\end{equation}
 defined as the map associated to the cylinder cobordism $[0,1]\times S^1\times \Sigma_g$ equipped with $U(N)$-bundle determined by the oriented $2$-cycle $[0,1]\times \gamma_{d}\cup \{(1/2,x)\}\times \Sigma_g$ where $x\in S^1$. The operator $\varepsilon$ commutes with all the operators \eqref{eq:s1sigmamaps}.

The Floer homology $ V_{g,d}^N$ is in fact a ring. The multiplication is induced by the cobordism which is the product of a pair of pants cobordism $S^1\sqcup S^1\to S^1$ with $\Sigma_g$. An identity element $\mathbf{1}$ is given by the relative invariant $D^2\times \Sigma_g$ with bundle determined by $D^2\times \{x_1,\ldots,x_d\}$. Sending each operator \eqref{eq:s1sigmamaps} and $\varepsilon$ to its evaluation on $\mathbf{1}$ induces an isomorphism of rings
\begin{equation}\label{eq:firstringpres}
	 V_{g,d}^N  = \bA^N_g[\varepsilon]/J_{g,d}^N
\end{equation}
where the $\Z$-graded $\C$-algebra $\bA_g^N$ is defined as follows:
\[
	 \bA_g^N := \bigotimes_{r=2}^N\C[\alpha_r,\beta_r]\otimes \Lambda^\ast(\psi_r^{i})_{1\leq i \leq 2g}
\]
The degrees of $\alpha_r,\beta_r,\psi_r^i$ are respectively $2r-2, 2r, 2r-1$. In the identification \eqref{eq:firstringpres}, the $\Z$-grading on $\bA_g^N$ reduced to the $\Z/4$-grading on $V_{g,d}^N$; on the right side of \eqref{eq:firstringpres}, the element $\varepsilon$ should be regarded as having degree $0$. The ideal $J_{g,d}^N\subset \bA_g^N[\varepsilon]$ contains $\varepsilon^N-1$ and is homogeneous with respect to the $\Z/4$-grading. There is a non-degenerate bilinear pairing
\begin{equation}\label{eq:pairingintro}
	\langle \cdot,\cdot \rangle : V_{g,d}^N \otimes V_{g,d}^N \to \C
\end{equation}
which is induced by $[0,1]\times S^1\times \Sigma_g$ viewed as a cobordism from two copies of $S^1\times \Sigma_g$ (identifying one copy by an orientation-reversing diffeomorphism) to the empty set.

Key to the development of Kronheimer and Mrowka's sutured instanton homology in the case $N=2$ are results on the eigenvalues of the operators \eqref{eq:s1sigmamaps}. There are two essential ingredients that are used, both following from the work of Mu\~{n}oz \cite{munoz} (see in particular \cite[Props. 7.1, 7.4]{km-sutures}). The first is the computation of the spectrum:
\begin{equation}\label{eq:n2spectrum}
	\Xi_{g,1}^2 = \left\{	(2a i^{r}, (-1)^r2)\; \mid \; a\in \Z, \; |a|\leq g-1, \; r\in \{0,1\} \right\}
\end{equation}
where $i=\sqrt{-1}$. The second ingredient regards ``extremal'' generalized eigenspaces:
\begin{equation}\label{eq:munoz2}
	\dim V_{g,1}^2(\pm i^{r}(2g-2),(-1)^{r}2) = 1.
\end{equation}
An important property is that the pairing  \eqref{eq:pairingintro} restricted to the $1$-dimensional space appearing in \eqref{eq:munoz2} is non-degenerate. In fact, the $1$-dimensionality is equivalent to non-degeneracy, see for example \cite[Lemma 5.11]{DX}.

In \cite{DX}, analogous properties for the case of $N=3$ are studied. To state the relevant results, first define, for any integer $d$ coprime to $3$: 
\begin{equation*}
	\mathcal{E}_{g,d}^3 :=  \left\{ (\sqrt{3}\zeta^{k}a, \sqrt{-3}\zeta^{2k} b, 3\zeta^{2k}, 0) \; \mid \; (a,b)\in \mathcal{C}_g, \; k\in \{0,1,2\} \right\} \subset \C^4.
\end{equation*}
Here $\zeta=e^{2\pi i/3}$, and $\mathcal{C}_g$ is the subset of the lattice $\Z^2$ given by
\[
	\mathcal{C}_g = \{ (a,b) \in \Z^2 \, \mid \, |a|+|b| \leq 2g-2, \; a\equiv b \, (\text{mod } 2)\}.
\]
Then, the following is a partial analogue of \eqref{eq:n2spectrum}.

\begin{prop}\label{prop:evinclusion}
	$\mathcal{E}_{g,d}^3 \subset \Xi_{g,d}^3$.
\end{prop}

\begin{proof}
	It is proved in \cite[Prop. 5.7]{DX} that the set
\begin{equation}\label{eq:epsilon1evs}
	 \left\{ (\sqrt{3}\zeta^{db}a, -\sqrt{-3}\zeta^{2db} b, 3\zeta^{2db}, 0) \; \mid \; (a,b)\in \mathcal{C}_g \right\} 
\end{equation}
is contained in $\Xi_{d,g}^3$. (Note that this set of eigenvalues differs slightly from that in \cite[Prop. 5.7]{DX} because of Remark \ref{rmk:conventions}.) In fact, these eigenvalues simultaneously occur with the eigenvalue $+1$ of $\varepsilon$. The remaining eigenvalues are obtained using Lemma \ref{lemma:evaction}.
\end{proof}

The inclusion of Proposition \ref{prop:evinclusion} is conjectured to be equality, see \cite[Conj. 7.3]{DX}. An analogue of \eqref{eq:munoz2} is essentially proved in \cite{DX} (see Proposition \ref{prop:dim1}):
\begin{equation}\label{eq:dx2}
	\dim V_{g,d}^3(\pm \sqrt{3}\zeta^k(2g-2), 0, 3\zeta^{2k}, 0) = 1
\end{equation}
where $k\in\{0,1,2\}$. Just as in the $N=2$ case, the pairing \eqref{eq:pairingintro} restricted to this $1$-dimensional space is non-degenerate;  Proposition \ref{prop:evinclusion} and property \eqref{eq:dx2}, with its non-degeneracy, are sufficient to define sutured instanton homology and prove an excision result in the case $N=3$, parallel to the case $N=2$, and this is explained in \cite[\S 5.2]{DX}. 

In the case $N=2$, Kronheimer and Mrowka prove a sutured decomposition result \cite[Prop. 7.11]{km-sutures} using \eqref{eq:n2spectrum}. This result implies that sutured instanton homology for taut sutured manifolds is nonzero, and leads to existence results for $U(2)$ representations. For $N=3$, if equality in Proposition \ref{prop:evinclusion} holds, then similar arguments carry through. However, the inclusion of Proposition \ref{prop:evinclusion} by itself is not sufficient.

On the other hand, inspection of the arguments in \cite{km-sutures} shows that in the $N=2$ case, equality of \eqref{eq:n2spectrum} is not necessary. The following weaker version of \eqref{eq:n2spectrum} suffices:
\[
	\Xi_{g,1}^2 \cap \left( \C\times \{  \pm 2\} \right)= \left\{	(2a i^{r}, (-1)^r2)\; \mid \; a\in \Z, \; |a|\leq g-1, \; r\in \{0,1\} \right\}.
\]
The same is true in the case $N=3$: equality in Proposition \ref{prop:evinclusion} is not necessary, and the following, our main technical result, is a weaker version which suffices:

\begin{theorem}\label{thm:mainev}
Let $d\in \Z$ be coprime to $3$, $g\in \Z_{\geq 0}$, and $\zeta=e^{2\pi i/3}$. If $\lambda\in \Xi_{g,d}^3$ and $\lambda=(\lambda_1,\lambda_2,3\zeta^j,0)$ for some $j\in \Z$, then $\lambda\in \mathcal{E}_{g,d}^3$. Equivalently (by Proposition \ref{prop:evinclusion}):
	\[
			\Xi_{g,d}^3 \cap \left(\C^2 \times C_3 \times \{0\}\right) = \mathcal{E}_{g,d}^3 
	\]
	where $C_3=\{3,3\zeta,3\zeta^2\}$ is the set of $3^\text{rd}$ roots of $27$.
\end{theorem}

In the remainder of this section we explain the strategy to prove Theorem \ref{thm:mainev}. Let $\cN_{g}=\cN_{g,d}^N$ be the moduli space of projectively flat $U(N)$-connections $A$ on $\Sigma_g$ with $\det(A)=A_0$, where $A_0$ is a fixed connection on a complex line bundle $L\to \Sigma_g$ of degree $d$. There is a natural isomorphism of rings (see Section \ref{sec:undeformed}):
\begin{equation}\label{eq:ordinarycohomologypres}
	H^\ast(\cN_g;\C) = \bA_g^N/I_g
\end{equation}
where $I_g$ is a homogeneous ideal in $\bA_g^N$. Consider the extended ideal
\[
	I_g':=(\varepsilon^N-1)I_g + \sum_{i=0}^{N-1} \varepsilon^i I_g \subset \bA_g^N[\varepsilon].
\]
Then, the relation ideal $J_{g,d}^N$ for $V_{g,d}^N$ from \eqref{eq:firstringpres} is a deformation of the ideal $I_g'$. Concretely, 
\[
	I_g' = \left( L(f) \; \mid \; f\in J_{g,d}^N \right)\subset \bA_g^N[\varepsilon]
\]
where $L(f)$ is the top degree homogeneous part of $f$. Here the degree of $\varepsilon$ is set equal to $0$. Furthermore, there is a complex vector space isomorphism
\begin{equation}\label{iso-vect-coh-ins}
	V_{g,d}^N \cong H^\ast(\cN_g;\C)[\varepsilon]/(\varepsilon^N-1).
\end{equation}
That is to say, the complex dimensions of the quotients of $\bA_g^N[\varepsilon]$ by $J_{g,d}^N$ and $I_g'$ are equal. These observations were first given by Mu\~{n}oz in the case $N=2$ \cite{munoz}; the case for general $N$ is similar, and discussed in \cite{DX}.

Define the \emph{simple type ideal} of $V_{g,d}^3$ as follows:
\begin{equation}\label{eq:simpletypeideal}
	S_{g,d}^3 =  \text{ker}(\beta_2^3-27)\cap \text{ker}(\beta_3)\cap \bigcap_{\substack{1\leq i\leq 2g\\ r=2,3}}\text{ker}(\psi_r^i) \subset V_{g,d}^3.
\end{equation}
The inclusion of Proposition \ref{prop:evinclusion} implies the following inequality:
\begin{equation}\label{eq:firstineqintro}
	\dim_\C S_{g,d}^3 \geq | \mathcal{E}_{g,d}^3 | = 3 (2g-1)^2.
\end{equation}
Furthermore, if equality in \eqref{eq:firstineqintro} holds, then in it is straightforward to see that in fact there can be no other eigenvalues in $\Xi_{g,d}^3$ of the form $(\lambda_1,\lambda_2,3\zeta^j,0)$, and Theorem \ref{thm:mainev} follows. Thus our goal is to prove the inequality
\begin{equation}\label{eq:desiredineqsimpletypeideal}
	\dim_\C S_{g,d}^3\leq  3(2g-1)^2.
\end{equation}
Define $\widetilde J_{g,d}^3$ to be the ideal of $\bA_g^3[\varepsilon]$ generated by $J_{g,d}^3$ and $\beta_2^3-27$, $\beta_3$, $\psi_r^i$, $\varepsilon^3-1$. The pairing \eqref{eq:pairingintro} satisfies $\langle ax,y\rangle = \langle x, ay\rangle$ for all $a\in V_{g,d}^N$. Thus there is an induced pairing
\[
	S_{g,d}^3 \otimes \bA_g^3[\varepsilon]/\widetilde{J}_{g,d}^3 \to \C 
\]
Nondegeneracy of \eqref{eq:pairingintro} implies the inequality
\begin{equation}\label{eq:ineqsimpletypeideal2}
	\dim_\C S_{g,d}^3 \leq \dim_\C  \bA_g^3[\varepsilon]/\widetilde{J}_{g,d}^3.
\end{equation}
On the other hand, consider the ideal
\[
	\widetilde{I}_g := I_g + (\beta_2^3, \beta_3, \psi_2^{i},\psi_3^i)_{1\leq i \leq 2g} \subset \bA_g^3
\]
and its extension $\widetilde{I}_g' :=(\varepsilon^3-1)\widetilde{I}_g  + \widetilde{I}_g + \varepsilon \widetilde{I}_g + \varepsilon^2 \widetilde{I}_g$ inside $\bA_g^3[\varepsilon]$. Since $\widetilde{J}_{g,d}^3$ is a deformation of the ideal $\widetilde{I}_g'$, it follows that there is an inequality
\begin{equation}\label{eq:ineqsimpletypeideal3}
	\dim_\C\bA_g^3[\varepsilon]/\widetilde{J}_{g,d}^3 \leq \dim_\C\bA_g^3[\varepsilon]/\widetilde{I}_{g}' = 3 \dim_\C\bA_g^3/\widetilde{I}_{g}.
\end{equation}
Therefore, the following result, together with inequalities \eqref{eq:ineqsimpletypeideal2} and \eqref{eq:ineqsimpletypeideal3}, proves the desired inequality \eqref{eq:desiredineqsimpletypeideal}, and hence proves Theorem \ref{thm:mainev}.

\begin{theorem}\label{thm:undeformeddimension}
	For $g\geq 1$, $\dim_\C\bA_g^3/\widetilde{I}_{g} \leq (2g-1)^2$.
\end{theorem}

\noindent This theorem is proved in the next two sections, where the ring $H^\ast(\cN_g)$ is studied. For reasons explained above (see also the end of this section), we call $\bA_g^3/\widetilde{I}_{g}$ the {\emph{undeformed simple type quotient}}. 

We now show how \eqref{eq:dx2} follows from \cite{DX} and Theorem \ref{thm:undeformeddimension}.

\begin{prop}\label{prop:dim1}
	$\dim V_{g,d}^3(\pm \sqrt{3}\zeta^k(2g-2), 0, 3\zeta^{2k}, 0) = 1$ for each $k\in\{0,1,2\}$. In particular, the generalized eigenspace for 
	$(\pm \sqrt{3}\zeta^k(2g-2), 0, 3\zeta^{2k}, 0)$ agrees with the corresponding eigenspace.
\end{prop}

\begin{proof}
From our above discussion, Theorem \ref{thm:undeformeddimension} implies
\begin{equation}\label{eq:equalityofsomequotient}
	\dim_\C\bA_g^3[\varepsilon]/\widetilde{J}_{g,d}^3 = 3(2g-1)^2.
\end{equation}
For $\lambda_0\in \C$ and $\lambda\in \Xi_{g,d}^3$ write $V(\lambda_0,\lambda)=V_{g,d}^3(\lambda)\cap\text{ker}(\varepsilon-\lambda_0)$. Then
\[
	V_{g,d}^3 = \bigoplus_{(\lambda_0,\lambda)\in \C\times \Xi_{g,d}^3} V(\lambda_0,\lambda).
\]
Write $\Pi$ for the projection from $V_{g,d}^3$ to $\bA_g^3[\varepsilon]/\widetilde{J}_{g,d}^3$. Since $|\cE_{g,d}^3|=3(2g-1)^2$, for each $\lambda\in \cE_{g,d}^3$ there is a {\emph{unique}} $\lambda_0\in \C$ such that $\Pi(V(\lambda_0,\lambda))$ is nonzero. For if this were not the case, the equality \eqref{eq:equalityofsomequotient} would be violated. Let $\lambda=(\pm \sqrt{3}(2g-2), 0, 3, 0)$. In \cite[\S 5.1]{DX} it is shown that $V(1,\lambda)$ is $1$-dimensional. Consequently,
\[
	\dim\left( V_{g,d}^3(\lambda)\cap \text{ker}(\varepsilon-1)\right) = 1.
\]
On the other hand, by the above remarks, it must be that $V_{g,d}^3(\lambda)\subset \text{ker}(\varepsilon-1)$. This proves the desired result in the case $k=0$. The cases where $k\in\{1,2\}$ then follow from the case $k=0$ and Lemma \ref{lemma:evaction}.
\end{proof}

We conclude this section with some commentary on our terminology used for the subspace $S_{g,d}^3\subset V_{g,d}^3$. First, for any oriented smooth $4$-manifold $X$ recall the definition
\[
	\bA^3(X) = \left( \text{Sym}^\ast(H_0(X)\otimes H_2(X))\otimes \Lambda^\ast H_1(X)  \right)^{\otimes 2}
\]
where complex coefficients are assumed. If $X$ is closed and $b^+(X)>1$, and $w\in H^2(X;\Z)$, then there is an associated $U(3)$ Donaldson-type invariant
\[
	D^3_{X,w}:\bA^3(X)\to \C,
\]
and we now review the outline of its construction. Let $z=(z_{i_1}\cdots z_{i_k})\otimes (z'_{j_1}\cdots z'_{j_l})\in \bA^3(X)$ where each $z_{i_s}$ and $z'_{j_s}$ in $H_i(X)$ for some $i\in \{0,1,2\}$. Consider the moduli space of $PU(3)$ instantons with energy $\kappa$ on $X$, with bundle determined by $w$, and cut down by divisors representing $\mu_2(z_{i_1}),\ldots \mu_2(z_{i_k}),\mu_3(z'_{j_1}),\ldots , \mu_3(z'_{j_l})$. The energy $\kappa$ is chosen so that the cut-down space has expected dimension $0$ (if this is not possible, the invariant is zero). For a generic metric and perturbation, the cut-down moduli space is a compact $0$-manifold, and $D^3_{X,w}(z)$ is the associated signed count. (In general, the blow-up trick of Morgan--Mrowka is also employed.) The following condition on pairs $(X,w)$ refines the definition of $U(3)$ simple type given in the introduction.

\begin{definition}\label{defn:simpletype}
Let $X$ be a closed oriented 4-manifold with $b^+(X)>1$ and $w \in H^2(X; \Z)$.
The pair $(X,w)$ is called $U(3)$ {\emph{simple type}} if
\begin{equation}\label{eq:sst}
D^{3}_{X, w}((x_{(2)}^3-27) z)=0,  \quad D^{3}_{X, w}(x_{(3)}z)=0 , \quad D^{3}_{X, w}(\delta z)=0
\end{equation}
for any $z\in \bA^3(X)$ and any $\delta \in \Lambda^* H_1(X)  \otimes \Lambda^* H_1(X) \subset  \bA^3(X)$.
We say $X$ is {\emph{$U(3)$ simple type}} if $(X,w)$ is $U(3)$ simple type for all $w \in H^2(X; \Z)$. 
\end{definition}

Let $(X,w)$ be a pair of a closed, smooth, oriented $4$-manifold and a 2-cycle $w$, with $b_1(X)=0$ and $b^+(X)>1$, which is also $U(3)$ simple type. Suppose further that $\Sigma\subset X$ is an embedded surface of genus $g$ in $X$ such that $\Sigma\cdot\Sigma=0$ and $d:=\Sigma \cdot w$ is coprime to $3$. Removing a regular neighborhood of $\Sigma$ from $(X,w)$ produces a pair $(X^\circ,w^\circ)$ with boundary $(S^1\times \Sigma_g,\gamma_{d})$. In particular, for any $z\in \bA^3(X^\circ)$ there are relative invariants
\begin{equation}\label{eq:relinvtssimpletype}
  D^3_{X^\circ,w^\circ}(z)\in V_{g,d}^3.
\end{equation}
The proof of \eqref{prop:evinclusion} from \cite{DX} produces eigenvectors with eigenvalues in $\mathcal{E}_{g,d}^3$ using such relative invariants (see also proof of Theorem \ref{thm:fukayafloersimpletypeideal}). A gluing formula expresses invariants of $(X,w)$ in terms of relative invariants, using the pairing \eqref{eq:pairingintro}:
\begin{equation}\label{eq:glueformulafors1timessigma}
	 D^3_{X,w}(zz') =  \langle D^3_{X^\circ,w^\circ}(z), z' \mathbf{1} \rangle,
\end{equation}
where $z'\in \bA^3_{g,d}$, which also induces an element of $\bA^3(X)$. The gluing formula \eqref{eq:glueformulafors1timessigma}, the simple type condition \eqref{eq:sst}, and the non-degeneracy of the pairing \eqref {eq:pairingintro} imply that
\begin{equation}\label{rel-inv-Sgd3}
	D^3_{X^\circ,w^\circ}(z)\in S_{g,d}^3.
\end{equation}
A consequence of Theorem \ref{thm:undeformeddimension} is
\[
	\dim_\C S_{g,d}^3 = 3 (2g-1)^2,
\]
which implies that the simple type ideal $S_{g,d}^3$ is in fact spanned by relative invariants coming from simple type $4$-manifolds.


\section{Mumford relations and their duals}\label{sec:undeformed}

As in the previous section, denote by $\cN_{g}=\cN_{g,d}^N$ the moduli space of projectively flat $U(N)$-connections $A$ on a Riemann surface $\Sigma_g$ of genus $g$ with $\det(A)=A_0$, where $A_0$ is a fixed connection on a line bundle $L\to \Sigma_g$ of degree $d\in \Z$. Assume as before that $d$ is coprime to $N$. Then $\cN_{g}$ is a smooth manifold of dimension $(N^2-1)(2g-2)$. By the Narasimhan--Seshadri correspondence, $\cN_{g}$ may be identified with the moduli space of rank $N$ stable holomorphic bundles over $\Sigma_g$ with fixed determinant of degree $d$.

 There is a universal $U(N)$-bundle $U \to \cN_{g}\times \Sigma_g$. This bundle is not unique, as tensoring it by any line bundle pulled back from $\cN_g$ gives another such choice. However,
 \[
 	\bP := U \otimes \det(U)^{-1/N}
 \]  
defines an element in the rational $K$-theory of $\cN_{g}\times \Sigma_g$, which is independent of the choice of the universal bundle $U$. We define cohomology classes
\begin{equation}\label{eq:moduligens}
	\alpha_r\in H^{2r-2}(\cN_g), \qquad \psi_{r}^{i}\in H^{2r-1}(\cN_g), \qquad \beta_r\in H^{2r}(\cN_g)
\end{equation}
using the K\"{u}nneth decomposition of the Chern class $c_r(\bP) \in H^\ast(\cN_g)\otimes H^\ast(\Sigma_g)$:
\begin{equation}\label{eq:moduligensdef}
	c_r(\bP) = \alpha_r \otimes \sigma + \sum_{i=1}^{2g} \psi_{r}^i \otimes \xi_{i} +  \beta_r \otimes 1 \qquad (2\leq r\leq N).
\end{equation}
Note $c_1(\bP)=0$. All cohomology groups are defined over $\C$, unless otherwise mentioned. (However, everything in this section can be done over $\Q$.) Recall that a symplectic integral basis of $\{\eta_i\}_{i=1}^{2g}$ for $H_1(\Sigma_g)$ was fixed earlier. In \eqref{eq:moduligensdef}, $\{\xi_i\}_{i=1}^{2g}$ is the integral basis of $H^1(\Sigma_g)$ satisfying $\xi_i(\eta_j)=\delta_{ij}$, and $\sigma$ is the integral generator of $H^2(\Sigma_g)$ given by the orientation. In particular, for $1\leq i \leq g$, we have  $\xi_i\xi_{g+i}=\sigma$ and $\xi_i \xi_j=0$ when $j\neq g+i$. The following is a reformulation of a result due to Atiyah and Bott \cite[Thm. 9.11]{AB} (see also \cite[Prop. 3.14]{DX}). 

\begin{prop}
	The cohomology ring $H^\ast(\cN_g)$ is generated by the elements $\alpha_r,\beta_r,\psi_r^i$ where $2\leq r\leq N$, $1\leq i\leq 2g$.
\end{prop}

\noindent This result induces the isomorphism \eqref{eq:ordinarycohomologypres} mentioned earlier.

We next turn to relations for these generators. Let $J_g$ be the Jacobian torus of $\Sigma_g$, viewed as the moduli space of flat $U(1)$-connections on $\Sigma_g$, or equivalently, the moduli space of holomorphic line bundles of degree zero. Let
\[
	V \to \cN_g\times J_g \times \Sigma_g.
\]
be defined as the tensor product of the pullback of $U$ with the pullback of the Poincar\'{e} bundle over $J_g\times\Sigma_g$. We have
\begin{equation}\label{eq:c1universal}
	c_1( V) = d \cdot 1\otimes 1 \otimes \sigma + \sum_{i=1}^{2g} 1\otimes d_i \otimes \xi_i + x \otimes 1 
\end{equation}
where $x\in H^2(\cN_g\times J_g)$, and $d_i\in H^1(J_g)$ generate $H^\ast(J_g)$ as an exterior algebra. Consider the projection $f:\cN_g\times J_g\times \Sigma_g\to \cN_g\times J_g$. The Grothendieck--Riemann--Roch formula expresses $c_i(f_! V)$ in terms of the generators \eqref{eq:moduligens} and elements of $H^\ast(J_g)$. Now assume
\begin{equation}\label{eq:degreeassumption}
	d = 2N(g-1) + d', \qquad 1\leq d'  < N.
\end{equation}
Throughout this subsection, $d'$ is fixed, and $g$ is a positive integer. As a consequence of stability and Serre duality, $H^1(\Sigma_g;\mathcal{E}\otimes \mathcal{L})=0$ for any stable rank $N$ bundle $\mathcal{E}$ and degree zero holomorphic line bundle $\mathcal{L}$ over $\Sigma_g$. Therefore $f_! V$ is an honest vector bundle over $\cN_g\times J_g$, whose rank can be computed using Riemann-Roch. Consequently, we have
\begin{equation}\label{eq:mumfordrels}
	c_i(f_! V)=0 \quad \text{ if } \quad i > \text{rk} (f_! V) = N(g-1)+d'. 
\end{equation}
Taking slant products of the Chern classes \eqref{eq:mumfordrels} with elements in $H^\ast(J_g)$ thus yields relations for the generators \eqref{eq:moduligens}. We call these {\emph{Mumford relations}}, following the discussion in \cite{AB}. In the case $N=2$, the Mumford relations were shown to be a complete set of relations for the ring $H^\ast(\cN_g)$ by Kirwan \cite{kirwan}.

When $N>2$, the Mumford relations do not give a complete set of relations for $H^\ast(\cN_g)$. Following \cite{earl}, we consider a line bundle $L\to \Sigma_g$ of degree $4(g-1)+1$. Let $\phi:\cN_g\times J_g\times \Sigma_g\to \Sigma_g$ be projection. Define the ``dual'' universal bundle
\[
	\overline V := V^\ast \otimes \phi^\ast L.
\]
Then under assumption \eqref{eq:degreeassumption}, a similar argument using stability and Serre duality implies that $f_!\overline{V}$ is an honest vector bundle. We then obtain
\begin{equation}\label{eq:dualmumfordrels}
	c_i(f_! \overline{V})=0 \quad \text{ if } \quad i > \text{rk} (f_! \overline{V}) = Ng-d'. 
\end{equation}
Again, taking slant products of the classes \eqref{eq:mumfordrels} with elements in $H^\ast(J_g)$ yields relations for the generators \eqref{eq:moduligens}. We call these {\emph{dual Mumford relations}}, following Earl. In the case $N=3$, the work of Earl \cite{earl} implies that the Mumford relations and the dual Mumford relations form a complete set of relations for $H^\ast(\cN_g)$.

In the following, we use Grothendieck--Riemann--Roch to compute the Chern classes of $f_! V$, $f_! \overline{V}$ and then use \eqref{eq:mumfordrels} and \eqref{eq:dualmumfordrels} to obtain relations in the cohomology ring $H^\ast(\cN_g)$. For this purpose, we may assume $x$ in \eqref{eq:c1universal} is zero, by tensoring $V$ by a formal line bundle over $\cN_g\times J_g$ whose first Chern class is $-x/N$. Following an observation of Zagier \cite[p.22]{zagier}, this assumption does not affect \eqref{eq:mumfordrels} and \eqref{eq:dualmumfordrels}.

When $N>3$, the Mumford relations and dual Mumford relations are not complete, and more relations are necessary. A complete set of relations for general $N$ was given by Earl and Kirwan \cite{earl-kirwan}. As our focus is the case $N=3$, we will only consider the Mumford and dual Mumford relations. The particular elements we consider are
\begin{equation}\label{eq:defzetas}
	\zeta_{m}^{g,k} := (-N)^k c_{m+k}(f_! V)/D_k, \qquad \overline{\zeta}_{m}^{g,k} :=(-N)^k c_{m+k}(f_! \overline{V})/D_k 
\end{equation}
where $D_k\in H_{2k}(J_g)$ has pairing $1$ with $d_1d_{g+1}d_2d_{g+2} \cdots d_k d_{g+k}\in H^{2k}(J_g)$ and trivial pairing with other exterior products of $d_i$. More precisely, we consider these classes in terms of the generators \eqref{eq:moduligens} as derived from Grothendieck--Riemann--Roch. Thus
\[
	\zeta_{m}^{g,k} , \;\;  \overline{\zeta}_{m}^{g,k} \;\; \in \bA^N_g = \C[\alpha_2,\ldots,\alpha_N,\beta_2,\ldots,\beta_N]\otimes \Lambda^\ast(\psi_r^i)
\]
By our discussion thus far, the cohomology ring for $\cN_g$ may be written as in \eqref{eq:ordinarycohomologypres},
\[
	H^\ast(\cN_g)= \bA^N_g / I_g,
\]
and using \eqref{eq:mumfordrels}, \eqref{eq:dualmumfordrels}, the ideal of relations $I_g$ contains the following elements:
\begin{gather}
	\zeta_m^{g,k}\in I_g \quad \text{ if }\quad m > N(g-1)-k+d', \label{eq:zetarelbound1} \\[0.3cm] 
	\overline{\zeta}_m^{g,k}\in I_g \quad \text{ if }\quad m > Ng-k-d'. \label{eq:zetarelbound2}
\end{gather}
We now study these relations after modding out by the classes $\psi_r^i$. 
In the computations below, there will frequently appear two constants:
\begin{equation}\label{eq:Nconsts}
	\mathsf{c}_{N,d'} := 1-\frac{d'}{N}, \qquad \overline{\mathsf{c}}_{N,d'} := 1 -\mathsf{c}_{N,d'}=  \frac{d'}{N}.
\end{equation}

We first obtain an expression for the generating functions of the polynomials $\zeta^{g,k}_m$ (with respect to the index $m$). Below, the notation ``$\equiv_\psi$'' means congruence modulo the ideal $(\psi_r^i)_{2\leq r\leq N, 1\leq i\leq 2g}$. By convention, we also set $\beta_1=\alpha_1=\alpha_0=0$ and $\beta_0=1$.

\begin{prop}\label{prop:generatingfunction}
	The generating series $F_{g,k}(t) := \sum_{m=0}^\infty \zeta_m^{g,k} t^m$ {\rm{mod}} $(\psi_r^i)$ is given by:
	\begin{equation*}
	\small F_{g,k}(t) \equiv_\psi \left( \sum_{i=0}^N \beta_i t^i \right)^{g-k-\mathsf{c}_{N,d'}} \left( \sum_{i=0}^N (1-\frac{i}{N})\beta_i t^i \right)^k G(t)
\end{equation*}
where the power series $G(t)\in \Q[\alpha_2,\ldots,\alpha_N,\beta_2,\ldots,\beta_N][\![t]\!]$ is defined by 
\begin{equation}\label{eq:goft}
	G(t)=\exp\left({\sum_{i=1}^N \alpha_i \frac{\partial}{\partial \beta_i} \left(\sum_{n=1}^\infty -\frac{(-t)^np_{n+1}}{n(n+1)}\right)}\right).
\end{equation}
\end{prop}

\vspace{.25cm}

The notation $p_{n}$ refers to the $n^\text{th}$ power symmetric function, viewed as a function of the elementary symmetric functions. More explicitly, recall that given variables $x_1,x_2,\ldots$ there are the elementary symmetric functions $e_n$ and the power symmetic functions $p_n$:
\[
	e_n = \sum_{i_1 < i_2 <\cdots < i_n} x_{i_1} x_{i_2}\cdots x_{i_n} \qquad \qquad p_n = \sum_i x_i^n
\] 
It is a basic result that $p_n$ can be written as a function in the elementary symmetric functions: $p_n=p_n(e_1,e_2,\ldots)$. An explicit relationship is given by
\begin{equation}\label{eq:sympolyrel}
	-\sum_{n=1}^\infty \frac{(-t)^np_n}{n}  = \log \left(\sum_{n=0}^\infty {e_n t^n} \right).
\end{equation}
 In the formula of Proposition \ref{prop:generatingfunction}, $\partial p_{n} / \partial\beta_i$ should be interpreted by identifying the $\beta_i$ with $e_i$ for $i=0$ and $2\leq i\leq N$, and setting the other $e_i=0$. Explicitly,
 \[
 		\frac{\partial p_{n}}{\partial\beta_i} = \frac{\partial}{\partial e_i}p_n(e_1,e_2,\ldots)\Bigr|_{e_i=\beta_i}
 \]  

\vspace{.1cm}

\begin{proof}[Proof of Proposition \ref{prop:generatingfunction}]
We adapt the computation of Zagier \cite[\S 6]{zagier}, which is for the case $N=2$. Grothendieck--Riemann--Roch says
\begin{equation}\label{eq:grr}
	\text{ch}(f_!  V ) = \left( \text{ch}( V) \text{td}(\Sigma_g)\right)/[\Sigma_g]
\end{equation}
where $\text{td}(\Sigma_g)=1-(g-1)\sigma$ is the Todd class of $\Sigma_g$. We have 
\[
	\text{ch}(V)=\text{ch}(\bP)\text{ch}(\det(V)^{1/N})
\]
where $\text{ch}(\bP)$ is interpreted via the natural map $H^\ast(\cN_g\times \Sigma_g)\to H^\ast(\cN_g\times J_g\times \Sigma_g)$ induced by projection. Using \eqref{eq:c1universal} (and recalling $x=0$) we obtain
\[
	\text{ch}(\text{det}(V)^{1/N}) = 1 + \frac{1}{N} d\otimes 1\otimes \sigma + \frac{1}{N}\sum_{i=1}^{2g} 1\otimes d_i\otimes \xi_i -  \frac{1}{N^2}1\otimes A\otimes \sigma
\]
where $A=\sum_{i=1}^g d_id_{g+i}\in H^2(J_g)$. We next give an expression for \text{ch}(\bP). Let $\gamma_i$ (resp. $\delta_i$), where $1\leq i \leq N$, be formal degree $2$ classes such that the $i^\text{th}$ elementary symmetric polynomial in the $\gamma_i$ (resp. $\delta_i$) is equal to $c_i(\bP)$ (resp. $\beta_i$). Then
\[
	\text{ch}(\bP) \;\; = \;\; \sum_{i=1}^N e^{\gamma_i} \;\; = \;\;  \sum a_{r_1,\ldots,r_N} c_1(\bP)^{r_1}\cdots c_{N}(\bP)^{r_N}
\]
for some constants $a_{r_1,\ldots,r_N}$. A direct computation from \eqref{eq:moduligensdef} gives
\begin{equation*}\label{eq:powerchernclassinproof}
	c_{1}(\bP)^{r_1}\cdots c_{N}(\bP)^{r_N} \,\equiv_\psi \, \beta_1^{r_1}\cdots \beta_{N}^{r_N}\otimes 1\otimes 1 + \sum_{i=0}^N\alpha_i \frac{\partial}{\partial\beta_i}(  \beta_1^{r_1}\cdots \beta_{N}^{r_N} )\otimes 1 \otimes \sigma.
\end{equation*}
Together with the identity $\sum_{i=1}^N e^{\delta_i} =   \sum a_{r_1,\ldots,r_N}  \beta_1^{r_1}\cdots \beta_{N}^{r_N}$, these relations yield
\[
	\text{ch}(\bP) \;\; \equiv_\psi \;\; \sum_{i=1}^N e^{\delta_i}\otimes 1\otimes 1 + \sum_{i,j=1}^N \alpha_i \frac{\partial}{\partial \beta_i} (e^{\delta_j}) \otimes 1\otimes \sigma
\]
With these observations in hand, we evaluate \eqref{eq:grr} to be
\begin{equation}\label{eq:chofdirectimage}
	\text{ch}(f_!V)\;\; \equiv_\psi \;\;(\frac{d}{N} - (g-1) - \frac{1}{N^2} A)\sum_{i=1}e^{\delta_i} \otimes 1 + \sum_{i,j=1}^N\alpha_i \frac{\partial}{\partial \beta_i}e^{\delta_j} \otimes 1
\end{equation}
To determine the Chern classes from this expression, we use the following \cite[Lemma 1]{zagier}: for any vector bundle $E$ over a space $X$, we have 
\begin{equation}\label{eq:chernclasstochernchar}
	\log c(E)  =- \sum_{n=1}^\infty \frac{(-1)^n}{n} s_n \qquad \Longleftrightarrow \qquad \text{ch}(E) = \text{rk}(E) + \sum_{n=1}^\infty \frac{s_n}{n!}
\end{equation}
where $s_n=s_n(E)\in H^{2n}(X)$. For $E=f_!V$, using \eqref{eq:chofdirectimage} we compute 
\[
	s_n \;\equiv_\psi\; (\frac{d}{N}-(g-1))p_n - \frac{n}{N^2}Ap_{n-1} + \sum_{i=1}^N\frac{1}{n+1}\alpha_i \frac{\partial}{\partial \beta_i} p_{n+1}
\]
where $p_n$ is the $n^{\text{th}}$ power symmetric function in $\delta_1,\ldots , \delta_N$. From \eqref{eq:chernclasstochernchar}, we obtain
\begin{equation}\label{eq:chernclassseriesinproof}
	c(f_!V)_t \; \equiv_\psi \; \left(\sum_{i=0}^N \beta_i t^i\right)^{\frac{d}{N}-(g-1)}\exp\left(-\frac{At}{N^2}\sum_{n=1}^\infty (-t)^{n-1}p_{n-1}\right)G(t)
\end{equation}
Here we use the notation $c(E)_t=\sum_{i=0}^\infty c_i(E)t^i$ for the power series associated to the total Chern class of $E$. Using $A^r / D_k = r!\delta_{rk}$, we obtain an expression for the slant product: 
\[
	c(f_!V)_t / D_k \; \equiv_\psi \; \left(\sum_{i=0}^N \beta_i t^i\right)^{\frac{d}{N}-(g-1)}\left(\frac{-t}{N^2}\sum_{n=1}^{\infty}(-t)^{n-1} p_{n-1}\right)^k G(t)
\]
Taking the derivative of relation \eqref{eq:sympolyrel} (with $e_k=\beta_k$) gives
\[
	\sum_{n=1}^\infty (- t)^{n-1}p_n = \sum_{i=0}^N i\beta_it^{i-1} / \sum_{i=0}^N \beta_i t^i.
\]
Noting $p_0=N$, we then have the relation
\[
	\sum_{n=1}^\infty (-t)^{n-1}p_{n-1} =  N - \sum_{i=0}^N i\beta_it^i / \sum_{i=0}^N \beta_i t^i = \sum_{i=0}^N (N-i)\beta_i t^i / \sum_{i=0}^N\beta_i t^i.
\]
Substituting this last expression into \eqref{eq:chernclassseriesinproof}, and using $t^k F_{g,k}(t)=(-N)^k  c(f_!V)_t/D_k$, as determined by \eqref{eq:defzetas}, gives the result.
\end{proof}

\begin{remark}
Setting all $\psi_r^i$ equal to zero simplifies this computation considerably. Explicit formulas can of course be obtained without this simplification; see \cite{earl,kirwan} for computations along these lines (using different generators).
\end{remark}

From Proposition \ref{prop:generatingfunction} we derive a recursive relation for the $\zeta_m^{g,k}$.

\begin{prop}\label{prop:mainrecursion}
	The polynomials $\zeta_m^{g,k} \; {\rm{mod}}\; (\psi_r^i)$ in the ring $\C[\alpha_2,\ldots,\alpha_N,\beta_2,\ldots,\beta_N]$ are determined recursively, for fixed $g,k$, as follows (with $\zeta_0^{g,k}=1$ and $\zeta_{m}^{g,k}=0$ for $m<0$):
	\begin{align}
	 N(m+1)\zeta_{m+1}^{g,k} &\;\;\equiv_\psi\;\; -\sum_{i,j=0}^N (N-j)\alpha_i\beta_j \zeta_{m-i-j+2}^{g,k} \;\;\; + \hspace{2cm} \label{eq:recursion1}   \\ 
 \hspace{-0.1cm} \sum_{\substack{i,j=0\\ (i,j)\neq (0,0)}}^N ( (g-k-\mathsf{c}_{N,d'}) & i(N-j) - (N-i)(m-i-j+1) +kj(N-j)) \beta_i\beta_j \zeta_{m-i-j+1}^{g,k} \nonumber
\end{align}
\end{prop}

\begin{proof}
	First consider the series $G(t)$ from \eqref{eq:goft}. We compute 
	\begin{align*}
		\frac{G'(t)}{G(t)} &= \sum_{i=0}^N t^{-2} \alpha_i \frac{\partial}{\partial \beta_i}\left( \sum_{n=1}^\infty \frac{(-1)^{n+1} p_{n+1} t^{n+1}}{n+1} \right) \\[.2cm]
						& = -\sum_{i=0}^N t^{-2} \alpha_i \frac{\partial}{\partial \beta_i}\left( \log \sum_{j=0}^N \beta_j t^j \right) = -\sum_{i=0}^N \alpha_i t^{i-2} / \sum_{j=0}^N \beta_j t^j.
	\end{align*}
	In the second equality, we have again used \eqref{eq:sympolyrel} with $e_k=\beta_k$. Next, we compute
	\begin{align*}
		\frac{F_{g,k}'(t)}{F_{g,k}(t)} = (g-k-\mathsf{c}_{N,d'})&\left(\sum_{i=0}^N i\beta_i t^{i-1}\right) \left(\sum_{i=0}^N \beta_i t^i \right)^{-1} \\[0.2cm]
		+ k\left(\sum_{i=0}^N i(1-\frac{i}{N})\beta_i t^{i-1}\right) & \left(\sum_{i=0}^N (1-\frac{i}{N})\beta_i t^i \right)^{-1}   - \left(\sum_{i=0}^N  \alpha_i t^{i-2}\right)\left(\sum_{i=0}^N \beta_i t^i\right)^{-1}. 
	\end{align*}
	Multiply both sides by $(\sum_{i=0}^N \beta_i t^i)(\sum_{i=0}^N (N-i)\beta_i t^i )F_{g,k}(t)$. Then, using 
	\[
		F_{g,k}(t) = \sum_{m=0}^\infty \zeta_{m}^{g,k}t^m,\qquad \quad F_{g,k}'(t)=\sum_{m=0}^\infty m \zeta_{m}^{g,k} t^{m-1},
\]
	 the desired recursion follows by extracting the coefficient of $t^m$ from each side. 
\end{proof}

Many other recursions may be extracted from Proposition \ref{prop:generatingfunction}. For example:

\begin{prop}\label{prop:rec3}
	The polynomials $\zeta_m^{g,k} \; {\rm{mod}}\; (\psi_r^i)$ in $\C[\alpha_2,\ldots,\alpha_N,\beta_2,\ldots,\beta_N]$ satisfy:
	\begin{align}
	 \zeta^{g+1,k}_m \;\; & \equiv_\psi \;\; \sum_{i=0}^N \beta_i\zeta_{m-i}^{g,k} \label{eq:recursion2} \\[.2cm]
	 \sum_{i=0}^N \beta_i\zeta_{m-i}^{g,k+1}\;\; & \equiv_\psi \;\;\sum_{i=0}^N \left( 1-i/N\right) \beta_i\zeta_{m-i}^{g,k}\label{eq:recursion3}
\end{align}
\end{prop}

\begin{proof}
The first relation follows using $F_{g+1,k}(t)=(\sum_{i=0}^N\beta_it^i)F_{g,k}(t)$, and the second relation follows from $(\sum_{i=0}^N \beta_it^i)F_{g,k+1}(t)=(\sum_{i=0}^N (1-i/N)\beta_it^i)F_{g,k}(t)$.
\end{proof}

The case of the dual Mumford elements $\overline{\zeta}_m^{g,k}$ is similar. In fact, if one makes the following changes to the expression for $F_{g,k}(t)$ in Proposition \ref{prop:generatingfunction}:
\begin{equation}\label{eq:dualchange}
	\alpha_i \mapsto \overline{\alpha}_i:= (-1)^i \alpha_i, \qquad \beta_i \mapsto \overline{\beta}_i:= (-1)^i \beta_i, \qquad \mathsf{c}_{N,d'} \mapsto \overline{\mathsf{c}}_{N,d'},
\end{equation}
then one obtains the generating function for $\overline{\zeta}_m^{g,k}$. For example, recursion \eqref{eq:recursion1} becomes
	\begin{align*}
	 N(m+1)\overline{\zeta}_{m+1}^{g,k} &\;\;\equiv_\psi\;\; -\sum_{i,j=0}^N (N-j)\overline{\alpha}_i\overline{\beta}_j\overline{\zeta}_{m-i-j+2}^{g,k} \;\;\; + \hspace{2cm}     \\ 
 \hspace{-0.1cm} \sum_{\substack{i,j=0\\ (i,j)\neq (0,0)}}^N ( (g-k-\overline{\mathsf{c}}_{N,d'}) & i(N-j) - (N-i)(m-i-j+1) +kj(N-j)) \overline{\beta}_i\overline{\beta}_j \overline{\zeta}_{m-i-j+1}^{g,k} \nonumber
\end{align*}
with the same initial conditions: $\overline{\zeta}_0^{g,k}=1$ and $\overline{\zeta}_{m}^{g,k}=0$ for $m<0$.

In the remainder of this subsection, we prove two lemmas that will be used later to understand the ideal of relations in the case of $N=3$.

\begin{lemma}\label{beta-2-2-rela}
Suppose $N\geq 3$. For any integers $g,k,m$ with $k\neq N/2-1$ we have:
\begin{align}
	&\beta_2^2 \zeta_{m-1}^{g,k} \in (\zeta_{m+1}^{g,k}, \zeta_{m+2}^{g,k},\zeta_{m+3}^{g,k},\alpha_a, \beta_2^3,\beta_b,\psi_{r}^i), \label{eq:b221N}\\[0.35cm] 
	& \beta_2^2 \overline{\zeta}_{m-1}^{g,k} \in (\overline{\zeta}_{m+1}^{g,k}, \overline{\zeta}_{m+2}^{g,k},\overline{\zeta}_{m+3}^{g,k},\alpha_a, \beta_2^3,\beta_b,\psi_{r}^i), \label{eq:b222N}
\end{align}
where the indices range over $4\leq a\leq N$, $3\leq b\leq N$, $2\leq r\leq N$, $1\leq i\leq 2g$.
\end{lemma}

\begin{proof}
	In this proof we write ``$\equiv$'' to mean congruent modulo $(\beta_2^3,\beta_b,\psi_{r}^i)$ with $3\leq b\leq N$, $2\leq r\leq N$, $1\leq i\leq 2g$. We prove \eqref{eq:b221N}, the case of \eqref{eq:b222N} being similar. First note that we can write \eqref{eq:recursion1} as follows:
	\begin{equation}\label{eq:betalemma1}
		(m+1)\zeta_{m+1}^{g,k} \equiv -  \sum_{i=2}^N \alpha_i \zeta_{m-i+2}^{g,k} -  q \sum_{i=2}^N\alpha_i\beta_2\zeta_{m-i}^{g,k} + r_{m}^{g,k}\beta_2\zeta_{m-1}^{g,k} + s_{m}^{g} \beta_2^2\zeta_{m-3}^{g,k}  .
	\end{equation}
	The constants here are given by $q=(N-2)/N$ and
	\begin{gather*}
		r_m^{g,k}=2g -\frac{4}{N}k -2\mathsf{c}_{N,d'} +2(m-1)\left(\frac{1}{N}-1\right), \\[.2cm]
		s_{m}^{g}=\left(\frac{N-2}{N}\right)(2g  - m -2\mathsf{c}_{N,d'} +3).
	\end{gather*}
	Now multiply \eqref{eq:betalemma1} by $\beta_2$ to obtain the following:
	\begin{equation}\label{eq:betalemma2}
		(m+1)\beta_2\zeta_{m+1}^{g,k} \equiv - \sum_{i=2}^N\alpha_i \beta_2 \zeta_{m-i+2}^{g,k} -  q\sum_{i=2}^N\alpha_i\beta^2_2\zeta_{m-i}^{g,k} + r_{m}^{g,k}\beta^2_2\zeta_{m-1}^{g,k} .
	\end{equation}
	Multiplying once more by $\beta_2$ gives
		\begin{equation}\label{eq:betalemma3}
		(m+1)\beta_2^2\zeta_{m+1}^{g,k} \equiv - \sum_{i=2}^N \alpha_i \beta_2^2\zeta_{m-i+2}^{g,k} .
	\end{equation}
		We can use \eqref{eq:betalemma3} to rewrite the middle term on the right side of \eqref{eq:betalemma2}, yielding:
			\begin{equation}\label{eq:betalemma4}
		(m+1)\beta_2\zeta_{m+1}^{g,k} \equiv - \sum_{i=2}^N\alpha_i \beta_2 \zeta_{m-i+2}^{g,k} + q(m+1)\beta^2_2\zeta_{m-1}^{g,k} + r_{m}^{g,k}\beta^2_2\zeta_{m-1}^{g,k} .
	\end{equation}
	The first term on the right side of \eqref{eq:betalemma4} can be rewritten using \eqref{eq:betalemma1}:
	\begin{align*}
		- \sum_{i=2}^N\alpha_i\beta_2\zeta_{m-i+2}^{g,k} \equiv & q^{-1}(m+3) \zeta_{m+3}^{g,k} + q^{-1} \sum_{i=2}^N \alpha_i \zeta_{m-i+4}^{g,k} \\
		&- q^{-1}r_{m+2}^{g,k}\beta_2\zeta_{m+1}^{g,k} - q^{-1}s_{m+2}^{g} \beta_2^2\zeta_{m-1}^{g,k}  .
	\end{align*}
	Note $q\neq 0$ since $N\geq 3$. Substituting this into \eqref{eq:betalemma4} we obtain
	\begin{align*}
		(m+1)\beta_2\zeta_{m+1}^{g,k} & \equiv  q^{-1}(m+3) \zeta_{m+3}^{g,k} +  q^{-1}\sum_{i=2}^N \alpha_i \zeta_{m-i+4}^{g,k} - q^{-1}r_{m+2}^{g,k}\beta_2\zeta_{m+1}^{g,k}  \\[0.2cm] & \qquad   - q^{-1}s_{m+2}^{g} \beta_2^2\zeta_{m-1}^{g,k}  + q(m+1)\beta^2_2\zeta_{m-1}^{g,k} + r_{m}^{g,k}\beta^2_2\zeta_{m-1}^{g,k} .
	\end{align*}
	Rearranging, we obtain the following expression:
	\begin{align*}
		 \left( q^{-1}s_{m+2}^{g}  -q(m+1)- r_{m}^{g,k}\right)\beta_2^2\zeta_{m-1}^{g,k} &\equiv \\[.2cm]
		q^{-1}(m+3) \zeta_{m+3}^{g,k} +  q^{-1} \sum_{i=2}^N  \alpha_i  & \zeta_{m-i+4}^{g,k} - q^{-1}r_{m+2}^{g,k}\beta_2\zeta_{m+1}^{g,k} - (m+1)\beta_2\zeta_{m+1}^{g,k}.
	\end{align*}
	The constant on the left side is equal to $(4k+4-2N)/N$, which is non-zero under the assumption of the proposition statement. Inspection of the right side of this last expression proves that $\beta_2^2\zeta_{m-1}^{g,k}$ is in the ideal claimed.
\end{proof}

\begin{lemma}\label{beta-2-rela}
Suppose $N\geq 3$. For any integers $g,k,m$ we have the following:
\begin{align}
	&\beta_2\zeta_{m-2}^{g,k+1} \in (\zeta_m^{g,k+1}, \zeta_m^{g,k},\beta_2^2,\beta_b,\psi_{r}^i), \label{eq:b21N}\\[0.35cm] 
	& \beta_2\overline{\zeta}_{m-2}^{g,k+1} \in (\overline{\zeta}_m^{g,k+1}, \overline{\zeta}_m^{g,k},\beta_2^2,\beta_b,\psi_{r}^i), \label{eq:b22N}
\end{align}
where the indices range over $3\leq b\leq N$, $2\leq r\leq N$, $1\leq i\leq 2g$.
\end{lemma}

\begin{proof}
	We prove \eqref{eq:b21N}, the case \eqref{eq:b22N} being similar. Consider relation \eqref{eq:recursion3}:
	\[
			\zeta_{m}^{g,k+1} + \beta_2 \zeta_{m-2}^{g,k+1}\;\; \equiv \;\; \zeta_{m}^{g,k} + \frac{N-2}{N}\beta_2\zeta_{m-2}^{g,k}
	\]
	where in this proof ``$\equiv$'' means congruent modulo the ideal $(\beta_2^2,\beta_b,\psi_r^i)$ where $b\geq 3$. Multiply this expression by $\beta_2$ and shift subscripts to obtain
	\[
		\beta_2\zeta_{m-2}^{g,k+1} \;\; \equiv \;\; \beta_2\zeta_{m-2}^{g,k}.
	\]
	These two expressions yield the following, which implies the result:
	\[
		\beta_2\zeta_{m-2}^{g,k+1}\;\; \equiv \;\;  \frac{N}{2} \left(\zeta_m^{g,k} - \zeta_m^{g,k+1}\right). \;\;\; \qedhere
	\]
\end{proof}

\section{The $N=3$ undeformed simple type quotient}\label{sec:simpletypequotient}
We now specialize to the case $N=3$ and focus on a quotient of the cohomology ring of the moduli space $\cN_{g}=\cN_{g,d}^N$. To be more specific, fix the choice of $d'$ in \eqref{eq:degreeassumption} to be $d'=1$. There is no loss in generality in making this choice, as the moduli space for $d'=2$ may be identified with that for $d'=1$ by the map which sends a stable rank $3$ bundle to its conjugate. We take the quotient of $H^\ast(\cN_g)$ by the curve classes $\psi_r^i$ and the cohomology classes $\beta_2^3$, $\beta_3$ given by the point classes. This quotient is a cyclic module over the ring $\bA^3_g$:
\begin{equation}\label{undeformed-simple-type-quotient}
	H^\ast(\cN_g)/(\beta_2^3,\beta_3,\psi_2^i,  \psi_3^i)_{1\leq i\leq 2g}= \bA^3_g / \widetilde I_g
\end{equation}
where $\widetilde I_g$ is the ideal $I_g+(\beta_2^3,\beta_3,\psi_2^i,  \psi_3^i)_{1\leq i\leq 2g}$. This quotient was introduced in Section \ref{sec:strategy}, where it was called the undeformed simple type quotient. Our main goal of this subsection is to prove Theorem \ref{thm:undeformeddimension}, which we restate here:
\[
	\dim_\C \bA^3_g / \widetilde I_g\leq (2g-1)^2 \qquad (g\geq 1).
\]
We continue to work with coefficients in $\C$, following our convention as set in the previous sections, although everything here works over $\bQ$.

We use the graded reverse lexicographic monomial ordering when dealing with polynomials in $\alpha_2,\alpha_3,\beta_2$, where the degrees of these elements are respectively $2,4,4$. In particular, 
\[
	\alpha_2^i\alpha_3^j\beta_2^k > \alpha_2^{i'}\alpha_3^{j'}\beta_2^{k'}
\]
if either $2i+4j+4k>2i'+4j'+4k'$, or $2i+4j+4k=2i'+4j'+4k'$ and the right-most nonzero entry of $(i-i',j-j',k-k')$ is negative. When taking leading terms below, it is always with respect to this monomial ordering. The leading term of a polynomial $p$ in the variables $\alpha_2,\alpha_3,\beta_2$ with this convention is denoted $\text{LT}(p)$.

We start with a simpler variation of \eqref{undeformed-simple-type-quotient} where we take the quotient of $H^\ast(\cN_g)$ by the curve classes $\psi_r^i$ and the point classes $\beta_r$ with $r=2, 3$ and $1\leq i \leq 2g$. Modulo the curve and the point classes, the power series $F_{g,k}(t)$ or Proposition \ref{prop:generatingfunction} is equal to $G(t)$, which is independent of $g$ and $k$. In fact, modulo the point classes, $G(t)$ is equal to $\exp(-\alpha_2 t-\alpha_3t^2/2)$. Motivated by this, let $\zeta_n,\overline \zeta_n \in \C[\alpha_2,\alpha_3]$ be defined by 
\[
  \sum_{n=0}^\infty \zeta_nt^n=\exp(\alpha_2 t+\alpha_3\frac{t^2}{2}),\hspace{1cm}
  \sum_{n=0}^\infty \overline \zeta_nt^n=\exp(\alpha_2 t-\alpha_3\frac{t^2}{2}).
\]
More explicitly, we have the expressions
\[
  \zeta_n=\sum_{0\leq j\leq n/2}\frac{1}{(n-2j)! j! 2^j}\alpha_2^{n-2j}\alpha_3^j,\hspace{1cm}
  \overline \zeta_n=\sum_{0\leq j\leq n/2}\frac{(-1)^j}{(n-2j)! j! 2^j}\alpha_2^{n-2j}\alpha_3^j.
\]
These polynomials satisfy the recursive relations 
\begin{equation}\label{recursive-formula}
  	m\zeta_m=\alpha_2\zeta_{m-1}+\alpha_3\zeta_{m-2},\hspace{1cm}
	m\overline\zeta_m=\alpha_2\overline\zeta_{m-1}-\alpha_3\overline\zeta_{m-2},
\end{equation}
in a similar way that $\zeta_{m}^{g,k}$ satisfy the relations in Proposition \ref{prop:mainrecursion}.

\begin{prop}\label{zeta-n-leading-term-ideal}
	The leading term ideal of the 
	ideal $I^0_n:=(\zeta_n,\zeta_{n+1},\overline \zeta_{n+1},\overline\zeta_{n+2})$ in the ring $\C[\alpha_2,\alpha_3]$ 
	includes the following monomials:
	\begin{equation}\label{gen-set-I0g}
	  \{\alpha_2^i\alpha_3^j\mid 2i+3j\geq 2n \}.
	\end{equation}	
\end{prop}
\noindent In fact, it will be a consequence of our proof of Theorem \ref{thm:undeformeddimension} that \eqref{gen-set-I0g} is a generating set for the leading term ideal of $I^0_n$.
\begin{proof}
	Define $\sigma_n$ and $\overline \sigma_n$ respectively as $(\zeta_n+\overline \zeta_n)/2$ and 
	$(\zeta_n-\overline \zeta_n)/2$. Then we have
	\[
	  \sigma_n:=\sum_{\substack{0\leq j\leq n/2\\j\overset{2}{\equiv} 0}}
	  \frac{1}{(n-2j)! j! 2^j}\alpha_2^{n-2j}\alpha_3^j,\hspace{1cm}
	\overline \sigma_n=\sum_{\substack{0\leq j\leq n/2\\j\overset{2}{\equiv} 1}}\frac{1}{(n-2j)! j! 2^j}\alpha_2^{n-2j}\alpha_3^j.
	\]
	First we claim that for any $1\leq i \leq n/3$, there are constants $c_0, c_1, 
	\ldots, c_{i-1}\in \Q$
	such that 
	\begin{equation}\label{LT-linear-comb-sigma-bar}
	  {\rm LT}(\sum_{j=0}^{i-1} c_j\alpha_2^{j}\overline \sigma_{n+i-j})=
	  \alpha_2^{n-3i+2}\alpha_3^{2i-1}.
	\end{equation}
	Suppose $p_0(x)=1$ and for $n\geq 1$, define $p_n(x)$ as the degree $n$ polynomial 
	$x(x+1)\cdot (x+n-1)$.  
	A straightforward computation shows that \eqref{LT-linear-comb-sigma-bar} is equivalent to
	finding $c_j$ satisfying 
	\begin{equation}\label{lin-sys}
	  \sum_{0\leq j \leq i-1}c_jp_{j}(n+i-1-2k)=\left\{\begin{array}{cc}
	  0& 1\leq k \leq 2i-1 \text{ and } k\overset{2}{\equiv} 1\\
	  c  & k=2i-1\end{array}\right.
	\end{equation}
	for some non-zero constant $c$. Suppose $M$ is the square matrix of size $i$ 
	such that for $0\leq m,j\leq i-1$, the $(m,j)$ entry of $M$ is equal to $p_j(n+i-4m+1)$.
	The linear system in \eqref{lin-sys} has a solution if $M$ is invertible. The determinant of 
	$M$ is equal to the determinant of the matrix $M'$ whose $(m,j)$ entry is $(n+i-4m+1)^j$. 
	This can be seen by applying a sequence of column operations. Now the matrix $M'$ is a 
	Vandermonde matrix, and it is easily seen that it is invertible. 

	A similar argument shows that for $0\leq i \leq n/3$, there are $d_0, d_1, \cdots,
	d_i \in \Q$ such that 
	\begin{equation}\label{LT-linear-comb-sigma}
	  {\rm LT}(\sum_{j=0}^{i} d_j\alpha_2^{j}\sigma_{n+i-j})=
	  \alpha_2^{n-3i}\alpha_3^{2i}.
	\end{equation}	
	We remark that the polynomial $\sum d_j\alpha_2^{j}\sigma_{n+i-j}$ in 
	\eqref{LT-linear-comb-sigma} is homogenous of degree $2n+2i$ with respect to the 
	grading defined on $\C[\alpha_2,\alpha_3]$.
	Furthermore, all the monomials appearing in this polynomial have an even power of $\alpha_3$.
	Similarly, the polynomial in \eqref{LT-linear-comb-sigma-bar} is 
	homogenous of degree $2n+2i$ and it contains only monomials with odd powers of 
	$\alpha_3$.
	
	Recursive formulas in \eqref{recursive-formula} and the identity in 
	\eqref{LT-linear-comb-sigma-bar} show that $ \alpha_2^{i}\alpha_3^{j}$ with $2i+3j\geq 2n$ and 
	$j$ being odd belongs to the leading term ideal of $I^0_n$.
	We cannot use \eqref{LT-linear-comb-sigma} to treat the case that $j$ is even because the polynomial in \eqref{LT-linear-comb-sigma} contains the term 
	$\sigma_n$ which does not belong to $I^0_n$. However, if we replace $\sigma_n$ in this sum
	with $\zeta_n$, which is an element of $I^0_n$, we obtain:
	\begin{align}
	  d_i\alpha_2^i\zeta_n+\sum_{j=0}^{i-1} d_j\alpha_2^{j}\sigma_{n+i-j}
	=&b_1\alpha_2^{n+i-2}\alpha_3+b_3\alpha_2^{n+i-6}\alpha_3^3+\dots+
	  b_{2i-1}\alpha_2^{n-3i+2}\alpha_3^{2i-1}\nonumber\\
	  &+\alpha_2^{n-3i}\alpha_3^{2i}+\text{ lower order terms}\label{lin-comb}
	\end{align}
	All the monomials appearing in the first line of the right hand side are of the form 
	$\alpha_2^i\alpha_3^j$ with $2i+3j\geq 2n$ and $j$ being odd. In particular, using what we just 
	proved for the monomials with odd powers of $\alpha_3$, we can find constants 
	$c_0'$, $\dots$, $c_{i-1}'$ such that 
	\begin{equation}\label{2-lin-comb}
	  {\rm LT}\left(\sum_{j=0}^{i-1} c_j'\alpha_1^{j}\overline \sigma_{n+i-j}-
	  \sum_{k=1}^{i}b_{2k-1}\alpha_2^{n+i-4k+2}\alpha_3^{2k-1}\right)=
	  k\alpha_2^{n-3i-2}\alpha_3^{2i+1}
	\end{equation}
	for some constant $k$. Using \eqref{lin-comb} and \eqref{2-lin-comb}, we have
	\[
	   {\rm LT}\left(d_i\alpha_2^i\zeta_n+\sum_{j=0}^{i-1} d_j\alpha_2^{j}\sigma_{n+i-j}
	   -\sum_{j=0}^{i-1} c_j'\alpha_1^{j}\overline \sigma_{n+i-j}\right)=\alpha_2^{n-3i}\alpha_3^{2i}.	
	   \qedhere
	\]
\end{proof}

Define $\overline I_g$ as the image of the ideal $\widetilde I_g$ with respect to the homomorphism
\begin{equation}\label{alg-hom}
  \bA^3_g  \to \C[\alpha_2,\alpha_3,\beta_2]
\end{equation}
given by mapping $\beta_3$ and the curve classes $\psi_r^i$ to $0$. Since $\widetilde I_g$ includes $\beta_3$ and the curve classes, we have $\bA^3_g / \widetilde I_g\cong \C[\alpha_2,\alpha_3,\beta_2]/\overline I_g$. Using \eqref{eq:zetarelbound1}--\eqref{eq:zetarelbound2}, we have
\begin{gather}
	\zeta_m^{g,k}\in \widetilde I_g \quad \text{ if }\quad 0\leq k \leq g,\quad m \geq  3g-k-1, 
	\label{rel-range}\\[2mm]
	\overline{\zeta}_m^{g,k}\in \widetilde I_g \quad \text{ if }\quad0\leq k \leq g,\quad m \geq 3g-k.
	\label{rel-range-bar}
\end{gather}
In the following proof, we slightly abuse notation and regard $\zeta_m^{g,k}$ and $\overline{\zeta}_m^{g,k}$ as elements of $\C[\alpha_2,\alpha_3,\beta_2]$ using the homomorphism \eqref{alg-hom}.

\begin{prop}\label{LT-I-g-bar}
	The leading term ideal of $\overline I_g$ includes the monomials 
	\begin{equation}\label{LT-gen}
	  \{\alpha_2^i\alpha_3^j\beta_2^k\mid k\leq 2,\, 2i+3j+2k\geq 4g-2 \} \cup\{\beta_2^3\}.
	\end{equation}
\end{prop}

\begin{proof}
	It follows from the definition of $\zeta_m$ and $\overline \zeta_m$ that 
	\begin{equation}\label{rel-zeta}
	  \zeta_m^{g,k}(-\alpha_2,-\alpha_3,0)=\zeta_m \quad \text{and} \quad
	  \overline \zeta_m^{g,k}(-\alpha_2,-\alpha_3,0)=\overline \zeta_m.
	\end{equation}  
	Thus \eqref{rel-range} and \eqref{rel-range-bar} with $k=g$ imply that 
	$\zeta_{2g-1}$, $\zeta_{2g}$, 
	$\overline \zeta_{2g}$ and $\overline \zeta_{2g+1}$ belong to $\overline I_g+(\beta_2)$.
	It follows from Proposition \ref{zeta-n-leading-term-ideal} that $\alpha_2^i\alpha_3^j$
	with $2i+3j\geq 4g-2$ is in the leading term ideal of $\overline I_g$.
	From Lemma \ref{beta-2-rela} and \eqref{rel-zeta}, we see that 
	\begin{align}
		&\beta_2\zeta_{m-2} \in (\zeta_m^{g,g}, \zeta_m^{g,g-1},\beta_2^2)\subset 
			\overline I_g+(\beta_2^2),\quad \text{ if }\quad m \geq  2g,\label{eq:b21}\\[0.35cm] 
		& \beta_2\overline{\zeta}_{m-2}\in (\overline{\zeta}_m^{g,g}, 
			\overline{\zeta}_m^{g,g-1},\beta_2^2)\subset 
			\overline I_g+(\beta_2^2),\quad \text{ if }\quad m \geq  2g+1. \label{eq:b22}
	\end{align}
	Therefore, another application of Proposition \ref{zeta-n-leading-term-ideal} implies that 
	$\alpha_2^i\alpha_3^j\beta_2$
	with $2i+3j\geq 4g-4$ is in the leading term ideal of $\overline I_g$. Finally, using 
	Lemma \ref{beta-2-2-rela} and \eqref{rel-zeta}, we see that 
	\begin{align}
		&\beta_2^2 \zeta_{m-1} \in (\zeta_{m+1}^{g,g}, \zeta_{m+2}^{g,g},\zeta_{m+3}^{g,g},
		\beta_2^3)\subset 
			\overline I_g,\quad \text{ if }\quad m \geq  2g-2,\label{eq:b221}\\[0.35cm] 
		& \beta_2^2 \overline{\zeta}_{m-1} \in (\overline{\zeta}_{m+1}^{g,g}, 
		\overline{\zeta}_{m+2}^{g,g},\overline{\zeta}_{m+3}^{g,g},\beta_2^3)\subset 
		\overline I_g,\quad \text{ if }\quad m \geq  2g-1 \label{eq:b222}, 
	\end{align}
	which shows that $\beta_2^2\zeta_{2g-3}$, $\beta_2^2\zeta_{2g-2}$, 
	$\beta_2^2\overline \zeta_{2g-2}$ and $\beta_2^2\overline \zeta_{2g-1}$ belong to 
	$\overline I_g$. Appealing to Proposition \ref{zeta-n-leading-term-ideal} again,
	we conclude that $\alpha_2^i\alpha_3^j\beta_2^2$
	with $2i+3j\geq 4g-6$ is in the leading term ideal of $\overline I_g$. Finally $\beta_2^3$ is 
	in the leading term ideal of $\overline I_g$ because it is an element of $\overline I_g$.
\end{proof}

Now we are almost ready to prove Theorem \ref{thm:undeformeddimension}. We only need the following lemma, which can be proved in a straightforward way by induction. 
\begin{lemma}\label{lattice-pt-count}
	For any non-negative integer $n$, let $f(n)$ denote the size of the following set:
	\[
	  S_n:=\{(i,j)\in \bZ_{\geq 0}\times \bZ_{\geq 0}\mid 2i+3j<n\}.
	\]
	If $n=6k+r$ with $0\leq r\leq 5$, then
	\[
	  f(n)=\frac{n^2-r^2}{12}+2\left \lfloor \frac{n}{6}\right\rfloor+f(r).
	\]
	Moreover, $f(0)=0$, $f(1)=1$, and $f(r)=r-1$ if we have $2\leq r \leq 5$.
\end{lemma}

\begin{proof}[Proof of Theorem \ref{thm:undeformeddimension}]
	By the isomorphism 
	$\bA^3_g / \widetilde I_g\cong \C[\alpha_2,\alpha_3,\beta_2]/\overline I_g$,
	it suffices to show that the cardinality of the set of monomials not in the leading term ideal of $\overline I_g$ is bounded above by $(2g-1)^2$. Lemma \ref{LT-I-g-bar} implies
	this cardinality is bounded above by the size of:
	\begin{equation}\label{eq:cardinalityofmonomials}
	  \{(i,j,k)\in \bZ_{\geq 0}\times \bZ_{\geq 0}\times \bZ_{\geq 0} \mid
	   2i+3j+2k<4g-2,\, k\leq 2\}.
	\end{equation}
	In the terminology of Lemma \ref{lattice-pt-count}, the size of \eqref{eq:cardinalityofmonomials} is 
	$f(4g-2)+f(4g-4)+f(4g-6)$. Since the set of mod $6$ remainders of $4g-6, 4g-4, 4g-2$ is $\{0,2,4\}$, Lemma \ref{lattice-pt-count} implies
	\begin{align*}
		f(4g-2)+f(4g-4)+f(4g-6)&=
		\frac{(4g-2)^2}{12}+\frac{(4g-4)^2}{12}+\frac{(4g-6)^2}{12}\\[0.5cm]
		&+2\left(
		\left \lfloor \frac{2g-3}{3}\right\rfloor+\left \lfloor \frac{2g-2}{3}\right\rfloor+
		\left \lfloor \frac{2g-1}{3}\right\rfloor\right)+\frac{7}{3}\\[0.5cm]
		& =(2g-1)^2. \qedhere
	\end{align*}
\end{proof}


\section{Sutured instanton homology}\label{sec:sutured}

In this section, we study sutured instanton homology for the gauge group $U(3)$. The construction is a slight variation of the one in \cite{DX}. Building on the work of that reference, we prove Theorem \ref{thm:intro-SHI-invariance}, which says that sutured instanton homology is well-defined, independent of the auxiliary choices in its construction. Following the strategy of \cite{km-sutures}, we prove Theorem \ref{thm:intro-SHI}, a non-vanishing result for {\it taut} sutured manifolds. This is used to prove our topological applications, Theorems \ref{thm:intro1} and \ref{thm:intro2}.

Suppose $(Y,\gamma)$ is a 3-admissible pair and $R$ is a connected surface of genus $g$ such that $R\cdot \gamma \equiv 1\pmod{3}$. For a point $y$ in $Y$, we generalize the notation from \eqref{eq:s1sigmamaps} and set
\[
	\beta_r=\mu_r(y), \qquad r\in \{2,3\},
\]
viewed as an endomorphism of $I_\ast^3(Y,\gamma)$. Analogous to \eqref{epsilon-product}, we can define an operator $\varepsilon(R)$ of degree $-4$ acting on $I_\ast^3(Y,\gamma)$ using the product cobordism $[-1,1]\times Y$ with $U(3)$-bundle determined by the oriented $2$-cycle $[-1,1]\times \gamma \cup \{0\}\times R$. When the choice of $R$ is clear from context, we write $\epsilon$ for $\epsilon(R)$. We also have the operators $\mu_2(R)$ and $\mu_3(R)$ acting on $I_\ast^3(Y,\gamma)$, which can be defined without the assumption $R\cdot \gamma \equiv 1\pmod{3}$. All of these operators commute, and hence we can consider their simultaneous (generalized) eigenspaces. 

\begin{prop}\label{eigen-value-gen-mfld}
	Suppose $(\lambda_2,\lambda_3,\eta_2,\eta_3)$ is a simultaneous eigenvalue of 
	the operators $(\mu_2(R), \mu_3(R), \beta_2,\beta_3)$ acting on $I_\ast^3(Y,\gamma)$ with $\eta_2^3=27$ and $\eta_3=0$. 
	Then $(\lambda_2,\lambda_3,\eta_2,\eta_3) \in \mathcal{E}_{g,1}^3$.
	Moreover, the generalized eigenspace for any eigenvalue of the form 
	$(\pm \sqrt{3}\zeta^k(2g-2), 0, 3\zeta^{2k}, 0)$ agrees with the corresponding eigenspace. 
\end{prop}
\begin{proof}
	This proposition is the $U(3)$ analogue of \cite[Corollary 7.2]{km-sutures} and can be verified in a similar way. 
	We use functoriality to see that any relation among
	$\alpha_r$ and $\beta_r$ in $V_{g,1}^3$ holds universally for any admissible pair $(Y,\gamma)$ and an embedded surface 
	$R$ as above. To be more precise, let $p$ be a polynomial with $4$ variables such 
	that $p(\alpha_2,\alpha_3,\beta_2,\beta_3)$ vanishes as an operator acting on $V_{g,1}^3$. Then we show that 
	$p(\mu_2(R), \mu_3(R), \beta_2,\beta_3)$ vanishes as an operator acting on $I_\ast^3(Y,\gamma)$. This is 
	sufficient to prove both claims in the statement of the proposition because they can be expressed in terms of  polynomial
	 relations among the operators $\mu_2(R)$, $\mu_3(R)$, $\beta_2$ and $\beta_3$, and then we can use the corresponding 
	 results in the special case of $V_{g,1}^3$ given in Theorem \ref{thm:mainev} and Proposition \ref{prop:dim1}.

	A regular neighborhood of $\{0\}\times R$ in the product cobordism $[-1,1]\times Y$
	can be used to decompose $[-1,1]\times Y$ as the composition of cobordisms $D^2\times R$ and $W$ with 
	three boundary components 
	$-Y$, $Y$ and $S^1\times R$. This also induces a decomposition of $\gamma$ where the intersection with 
	$D^2\times R$ can be assumed to be $D^2\times \{x\}$ for $x\in R$. Suppose also that $w$ is the induced 2-cycle on $W$.
	Then functoriality implies that for any polynomial $p$ of $4$ variables and any $v\in I_\ast^3(Y,\gamma)$
	we have
	\[
	  p(\mu_2(R), \mu_3(R), \beta_2,\beta_3)(v)=I^3_*(W,w)(v \otimes p(\alpha_2,\alpha_3,\beta_2,\beta_3)(\mathbf{1})).
	\]
	In particular, if $p(\alpha_2,\alpha_3,\beta_2,\beta_3)$ is a trivial operator acting on $V_{g,1}^3$, then the action of 
	$p(\mu_2(R), \mu_3(R), \beta_2,\beta_3)$ on $I_\ast^3(Y,\gamma)$ is trivial.
\end{proof}

We define the instanton Floer homology group $I_*^3(Y,\gamma\vert R)$ as a simultaneous eigenspace for the point classes and the operators associated to the surface $R$ in the following way: 
\begin{equation}\label{surface-instanton}
  I_*^3(Y,\gamma\vert R)=\ker(\mu_2(R)-\sqrt{3} (2g-2))\cap 
  \ker(\mu_3(R))\cap \ker(\beta_2-3)\cap \ker(\beta_3).
\end{equation}
In particular, equation \eqref{eq:dx2} implies that 
\begin{equation}\label{surface-instanton-special}
  I_*^3(S^1\times \Sigma_g, S^1\times \{x\} \vert \Sigma_g)=\bC.
\end{equation}
If $R'$ is disconnected, we modify \eqref{surface-instanton} so that the intersection includes each of the operators $\mu_2(R')-\sqrt{3} (2g(R')-2)$ and $\mu_3(R')$ for each connected component $R'$ of $R$.  In the case that $(Y,\gamma)$ is the disjoint union of admissible pairs $(Y_0,\gamma_1)$, $(Y_1,\gamma_1)$ and $R\subset Y$ is given by $R_0 \sqcup R_1$ with $R_i\subset Y_i$, $R_i \cdot \gamma_i \equiv 1$ mod $3$, then we define  
\[
  I_*^3(Y,\gamma\vert R)=I_*^3(Y_0,\gamma_1\vert R_0)\otimes I_*^3(Y_1,\gamma_1\vert R_1).
\]
This can be extended to more than two connected components in the same way.

\begin{remark}
	In \cite{DX}, the instanton homology group $I_*^3(Y,\gamma\vert R)$ is defined by taking the 	
	simultaneous generalized kernel of the operators in \eqref{surface-instanton} 
	and the operator $\epsilon -1$. Proposition \ref{eigen-value-gen-mfld} shows that the 
	generalized kernel for the operators in \eqref{surface-instanton} agrees with the ordinary kernel.
	Furthermore, we show in the proof of Proposition \ref{prop:dim1} that in the case of 
	$(S^1\times \Sigma_g, S^1\times \{x\})$ any element in \eqref{surface-instanton} already 
	belongs to the kernel of $\epsilon-1$. Therefore, the proof of 
	Proposition \ref{eigen-value-gen-mfld} shows that the same claim holds for an arbitrary 
	pair $(Y,\gamma)$. As a consequence of these observations, our definition in 
	\eqref{surface-instanton} agrees with that of \cite{DX}.
\end{remark}

\begin{prop}\label{eigenvalue-S}
	Suppose $S$ is an embedded surface in $Y$. Then the operators $\mu_2(S)$ and 
	$\mu_3(S)$ preserve the subspace $I_*^3(Y,\gamma\vert R)$ of $I_*^3(Y,\gamma)$. 
	Furthermore, if 
	$(\lambda_2,\lambda_3)$ is a simultaneous eigenvalue of $(\mu_2(S),\mu_3(S))$, then there are $a,b\in \Z$ 
	with $a\equiv b\pmod{2}$ such that $(\lambda_2,\lambda_3)=(\sqrt3 a,\sqrt 3 ib)$
	and 
	\begin{equation}\label{str-ineq}
	  |a|+|b| \leq 2g(S)-2.
	\end{equation}
\end{prop}

This proposition is the counterpart of \cite[Proposition 7.5]{km-sutures}. However, the proof there seems to require some modifications, even in the case $N=2$. The modification used in the following proof was communicated to us by Peter Kronheimer.
\begin{proof}
	In the case that $S\cdot \gamma \equiv 1\pmod{3}$, the claim follows from Proposition \ref{eigen-value-gen-mfld} 
	and the case $S\cdot \gamma \equiv -1\pmod{3}$ can be verified in a similar way. Using a topological trick, the case 
	$S\cdot \gamma \equiv 0\pmod{3}$ can be also reduced to the previous cases. Suppose $v\in I_*^3(Y,\gamma\vert R)$
	is a simultaneous eigenvector of $(\mu_2(S),\mu_3(S))$ with eigenvalues $(\lambda_2,\lambda_3)$.
	Suppose $\sigma_n$ is the homology class $n[S]+[R]$, which is represented by a connected surface $S_n$ in $Y$.
	Since the $\mu$ operators depend only on the homology classes of the involved surfaces, $v$
	is a simultaneous eigenvector of $(\mu_2(S_n),\mu_3(S_n))$ with eigenvalues $(n\lambda_2+\sqrt{3} (2g-2),n\lambda_3)$.
	We have $S_n\cdot \gamma \equiv 1\pmod{3}$, which in the case that $n=1$ implies that 
	$(\lambda_2,\lambda_3)=(\sqrt3 a,\sqrt 3 ib)$ for some integers $a$ and $b$ with the same parity. 
	
	Next, to show that $(a,b)$ satisfies \eqref{str-ineq}, we need some control on the genus of the connected surface $S_n$. 
	In fact, it suffices to find an embedded surface $S_n$ in $[-1,1]\times Y$ with the same homology class. 
	Take a cyclic $n$-sheeted covering 
	$\widetilde S$ of $S$. It is straightforward to see that $\widetilde S$ can be embedded in $D^2\times S$ in such a way that 
	the composition of this embedding with the projection map $D^2\times S \to S$ is the covering projection $\widetilde S \to S$.
	In particular, the genus of $\widetilde S$ is equal to $n(g(S)-1)+1$. The embedding of $\widetilde S$ 
	in $D^2\times S$ induces an embedding of this surface in a neighborhood of $\{0\}\times S\subset [-1,1]\times Y$ realizing 
	the homology class $n[S]$.  
	By tubing this surface and a disjoint copy of $R$, we obtain a connected surface $S_n$ of genus $n(g(S)-1)+g+1$ with
	the homology class $n[S]+[R]$. Since $S_n\cdot \gamma \equiv 1\pmod{3}$ (and the self-intersection of $S_n$ is trivial), 
	we have
	\[
	  |na+2g-2|+|nb|\leq 2n(g(S)-1)+2g.
	\]
	Diving by $n$ and taking $n\to \infty$ gives \eqref{str-ineq}.
\end{proof}

\begin{remark}\label{genus-one}
	For a genus one surface $T$, the group $I_*^3(S^1\times T,S^1\times \{x\})$ is 3-dimensional and hence it splits as the sum of 1-dimensional eigenspaces for the three simultaneous 
	eigenvalues in $\mathcal E^3_{1,1}$. In particular, the actions of $\mu_2(T)$ and $\mu_3(T)$ 
	are trivial on $I_*^3(S^1\times T,S^1\times \{x\})$.  Using a similar argument as in the proof of
	Proposition \ref{eigen-value-gen-mfld}, we can see more generally that if $(Y,\gamma)$ is an 
	admissible pair 
	and $T$ is an embedded surface of genus $1$ in $Y$ with $\gamma \cdot T\equiv 1$ mod $3$,
	then the actions of $\mu_2(T)$ and $\mu_3(T)$ are trivial. 
	In particular, we have
	\[
	  I_*^3(Y,\gamma\vert T)= \ker(\beta_2-3)\cap \ker(\beta_3).  
	\]
\end{remark}

Similar to \cite[Corollary 7.6]{km-sutures}, we consider the action of $(\mu_2(\sigma),\mu_3(\sigma))$ on $I_*^3(Y,\gamma\vert R)$ for all homology classes $\sigma \in H_2(Y;\Z)$  to obtain a splitting of $I_*^3(Y,\gamma\vert R)$ as
\begin{equation}\label{es-decomp}
  I_*^3(Y,\gamma\vert R)=\bigoplus_s I_*^3(Y,\gamma\vert R;s)
\end{equation}
where the direct sum is over all homomorphisms 
\[
  s:H_2(Y;\Z) \to \Gamma \subset \Z\oplus \Z
\]
with $\Gamma$ being the sublattice of $\Z\oplus \Z$ given by pairs $(a,b)$ with $a\equiv b\pmod{2}$. For $s=(s_2,s_3)$ as above, the summand $I_*^3(Y,\gamma\vert R;s)$ is given as
\[
  \bigcap_{\sigma \in H_2(Y;\Z)} \bigcup_{N\geq 0}\(\ker\(\mu_2(\sigma)-\sqrt3 s_2(\sigma)\)^N\cap \ker\(\mu_3(\sigma)-\sqrt3  i s_3(\sigma)\)^N \).
\]
As a corollary of Proposition \ref{eigenvalue-S},
for any $\sigma\in H_2(Y;\Z)$ with a surface representative $S$ of genus $g$, the summand $I_*^3(Y,\gamma\vert R;s)$ can be non-trivial only if
\[
  |\!|s(\sigma)|\!|_1\leq 2g(S)-2.
\]
Here $|\!|\cdot |\!|_1$ denotes the $L^1$ norm of vectors in $\bR^2$.

The sutured instanton Floer homology group $SHI_*^3(M,\alpha)$ is defined with the aid of the instanton Floer homology groups in \eqref{surface-instanton} for any {\it balanced} sutured manifold $(M,\alpha)$. Following \cite{Gab:fol-sut,juhasz}, a balanced sutured manifold $(M,\alpha)$ consists of an oriented 3-manifold $M$ without any closed component and a collection of oriented simple closed curves $\alpha$ in the boundary of $M$. The boundary of $M$ is decomposed into three parts
\[
  \partial M=A(\alpha)\cup R_+(\alpha) \cup R_-(\alpha),
\]
where $A(\alpha)$ is the closure of a tubular neighborhood of $\alpha$. The connected components of $\partial M\setminus A(\alpha)$ are oriented, and $R_+(\alpha)$ (resp. $R_-(\alpha)$) is the union of such connected components whose orientation is given by the outward-normal-first convention (resp. inward-normal-first convention). The 2-dimensional manifolds $R_{\pm}(\alpha)$ do not have any closed connected component and the induced orientation on any of their boundary components (using outward-normal-first convention) agrees with the orientation of the corresponding suture. (Note that this condition fixes the orientation of the connected components of $\partial M\setminus A(\alpha)$.) Finally we require that $\chi(R_+(\alpha))=\chi(R_-(\alpha))$.

\begin{example}\label{prod-sutured-mfld} (Product sutured manifolds) Let $F_{g,k}$ denote the oriented surface of genus $g$ with $k\geq 1$ 
	boundary components. Then $M=[-1,1] \times F_{g,k}$ and $\alpha=\{0\}\times \partial F_{g,k}$ give a balanced sutured manifold with 
	$R_{\pm}(\alpha)=\{\pm 1\}\times F_{g,k}$ and $A(\alpha)=[-1,1] \times \partial F_{g,k}$.
\end{example}

\begin{example}\label{knot-3-man-sutured-mfld}
	Any closed oriented 3-manifold $Y$ with a basepoint can be used to produce a sutured 
	manifold $(Y(1),\alpha(Y))$, where $Y(1)$ is the complement of a 
	ball neighborhood of the basepoint in $Y$ and $\alpha(Y)$ is a simple closed curve in 
	the boundary of $Y(1)$. Any knot $K$ in a 3-manifold $Y$ can be used 
	to produce a sutured manifold $(Y(K),\alpha(K))$ where $Y(K)$ is the exterior of $K$ and 
	$\alpha(K)$ consists of two meridional simple closed curves. 
\end{example}

The closure of a balanced sutured manifold $(M,\alpha)$ is a closed 3-manifold $Z_\alpha$, defined in the following way. Suppose the number of sutures is equal to $k$, and consider the product sutured manifold $[-1,1]\times F_{g,k}$ for an arbitrary $g$. Gluing the neighborhood $A(\alpha)$ of the sutures in $\partial M$ to $[-1,1]\times \partial F_{g,k}$ determines a 3-manifold $Z^0_{\alpha}$ with two boundary components $\overline R_+$ and $\overline R_-$. The surface $\overline R_\pm$ is the union of $R_\pm$ and $\{\pm1\}\times F_{g,k}$. Since $(M,\alpha)$ is balanced, $\overline R_+$ and $\overline R_-$ are connected oriented surfaces of the same genus. We pick an orientation-preserving diffeomorphism $\varphi: \overline R_+ \to \overline R_-$ to identify these two boundary components, obtaining the closure $Z_\alpha$. 

The surfaces $\overline R_{\pm}$ determine a closed surface $\overline R\subset Z_\alpha$. We require that there is a simple closed curve $c$ in $F_{g,k}$, that gives rise to non-separating curves in $\overline R_\pm$ and the gluing map $\varphi$ maps these curves to each other. (This can always be arranged, for example, by taking $g\geq 1$ and setting $c$ to be a non-separating oriented simple close curve in $F_{g,k}$. ) The curve $c$ determines a non-separating closed curve in $\overline R$, which is still denoted by $c$. In particular, we may fix another oriented simple closed curve $c'$ in $\overline R$ intersecting $c$ transversely at one point. By fixing a basepoint $x\in F_{g,k}$ and demanding that $\varphi(x)=x$, we obtain a curve $\gamma\subset Z_\alpha$ from $[-1,1]\times \{x\}\subset Z_\alpha^0$. The $U(3)$ sutured instanton homology of $(M,\alpha)$ is defined as 
\begin{equation*}\label{eq:suturedgroupdef}
  SHI_*^3(M,\alpha):=I_*^3(Z_\alpha,\gamma\vert \overline R).
\end{equation*}
We now prove Theorem \ref{thm:intro-SHI-invariance}, which says that this sutured homology group is an invariant of $(M,\alpha)$, i.e. it does not depend on the choice of $g$ nor the gluing map $\varphi$.

\begin{proof}[Proof of Theorem \ref{thm:intro-SHI-invariance}]
	A version of excision for instanton Floer homology groups $I_*^3(Y,\gamma\vert R)$ is proved in \cite[Theorem 5.16]{DX}, and is used to show 
	that $SHI_*^3(M,\alpha)$ is independent of the gluing map $\varphi$. 
	Using the excision theorem in \cite{DX} and Theorem \ref{thm:mainev}, we show 
	independence from $g$ following the argument in \cite{km-sutures}. 
	This requires a further understanding of the instanton homology of $S^1\times \Sigma_g$ 
	for different $U(3)$ bundles over this manifold. In the following, let $c_0$ and $c_0'$ be 
	non-separating oriented simple closed curves in $\Sigma_g$ that have exactly one 
	transversal intersection point. 
	The curve $c_0$ determines the 2-dimensional torus $T=S^1\times c_0$
	in $S^1\times \Sigma_g$. By fixing a basepoint in $S^1$, we may regard $c_0'$ as a 
	1-cycle in $S^1\times \Sigma_g$. We also write $\gamma_1$ for the 
	$1$-cycle $S^1\times \{x\}$ in $\Sigma_g$.
	
	First consider the instanton Floer homology group 
	$B:=I_*^3(S^1\times \Sigma_g,\gamma_1+c_0'\vert \Sigma_g)$. 
	Applying the excision result of \cite[Theorem 5.16]{DX} twice in the same 
	way as in the proof of \cite[Proposition 7.8]{km-sutures}, we obtain an isomorphism
	\begin{equation}\label{eq:B-iso}
		B\otimes B\otimes B \cong I_*^3(S^1\times \Sigma_g,\gamma_1+3c_0'\vert \Sigma_g).
	\end{equation}
	As the Floer groups $I_*^3(Y,\gamma)$ depend only on the element of $H^2(Y;\Z/3)$ induced by $\gamma$, the right side of 
	\eqref{eq:B-iso} is isomorphic to $\bC$ by \eqref{surface-instanton-special}. Therefore, $B$ is also 1-dimensional.  
	
	Next, we consider the instanton Floer homology group $I_*^3(S^1\times \Sigma_g,c_0'\vert T)$. The genus 
	one version of the excision theorem of \cite[Theorem 5.16]{DX} implies that 
	\begin{equation}\label{iso-excision-1}
	  I_*^3(S^1\times \Sigma_g,c_0'\vert T)\otimes I_*^3(S^1\times \Sigma_1,\gamma_1+c_0'\vert T) \cong I_*^3(S^1\times \Sigma_g,\gamma_1+c_0'\vert T).
	\end{equation}
	The excision isomorphism intertwines the action of $\mu_i(\Sigma_g)+\mu_i(\Sigma_1)$
	on the left hand side of \eqref{iso-excision-1} and the action of $\mu_i(\Sigma_g)$ on the right hand side. This follows from the fact that the 
	excision isomorphism is given by a homomorphism associated to a cobordism 
	\[
	  W:S^1\times \Sigma_g \sqcup S^1\times \Sigma_1\to S^1\times \Sigma_g
	\]  
	and the homology class $[\Sigma_g]+[\Sigma_1]$ induced from the incoming end and $[\Sigma_g]$ from the outgoing end are homologous on $W$. According to Remark \ref{genus-one}, the action of $\mu_i(\Sigma_1)$ is trivial 
	and hence $ I_*^3(S^1\times \Sigma_g,c_0'\vert T)$ and 
	$I_*^3(S^1\times \Sigma_g,\gamma_1+c_0'\vert T)$ are isomorphic as modules over 
	$\Q[\mu_2(\Sigma_g),\mu_3(\Sigma_g)]$. In particular, this 
	shows that the simultaneous eigenvalues of the operators $(\mu_2(\Sigma_g),\mu_3(\Sigma_g))$ acting on $I_*^3(S^1\times \Sigma_g,c_0'\vert T)$ are of the form $(\sqrt3 a,\sqrt 3 ib)$ with $ |a|+|b| \leq 2g-2$ and 
	the $(\sqrt{3} (2g-2),0)$-eigenspace is $1$-dimensional. 
	
	Now, let $Z_\alpha$ be a closure of $(M,\alpha)$ given by the 
	surface $F_{g,k}$ and a gluing map $\varphi$. 
	Replacing $F_{g,k}$ with $F_{g+1,k}$ and stabilizing $\varphi$ in the obvious
	way determines a different closure $Z_\alpha'$. We also write $\overline R$ and $\overline R'$ for 
	the distinguished surfaces in $Z_\alpha$ and $Z_\alpha'$ whose 
	genera are related by $g(\overline R')=g(\overline R)+1$.
	These closed curves $c$ and $c'$ in $\overline R$ determine two oriented simple closed curves 
	in $\overline R'$ which we still denote by $c$ and $c'$.
	
	To prove our claim, we need to show that 
	\begin{equation}\label{genus-stab}
	  	I_*^3(Z_\alpha,\gamma\vert \overline R)\cong I_*^3(Z_\alpha',\gamma\vert \overline R').
	\end{equation}
	By applying the excision result in  \cite[Theorem 5.16]{DX} for the copies of the 
	surface $\overline R$ in the two admissible pairs $(Z_\alpha,\gamma)$, $(S^1\times \overline R,\gamma_1+c')$ and using the $1$-dimensionality of the latter vector space, we conclude that 
	\[
	  I_*^3(Z_\alpha,\gamma\vert \overline R) \cong I_*^3(Z_\alpha,\gamma+c'\vert \overline R).
	\]
	Thus, to show \eqref{genus-stab}, it suffices to verify that
	\begin{equation}\label{genus-stab-diff-w}
	  	I_*^3(Z_\alpha,\gamma+c'\vert \overline R)\cong I_*^3(Z_\alpha',\gamma+c'\vert \overline R').
	\end{equation}	
	By our assumption on $c$ and the gluing map $\varphi$, there are copies of $T=S^1\times c$ in $Z_\alpha$ and $Z_\alpha'$. Another application of \cite[Theorem 5.16]{DX} similar to \eqref{iso-excision-1} implies that
	\[
	  I_*^3(S^1\times \Sigma_2,c_0'\vert T) \otimes I_*^3(Z_\alpha,\gamma+c'\vert T) \cong I_*^3(Z_\alpha',\gamma+c'\vert T),
	\]
	and this isomorphism intertwines the action of $\mu_i(\Sigma_2)+\mu_i(\overline R)$
	on the left hand side and the action of $\mu_i(\overline R')$ on the right hand side. Combining this fact, Remark \ref{genus-one}, Proposition \ref{eigenvalue-S} and our analysis of the instanton Floer homology group 
	$I_*^3(S^1\times \Sigma_2,c_0'\vert T)$ verifies the claimed isomorphism in 
	\eqref{genus-stab-diff-w}.
\end{proof}

From the proof of Theorem \ref{thm:intro-SHI-invariance} one can see that the above construction can be generalized in various directions. First, one can consider non-trivial $U(3)$ bundles on sutured manifolds. More precisely, let $(M,\alpha)$ be a sutured manifold and $w$ be a properly embedded oriented curve in $M$ such that $w$ is disjoint from $A(\alpha)$ and the intersection of $w$ with $R_{\pm}(\alpha)$ is a collection of points
$\pi_{\pm}=\{p_1^{\pm},\dots,p_k^{\pm}\}$ such that the intersection of $w$ with $R_{\pm}(\alpha)$ at the points $p_i^{\pm}$ have the same sign. In forming the closure $Z_\alpha$ of $(M,\alpha)$, we require that the gluing map sends the point $p_i^{+}$ to $p_i^{-}$. Thus we obtain a closed oriented curve $\overline w$. We define the sutured instanton homology $SHI_*^3(M,\alpha)_w$ of $(M,\alpha,w)$ as the Floer homology group $I_*^3(Z_\alpha,c'+\overline w\vert \overline R)$. In particular, the instanton Floer homology of a product sutured manifold for any choice of $w$ is still 1-dimensional. 

Following the same proof as that of Theorem \ref{thm:intro-SHI-invariance}, we see that this Floer homology group is independent of the specific choice of the closure and is also isomorphic to  the instanton homology groups $I_*^3(Z_\alpha,d\cdot \gamma+\overline w\vert \overline R)$ and $I_*^3(Z_\alpha,d\cdot \gamma+c'+\overline w\vert \overline R)$  where for the former instanton homology group we need that $d+\overline w\cdot \overline R \nequiv 0$ mod $3$. It is straightforward to see that the isomorphism class of $SHI_*^3(M,\alpha)_w$ depends only on the homeomorphism type of $(M,\alpha)$ and the isomorphism type of the $U(3)$-bundle on $M$ determined by $w$. Furthermore, for any $U(3)$-bundle on $M$ one can arrange $w$ satisfying the above requirements. 

We can also see from the proof of Theorem \ref{thm:intro-SHI-invariance} that to form the closure of a sutured manifold $(M,\alpha)$ we do not necessarily need to use a connected sutured manifold $[-1,1]\times F_{g,k}$. We can use a product sutured manifold $[-1,1]\times F$ as long as $\overline R$ is connected and each connected component of $F$ has a simple closed curve that becomes a non-separating curve in $\overline R$. This flexibility in forming the closure will be useful below. Finally, as another consequence of Theorem \ref{thm:intro-SHI-invariance}, we make the following observation. 

\begin{lemma}\label{prod-1-handle-glue}
	Gluing a product 1-handle to a sutured manifold $(M,\alpha,w)$ along its sutures 
	does not change the isomorphism type of 
	$SHI_*^3(M,\alpha)_w$.
\end{lemma}
A product 1-handle is the product $[-1,1]\times H$ where $H$ is the 2-dimensional $1$-handle given as $I\times I$ for an interval $I$. Fixing an embedding of $I\times \partial I$ into $\alpha$ determines an embedding of $[-1,1]\times I\times \partial I$ into $A(\alpha)$. Now to glue the product 1-handle $[-1,1]\times H$ to $(M,\alpha,w)$, we identify part of the boundary of the product 1-handle given by $[-1,1]\times I\times \partial I$ with its image in $A(\alpha)$ via the embedding. We assume that this gluing is done in a way that the resulting 3-manifold is orientable. With this assumption, the resulting 3-manifold admits the structure of a sutured manifold in an obvious way.

\begin{proof}
	Suppose $(M',\alpha',w)$ is obtained by gluing a product 1-handle to $(M,\alpha,w)$. 
	A closure of $(M',\alpha',w)$, obtained by
	gluing the product sutured manifold $[-1,1]\times F_{g,d}$ to $M'$, can be regarded as a 
	closure of $(M,\alpha)$, where we use
	the product sutured manifold  $[-1,1]\times (F_{g,d}\cup H)$ in forming the closure. 
	From this one can easily see that 
	the sutured instanton homologies of $(M,\alpha,w)$ and $(M',\alpha',w')$ are isomorphic to each other.
\end{proof}

The operation of {\it surface decomposition} can be used to simplify sutured manifolds \cite{Gab:fol-sut}. A {\it decomposing surface $S$} in a balanced sutured manifold $(M,\alpha)$ is a properly oriented surface $S$ in $M$ such that any connected component of $\partial S\cap A(\gamma)$ is either a properly embedded non-separating arc in $A(\gamma)$ or a simple closed curve oriented in the same sense as the suture in the corresponding connected component of $A(\gamma)$. Removing a small tubular neighborhood $N(S)$ of $S$ from $M$ produces a new sutured manifold $(M',\alpha')$ with 
\begin{align*}
  A(\alpha')&=\(A(\alpha)\cap \partial M'\) \cup N_{\partial M'}(S_+\cap R_-(\alpha))\cup N_{\partial M'}(S_-\cap R_+(\alpha)),
  \nonumber\\[2mm]
  R_{\pm}(\alpha')&=\(R_{\pm}(\alpha)\cap M'\) \cup S_{\pm} \setminus {\rm int}(A(\alpha'))\nonumber
\end{align*}
where, after identifying $N(S)$ with $[-1,1]\times S$ as an oriented 3-manifold, $S_{\pm}$ is given by $\{\mp 1\}\times S\subset \partial N(S)\cap M'$. This operation of surface decomposition is usually denoted
\[
  (M,\alpha)\stackrel{S}{\rightsquigarrow}(M',\alpha').
\]
We may extend this definition in an obvious way in the presence of non-trivial bundle data $w$. If $w$ is a properly oriented simple closed curve intersecting $S$ and its boundary transversely, then the intersection $w'$ of $w$ with $M'$ determines a properly embedded oriented curve in $M'$ with the required properties. In this case, we write
\[
	(M,\alpha,w)\stackrel{S}{\rightsquigarrow}(M',\alpha',w').
\]
Theorem \ref{thm:mainev} allows us to prove an analogue of surface decomposition theorems in \cite{J:surface-decomp, km-sutures} for our version of instanton Floer homology.

\begin{prop}\label{sur-decom}
	Suppose $S$ is a decomposing surface for a sutured manifold $(M,\alpha,w)$. Assume that $S$ does not 
	have any closed components, and 
	for every connected component $V$
	of $R_{\pm} (\alpha)$, the set of closed components of $\partial S\cap V$ consist of parallel oriented 
	boundary-coherent simple closed curves. Suppose $(M',\alpha',w')$ is the 
	sutured manifold obtained from 
	decomposing along $S$. Then $SHI_*^3(M',\alpha')_{w'}$ is a summand of 
	$SHI_*^3(M,\alpha)_w$.
\end{prop}

An oriented simple closed curve $c$ in an oriented surface $V$ is boundary coherent if either $c$ is non-separating or removing $c$ from $V$ gives a disconnected surface with a connected component $V_0$ whose only boundary component is $c$. In the latter case, we require that the orientation of $c$ is given by the outward-normal-first convention applied to $V_0$.
\begin{proof}
	We follow a similar argument as in the proof of Theorem 
	\cite[Proposition 6.9 and Proposition 7.11]{km-sutures}.
	Without loss of generality, we can assume $S$ is connected. We can also assume that all connected components of 
	$\partial S$ have non-empty intersection with $R_{\pm}(\alpha)$ using 
	\cite[Lemma 4.5]{J:surface-decomp}. Next, we glue product 1-handles to $(M,\alpha)$ and $S$ 
	as in the proof of 
	\cite[Proposition 6.9]{km-sutures} to obtain a decomposing surface in a sutured manifold where $\partial S$ 
	consists of simple closed curves $C_1^\pm$, $\dots$
	$C_{n_\pm}^{\pm}$ in $R_{\pm}(\alpha)$. Lemma \ref{prod-1-handle-glue} implies that proving the claim for this new sutured manifold
	and the decomposing surface implies the claim for the original surface decomposition. 
	
	The closed curves $C_i^\pm$ determine linearly independent homology classes in $H_1(R_\pm(\alpha))$.
	If $n_+\neq n_-$, we may apply further finger moves as in 
	\cite[Lemma 4.5]{J:surface-decomp} and then glue product 1-handles as in \cite[Proposition 6.9]{km-sutures} to 
	increase the number of the boundary components of $\partial S$ in one of $R_{\pm}(\alpha)$ while preserving the 
	number of such components in the other one. Thus we may assume $n_+=n_-$. 	In summary, the boundary of our decomposing surface satisfies similar assumptions as in 
	\cite[Lemma 6.10]{km-sutures}.
	
	In order to form a closure of $(M,\alpha,w)$, first we glue a product sutured manifold 
	$[-1,1]\times F_{g,d}$ to 
	$M$ along $A(\alpha)$. The two boundary components $\overline R_\pm(\alpha)$ of the resulting 3-manifold
	contains the curves $C_i^{\pm}$ which are still linearly independent in $H_1(\overline R_\pm(\alpha))$.
	In particular, we can pick a diffeomorphism 
	$\varphi:\overline R_+(\alpha) \to \overline R_-(\alpha)$, which maps
	$C_i^+$ to $C_i^-$ in an orientation-reversing way. By forming the closure $Z_\alpha$ of 
	$(M,\alpha,w)$ via $\varphi$
	we obtain closed oriented connected surfaces $\overline R$ and $\overline S$ induced by 
	$\overline R_{\pm }(\alpha)$
	and $S$. Moreover, these two surfaces intersect in a collection of simple closed curves 
	$C_i$ that are induced by 
	$C_i^\pm$. By smoothing out these intersection curves we obtain another closed 
	oriented connected surface $F$
	in the same homology class as $[\overline R]+[\overline S]$. We assume that 
	$\overline w\cdot \overline R\equiv 0$ mod $3$. 
	Then 
	\[
		SHI_\ast^3(M,\alpha)_w = I_*^3(Z_\alpha,\gamma+\overline w\vert \overline R),
	\]
	where $\gamma$ is
	induced by a point in $F_{g,d}$ in the same way as
	before and $\overline w$ is the closure of $w$. The proof in the case
	$\overline w\cdot \overline R\nequiv 0\pmod{3}$ is similar, as we can replace the 1-cycle $\gamma$
	with some other multiple of it to define the instanton Floer homology of $(M,\alpha,w)$. 
	
	The key observation of \cite[Lemma 6.10]{km-sutures} is that $Z_\alpha$ can be also regarded as a closure 
	for $(M',\alpha',w')$ where the counterpart of the surface $\overline R$ is $F$. It can be easily 
	seen that the closure of
	$w'$ is still $\overline w$. 
	To be more precise, there is a disconnected surface $T$ without any closed component such that after gluing 
	the product sutured manifold $[-1,1]\times T$ to $(M',\alpha',w')$ and picking an appropriate
	 gluing map $\varphi'$
	we obtain a 3-manifold diffeomorphic to $Z_\alpha$ together with the surface $F$ and the 
	1-cycle $\overline w$. The disconnected surface $T$ satisfies 
	the required property mentioned above such that it can be used to define $SHI_*^3(M',\alpha')_{w'}$.
	In particular, this sutured instanton Floer homology group is isomorphic to 
	$I_*^3(Z_\alpha,\gamma+\overline w\vert F)$. 
	
	Let $v\in I_*^3(Z_\alpha,\gamma\vert F)$ be a simultaneous eigenvector for the action of the operators 
	$(\mu_2(\overline R),\mu_3(\overline R))$ with eigenvalues $(\lambda_2,\lambda_3)$. Since 
	$\overline R\cdot \gamma=1$, Proposition \ref{eigenvalue-S} implies that $(\lambda_2,\lambda_3)=(\sqrt3 a,\sqrt 3 ib)$
	with $a$ and $b$ of the same parity and
	\begin{equation}\label{str-ineq-special}
	  |a|+|b| \leq 2g(\overline R)-2.
	\end{equation}
	We also have $[F]=[\overline R]+[\overline S ]$, $\chi(F)=\chi(\overline R)+\chi(\overline S)$, 
	which implies that $v$ is 
	a simultaneous eigenvector for the action of
	$(\mu_2(\overline S),\mu_3(\overline S))$ with eigenvalues 
	$(\sqrt3(2g(F)-2)-\lambda_2,-\lambda_3)$. We apply 
	Proposition \ref{eigenvalue-S} again to get a bound on the norm of these eigenvalues:
	\begin{equation}\label{str-ineq-special-1}
	  |2g(F)-2-a|+|b| \leq 2g(\overline S)-2,
	\end{equation}
	The inequalities in \eqref{str-ineq-special} and \eqref{str-ineq-special-1} imply that 
	$(a,b)=(2g(\overline R)-2,0)$. As a result, the only simultaneous eigenvalue of 
	$(\mu_2(\overline R),\mu_3(\overline R))$ acting on $v\in I_*^3(Z_\alpha,\gamma\vert F)$
	is $(\sqrt3 (2g(\overline R)-2),0)$. This in turn implies that $I_*^3(Z_\alpha,\gamma\vert F)$
	is the summand of $I_*^3(Z_\alpha,\gamma\vert \overline R)$ given by the simultaneous 
	eigenspace of the operators $(\mu_2(\overline S),\mu_3(\overline S))$ corresponding to the 
	eigenvalue $(\sqrt{3}(2g(\overline S)-2),0)$. In particular, 
	$SHI_*^3(M',\alpha')$ is a summand of $SHI_*^3(M,\alpha)$.
\end{proof}

Recall that a sutured manifold $(M,\alpha)$ is {\emph{taut}} if $M$ is irreducible and $R_+(\alpha)$, $R_-(\alpha)$ are {\it norm minimizing} in their homology classes in $H_2(M,A(\gamma))$ \cite[Definition 2.4]{Gab:fol-sut}. (In general, if $Y$ is a 3-manifold and $Z$ is a codimension $0$ submanifold of $\partial Y$, then an embedding $(S,\partial S)$ into $(Y,Z)$ for a surface $S$ is norm minimizing if $S$ is incompressible and $S$ realizes the Thurston norm of the homology class $[S]\in H_2(Y,Z)$.) If $(M,\alpha)$ is taut, then we say $(M,\alpha,w)$ is taut for any choice of a 1-cycle $w$.

\begin{cor}\label{non-vanishing-sutured}
	For any balanced taut sutured manifold $(M,\alpha,w)$, the sutured instanton homology 
	group $SHI_*^3(M,\alpha)_w$
	is non-trivial. 
\end{cor}
\begin{proof}
	Following the proof of \cite[Theorem 1.4]{J:surface-decomp}, there 
	is a sequence of decompositions 
	\begin{equation}\label{hierarchy}
	  (M,\alpha)\stackrel{S_1}{\rightsquigarrow}(M_1,\alpha_1)\stackrel{S_2}{\rightsquigarrow}
	  \cdots \stackrel{S_n}{\rightsquigarrow}(M_n,\alpha_n)
	\end{equation}
	such that each $S_i$ satisfies the assumptions in Proposition \ref{sur-decom} and 
	$(M_n,\alpha_n)$ is a product sutured manifold. Now the claim follows from 
	Proposition \ref{sur-decom} and the fact that sutured instanton homology of a product 
	sutured manifold for any $U(3)$-bundle is 1-dimensional.
\end{proof}

\begin{cor} \label{non-vanishing-3-man}
	Suppose $Y$ is an irreducible 3-manifold, $\gamma$ is a $1$-cycle in $Y$ and $R$ is a norm minimizing embedded surface in $Y$.
	Then $I^3_*(Y\# T^3,\gamma+\gamma_1\vert R\#T^2)$ is non-trivial, where $\gamma_1$ is the 1-cycle in $T^3$ given by $S^1\times \{x\}$ with $x\in T^2$.
\end{cor}

\begin{proof}
	Cutting $(Y,\gamma)$ along $R$ produces a 3-manifold with two boundary components $R_+$ and $R_-$, which are copies of 
	$R$. Glue a 1-handle to this 3-manifold along the discs $D_\pm \subset R_{\pm}$ which correspond to a fixed disc $D\subset R$. The 
	resulting 3-manifold $M$ is a balanced sutured manifold with one suture $\alpha$ and the complement of an annular neighborhood of the suture in 
	the boundary is given by the surfaces $R_{\pm}\setminus D_{\pm}$. The $1$-cycle $\gamma$ induces a $1$-cycle $w$ in the sutured manifold $(M,\alpha)$.
	We may also regard $M$ is a submanifold of $Y$. In particular, 
	the properly embedded surfaces $R_{\pm}\setminus D_{\pm}$ in $(M,A(\alpha))$ are norm minimizing because $R$ is norm minimizing in $Y$.
	Furthermore, if $M$ is reducible, then the irreducibility of $Y$ implies that $R$ can be embedded in a ball in $Y$ which contradicts the assumption that 
	$R$ is norm minimizing. Thus $(M,\alpha,w)$ is taut, and hence $SHI_*^3(M,\alpha,w)$ is non-trivial.
	
	We take a closure of $(M,\alpha,w)$ by gluing $[-1,1]\times F_{1,1}$ and then gluing the two 
	boundary components of the resulting 3-manifold in the obvious way. The resulting closure can be identified with $Y\#T^3$ with the distinguished embedded 
	surface $R\#T^2$ and the $1$-cycle $\overline w=\gamma$. In particular, $SHI_\ast^3(M,\alpha)_w$ is equal to 
	$I^3_*(Y\# T^3,\gamma+\gamma_1\vert R\#T^2)$.
\end{proof}

\begin{remark}
	A similar proof can be used to show that $I^2_*(Y\# T^3,\gamma+\gamma_1\vert R\#T^2)\neq 0$. In the case that $(Y,\gamma)$ is $2$-admissible, combining this 
	with the connected sum theorems of instanton Floer homology in the admissible case \cite{scadutothesis}, 
	one can see that $I^2_*(Y,\gamma\vert R)$ is also non-trivial. 
	This is essentially the same non-vanishing 
	result as in \cite[Theorem 7.21]{km-sutures}. However, it seems that one needs to modify the statement and the proof of \cite[Theorem 7.21]{km-sutures}.
	It is reasonable to expect that there is a connected sum theorem for $U(3)$ instanton Floer homology which implies that 
	$I^3_*(Y\# T^3,\gamma+\gamma_1\vert R\#T^2)$ is non-trivial only if $I^3_*(Y,\gamma\vert R)\neq 0$ whenever $(Y,\gamma)$ is $3$-admissible. 
	In Section \ref{non-vanishing-symp}, we show that $I^3_*(Y,\gamma)$ is non-trivial using a non-vanishing result for symplectic 4-manifolds.
\end{remark}

\begin{proof}[Proof of Theorem \ref{thm:intro2}]
	Suppose $\gamma$ is a 1-cycle representing the Poincar\'e dual of $\omega$. To give a representation $\rho:\pi_1(Y)\to PU(3)$ 
	satisfying the required property, it suffices to find a projectively flat connection on a $U(3)$-bundle over $Y$ with $c_1={\rm PD}(\gamma)$.
	Furthermore, we may assume that $Y$ is prime. 
	If $Y$ is a rational homology sphere, then there is a flat $U(1)$-connection on $Y$ whose first Chern class is ${\rm PD}(\gamma)$. By taking the sum of 
	this connection and the trivial $SU(2)$ connection, we obtain a $U(3)$ flat connection with the required property. 
	If $Y=S^1\times S^2$, then the assumption implies that $\omega$ is trivial and we may take the trivial flat connection. 
	Otherwise $Y$ is irreducible with positive $b_1$ and Corollary \ref{non-vanishing-3-man} implies that $I_*(Y\# T^3,\gamma+\gamma_1\vert R)$ 
	is not zero, where $R$ is a norm minimizing embedded surface in $Y$ (representing a non-trivial homology class). In particular, 
	there exists a projectively flat connection on the $U(3)$-bundle over $Y$ with $c_1={\rm PD}(\gamma)$. This gives a representation $\rho:\pi_1(Y)\to PU(3)$ 
	satisfying the claim.
\end{proof}

For a knot $K$ in a 3-manifold $Y$, the $U(3)$ instanton knot homology of $(Y,K)$, denoted by $KHI_*^3(Y,K)$, is defined to be $SHI_*^3(Y(K),\alpha(K))$, where $(Y(K),\alpha(K))$ is the sutured manifold of Example \ref{knot-3-man-sutured-mfld}. As explained in \cite{km-sutures}, a closure of $(Y(K),\alpha(K))$ is given by $Z(K)$, the 3-manifold obtained by gluing $S^1\times F_{1,1}$ to the exterior of $K$ such that $S^1\times \{x\}$ is mapped to a meridian of $K$ for any $x\in \partial F_{1,1}$. Let $c$ and $c'$ be two simple closed curves in $F_{1,1}$ intersecting transversely in exactly one point and $T=S^1\times c$. Then 
\begin{equation}\label{eq:knothomologyu3}
  KHI_*^3(Y,K)=I_*^3(Z(K),c'\vert T).
\end{equation}
This instanton knot homology group is isomorphic to $I_*^3(Z(K),\gamma+c'\vert T)$ where $\gamma$ is the $1$-cycle $S^1\times \{x\}$ for some $x\in F_{1,1}$. Now if $K$ is null-homologous, then we can pick the gluing map in the definition of $Z(K)$ so that $\{{pt}\}\times \partial F_{1,1}$ is glued to a longitude of $K$. In this case, we can glue a Seifert surface $S$ of genus $g$ to $F_{1,1}$ and obtain an embedded surface $\overline S$ in $Z(K)$ of genus $g+1$. In particular, $(\mu_2(S),\mu_3(S))$ gives a pair of operators acting on $KHI_*^3(Y,K)$. The simultaneous generalized eigenspace decomposition with respect to the action of these operators determines a splitting of $KHI_*^3(Y,K)$ given as follows, that depends only on the homology class of $\overline S$:
\begin{equation}\label{eq:knothomologyu3graded}
  KHI_*^3(Y,K)=\bigoplus_{(a,b)\in \mathcal C_{g+1}}KHI_*^3(Y,K;a,b),
\end{equation}
where $KHI_*^3(Y,K;a,b)$ is the generalized eigenspace of $(\mu_2(S),\mu_3(S))$ for the eigenvalues $(\sqrt 3 a,\sqrt 3 i b)$. To limit the possible eigenvalues appearing in this decomposition, we have used Proposition \ref{eigenvalue-S}. The decomposition \eqref{eq:knothomologyu3graded} will be discussed further in Section \ref{sec:alexander}. 

\begin{proof}[Proof of Theorem \ref{thm:intro1}]
	Suppose $S$ is a Seifert surface of minimal genus for the knot $K$. Then the decomposition of $(Y(K),\alpha(K))$ along $S$ determines a sutured manifold
	$(Y(S),\alpha(S))$. It is shown in the proof of \cite[Proposition 5.33]{DX} that 
	\[
	  SHI_*^3(Y(S),\alpha(S)) \cong KHI_*^3(Y,K;\pm 2g,0).
	\]
	(This can be regarded as an instance of Theorem \ref{sur-decom} on surface decompositions.) 
	
	To prove the existence of the desired representation, we can assume that $Y\setminus K$ is irreducible. In the case that $Y\setminus K$ is irreducible, 
	$(Y(S),\alpha(S))$ is a taut sutured manifold. Corollary \ref{non-vanishing-sutured} implies that $SHI_*^3(Y(S),\alpha(S))$ is non-trivial and hence
	the rank of $KHI_*^3(Y,K)$ is at least $2$. Now the claim follows from \cite[Corollary 5.32]{DX}.
\end{proof}


\section{The Structure Theorem}\label{sec:structure}

In this section, we prove Theorem \ref{thm-intro:structure}, the $U(3)$ analogue of Kronheimer and Mrowka's celebrated structure theorem for $U(2)$ Donaldson invariants \cite{km-structure}. In the first subsection, we provide background on Fukaya--Floer instanton homology, focusing on the case of $U(3)$. In the second subsection, using these preliminaries, we prove the structure theorem. 

\subsection{Fukaya--Floer homology of $S^1\times \Sigma_g$}\label{FF3}

Fukaya--Floer homology is a variation of instanton Floer homology that is helpful to understand the $U(N)$ Donaldson invariants of a pair $(X,w)$ for some $z\in\bA^N(X)$, where $(X,w)$ is naturally written as a connected sum of $(W,c)$ and $(W',c')$ whose boundaries are an $N$-admissible pair $(Y,\gamma)$ (with different orientations) but $z$ is not necessarily induced by an element of $\bA^N(W)\otimes \bA^N(W')$ (see the gluing formula \eqref{pairing-FFH} below). The original idea of Fukaya--Floer homology goes back to \cite{Fuk:FF}, which was further developed in \cite{DB:FF} in the case that $N=2$. Here we follow \cite{DX} to give a review of the general properties of Fukaya--Floer homology in the case that $N=3$, and hence we often drop ``$3$'' from our notations. Then we proceed to study Fukaya--Floer homology of $S^1\times \Sigma_g$. For more details on the background material, the reader can see Subsections 3.3 and 6.3 of \cite{DX}. We remark that even in the case $N=2$, the algebraic formulation of \cite{DX} is more involved than what is proposed in \cite{DB:FF} because of bubbling phenomena. 

Suppose $(Y,\gamma)$ is an admissible pair and $L=(l_2,l_3)$ is a pair of elements of $H_1(Y;\Z)$. The Fukaya--Floer homology group $\mathbb I_*(Y,\gamma,L)$ is a module over a ring $R_3$, which is defined in the following way. First for any non-negative integer $j$ consider the ring 
\begin{equation}\label{}
  R_{3,j}:=\C[s_{2,i},s_{3,i}; 1\leq i\leq j]/(s_{2,i}^2,s_{3,i}^2).
\end{equation}
If $j\geq l$, then there is a homomorphism $R_{3,j}\to R_{3,l}$ that maps $s_{k,i}$ to $s_{k,i}$ if $i\leq l$ and to $0$ if $i>l$. Now let $R_3$ be the inverse limit of this inverse system of rings. In particular, 
\[
  t_k :=\sum_{i=0}^\infty s_{k,i}
\]
is an element of $R_3$ and this determines an algebra monomorphism from $\C[\![t_2,t_3]\!]$ to $R_3$.

The $R_3$-module $\mathbb I_*(Y,\gamma,L)$ is also defined as the inverse limit of an inverse system. For each $j$, there is a chain complex $(\fC_*^{\pi_j}(Y,\gamma)\otimes R_{3,j},d_j)$ defined over the ring $R_{3,j}$, where $\fC_*^{\pi_j}(Y,\gamma)$ is a choice of instanton Floer chain complex for the admissible pair $(Y,\gamma)$ and does not depend on $L$. The differential $d_j$ has the form 
\begin{equation}\label{dj}
  d_j=\sum_{S_2,\,S_3\subset [j]}\(\prod_{i\in S_2}s_{2,i}\)\(\prod_{i\in S_3}s_{3,i}\)d_{j}^{S_2,S_3}
\end{equation}
where $[j]=\{1,\,2,\dots,\,j\}$ and $d_{j}^{S_2,S_3}:\fC_*^{\pi_j}(Y,\gamma) \to \fC_*^{\pi_j}(Y,\gamma)$. In particular, $d_{j}^{\emptyset,\emptyset}$ is the ordinary Floer differential. If $j\geq l$, then there is a chain map 
\begin{equation}\label{Fjl-dom-co}
  F_{j,l}:(\fC_*^{\pi_j}(Y,\gamma)\otimes R_{3,j},d_j) \to (\fC_*^{\pi_l}(Y,\gamma)\otimes R_{3,l},d_l)
\end{equation}
of $R_{3,j}$-modules such that $F_{l,k} \circ F_{j,l}$ is chain homotopy equivalent to $F_{j,k}$. Analogous to the differential maps $d_j$, the chain maps have the form
\begin{equation}\label{Fjl}
  F_{j,l}=\sum_{S_2,\,S_3\subset [j]}\(\prod_{i\in S_2}s_{2,i}\)\(\prod_{i\in S_3}s_{3,i}\)F_{j,l}^{S_2,S_3},
\end{equation}
where $F_{j,l}^{\emptyset,\emptyset}:\fC_*^{\pi_j}(Y,\gamma) \to \fC_*^{\pi_l}(Y,\gamma)$ is the continuation map defining a chain homotopy equivalence between two chain complexes representing $I_*(Y,\gamma)$. The homology of $(\fC_*^{\pi_j}(Y,\gamma)\otimes R_{3,j},d_j)$ together with the homomorphisms induced by $F_{j,l}$ defines an inverse system and $\mathbb I_*(Y,\gamma,L)$ is the inverse limit of this system.

From \eqref{dj} and \eqref{Fjl}, it is clear that the homomorphisms $d_j$ and $F_{j,l}$ are compatible with a filtration on the Fukaya--Floer complexes. First define a filtration on $R_{3,j}$:
\[
  R_{3,j} =\mathcal F^0R_{3,j}\supset \mathcal F^1R_{3,j}\supset \mathcal F^2R_{3,j} \cdots \supset \mathcal F^{2j+1}R_{3,j}=0,
\]
where $\mathcal F^kR_{3,j}$ contains linear combinations of monomials 
\[
  \hspace{2cm}\(\prod_{i\in S_2}s_{2,i}\)\(\prod_{i\in S_3}s_{3,i}\) \hspace{1cm} \text{such that} \hspace{1cm} |S_1|+|S_2|\geq k.
\]
In particular, $\mathcal F^kR_{3,j} \cdot \mathcal F^lR_{3,j}$ is a subset of $\mathcal F^{k+l}R_{3,j}$, and the associated graded 
part of this filtration is a direct sum of $2^{2j}$ copies of $\C$. This filtration induces a filtration on $\fC_*^{\pi_j}(Y,\gamma)\otimes R_{3,j}$, and $d_j$ is a filtration preserving homomorphism such that the induced map at the level of the associated graded part is $d_{j}^{\emptyset,\emptyset}\otimes 1$. A similar comment applies to $F_{j,l}$. From these filtrations one can obtain a spectral sequence for any $j$ whose second page is $I_*(Y,\gamma)\otimes \C^{2^{2j}}$ and it abuts to $\mathbb I^{3,j}_*(Y,\gamma,L):=H(\fC_*^{\pi_j}(Y,\gamma)\otimes R_{3,j},d_j)$.

Fukaya--Floer homology is functorial with respect to cobordisms. Suppose $(W,c):(Y,\gamma)\to (Y',\gamma')$ is a cobordism of 3-admissible pairs, $z\in \bA^3(W)$, and $\Gamma$, $\Lambda$ are properly embedded oriented surfaces such that $\Gamma\cap Y$, $\Lambda\cap Y$ represent homology classes $l_2,\,l_3\in H_1(Y;\Z)$ and $\Gamma\cap Y'$, $\Lambda\cap Y'$ represent homology classes $l_2',\,l_3'\in H_1(Y';\Z)$. Then there is an $R_3$-module homomorphism 
\[
  \mathbb I(W,c,ze^{t_2\Gamma_{(2)}+t_3\Lambda_{(3)}}):\mathbb I_*(Y,\gamma,L) \to \mathbb I_*(Y',\gamma',L')
\]
with $L=(l_2,l_3)$ and $L'=(l_2',l_3')$. There is a slight variation of the above construction when one of the ends of the cobordism $(W,c)$ is empty. If $Y'$ is empty, then $W$ is a 4-manifold with boundary $-Y$ and we have an $R_3$-module map
\[
  D_{W,c}(ze^{t_2\Gamma_{(2)}+t_3\Lambda_{(3)}}):\mathbb I_*(Y,\gamma,L) \to R_3
\]
and if $Y$ is empty, then $W$ is a 4-manifold with boundary $Y'$ and we have an element
\[
  D_{W,c}(ze^{t_2\Gamma_{(2)}+t_3\Lambda_{(3)}})\in \mathbb I_*(Y',\gamma',L').
\]
These cobordism maps are defined by first constructing $R_{3,j}$-module chain maps between $(\fC_*^{\pi_j}(Y,\gamma)\otimes R_{3,j},d_j)$ and $(\fC_*^{\pi_j}(Y',\gamma')\otimes R_{3,j},d_j)$ that commute with the maps \eqref{Fjl-dom-co} up to chain homotopy. Furthermore, these chain maps respect the filtrations induced by that of $R_{3,j}$, and the leading order terms with respect to such filtrations are given by the cobordism maps of ordinary instanton Floer complexes.

The above homomorphisms are well-behaved with respect to composition of cobordisms. For instance if $(W,c)$ is a pair with boundary $(Y,\gamma)$ and $(W',c')$ is a pair with boundary the orientation-reversal of $(Y,\gamma)$, then we can glue them to obtain a closed pair $(W\#W',c\#c')$. If $\Gamma$, $\Lambda$ are properly embedded surfaces in $W$ and $\Gamma'$, $\Lambda'$ are properly embedded surfaces in $W'$ such that $\Gamma$ and $\Gamma'$ (respectively, $\Lambda$ and $\Lambda'$) agree over the boundary and we can glue them to obtain a closed oriented embedded surface $\Gamma\#\Gamma'$ (respectively, $\Lambda\#\Lambda'$), then
\begin{align}\label{pairing-FFH}
		\langle D_{W,c}(ze^{t_2\Gamma_{(2)}+t_3\Lambda_{(3)}}),\;\; &
		D_{W',c'}(z'e^{t_2\Gamma'_{(2)}+t_3\Lambda'_{(3)}})\rangle \nonumber\\[0.75mm]
		=&D_{W\#W',c\#c'}(zz'e^{t_2\Gamma\#\Gamma'_{(2)}+t_3\Lambda\#\Lambda'_{(3)}}),
\end{align}
where $z\in \bA^3(W)$ and $z'\in \bA^3(W')$, and the left hand side is the obvious pairing. The invariant on the right hand side of \eqref{pairing-FFH} is given by $U(3)$ invariants of $(W\#W',c\#c')$ when $b^+(W\#W')>1$. In the special case that $b^+(W\#W')=1$, one can still define $U(3)$ polynomial invariants for $(W\#W',c\#c')$. However, this invariant depends on the choice of the metric and the right hand side of \eqref{pairing-FFH} is the invariant for a metric that we stretch along the embedded surface $Y$ in $W\#W'$. We also remark that the right hand side of \eqref{pairing-FFH} is an element of the subalgebra $\C[\![t_2,t_3]\!]$ of $R_3$.

The main instance of $U(3)$ Fukaya--Floer homology relevant to this paper is that of $(S^1\times \Sigma_g,\gamma_d,L)$ with $d$ coprime to $3$ and $L=([S^1\times \{{pt}\}],[S^1\times \{{pt}\}])$. Following a similar notation as in Section \ref{sec:strategy}, we write 
\[
	\widetilde {\mathbb V}^3_{g,d} := \mathbb I_\ast(S^1\times \Sigma_g,\gamma_d,L),
\]
which is an $R_3$-module. Similar to \eqref{eq:pairingintro}, $[0,1]\times(S^1\times \Sigma_g, \gamma_d, L)$, viewed as a cobordism from two copies of $(S^1\times \Sigma_g,\gamma_d,L)$ to the empty set, induces a bilinear pairing
\begin{equation}\label{eq:pairingfukfloer}
	\langle \cdot,\cdot \rangle : \widetilde {\mathbb V}^3_{g,d} \otimes \widetilde {\mathbb V}^3_{g,d} \to R_3.
\end{equation}

Analogous to $V^3_{g,d}$ and in the same way as in \eqref{iso-vect-coh-ins}, $\widetilde {\mathbb V}^3_{g,d}$ is isomorphic to the cohomology of $\cN_g$ with an appropriate coefficient ring. An explicit version of the isomorphism in \eqref{iso-vect-coh-ins} is given in \cite[Theorem 3.18]{DX}. Focusing on the $N=3$ case, there is a vector space homomorphism $S:H^\ast(\cN_g;\C) \to \bA_g^3$ such that the map 
\begin{equation}\label{iso-coh-ins}
  P:H^\ast(\cN_g;\C)[\epsilon]/(\epsilon^3-1)\to V_{g,d}^3
\end{equation}
defined using the relative invariants 
\[
  P(\epsilon^i \cdot \sigma)=D_{\Delta_g,\delta_{g,d}+i\Sigma}(S(\sigma)),
\] 
with $\Delta_g:=D^2\times \Sigma_g$ and $\delta_{g,d}=D^2\times \{x_1,\dots,x_d\}$, is an isomorphism. Similarly, we can define $\mathbb P:H^\ast(\cN_g;R_3)[\epsilon]/(\epsilon^3-1)\to \widetilde {\mathbb V}_{g,d}^3$, an analogue of \eqref{iso-coh-ins} for the Fukaya--Floer homology group of $S^1\times \Sigma_g$, by setting
\begin{equation}\label{rel-inv-gen}
  {\mathbb P}(\epsilon^i \cdot \sigma)=D_{\Delta_g,\delta_{g,d}+i\Sigma}(S(\sigma)e^{t_2D_{(2)}+t_3D_{(3)}}).
\end{equation}
Here we extend $S$ as a module homomorphism $H^\ast(\cN_g;\C)\otimes R_3 \to \bA_g^3\otimes R_3$ in the obvious way and $D$ denotes the disc $D\times \{{pt}\}$ in $\Delta_g$.

\begin{lemma}
	The $R_3$-module map ${\mathbb P}:H^\ast(\cN_g;R_3)[\epsilon]/(\epsilon^3-1)\to \widetilde {\mathbb V}_{g,d}^3$ is an isomorphism. In particular, $\widetilde {\mathbb{V}}^{3}_{g, d}$ is a free 
	$R_3$-module. 
\end{lemma}
\begin{proof}
	It suffices to show that $\mathbb P_j:H^\ast(\cN_g;R_{3,j})[\epsilon]/(\epsilon^3-1)\to
	{\mathbb I}^{3,j}_*(S^1\times \Sigma_g,\gamma_{g,d},L_g)$,  the $R_{3,j}$-module homomorphism given by 
	\[
  	  {\mathbb P}_j(\epsilon^i \cdot \sigma)=D^{3,j}_{\Delta_g,\delta_{g,d}+i\Sigma}(S(\sigma)e^{D_{(2)}+D_{(3)}}),
	\]	
	is an isomorphism. Since $\mathbb P_j$ is an $R_{3,j}$-module homomorphism, it is a filtration preserving 
	homomorphism with respect to the 
	filtration induced by that of $R_{3,j}$. The induced morphism of spectral sequences on the 
	second page maps 
	$H^\ast(\cN_g;\C)[\epsilon]/(\epsilon^3-1)\otimes R_{3,j}$ to $I_*(S^1\times \Sigma_g,\gamma_{g,d})\otimes R_{3,j}$ 
	by the map $P\otimes 1$.
	In particular, it is an isomorphism.
\end{proof}

The ring $R_3$ is not an integral domain and for our purposes it is easier to work with modules over an integral domain. We define $\mathbb{V}^{3}_{g, d}$ be the $\C[\![t_2, t_3]\!]$-module given by
\[
	\mathbb{V}^{3}_{g, d} := \text{im}\left( \mathbb{P}|_{H^\ast(\cN_g;\C[\![t_2, t_3]\!])[\epsilon]/(\epsilon^3-1)}\right)  
\]
It is clear from the definition that $\mathbb{V}^{3}_{g, d}$ as a $\C[\![t_2, t_3]\!]$-module is isomorphic to $\C[\![t_2, t_3]\!]^{3n_{g}}$ with $n_{g}=\dim_\C H^\ast(\cN_g;\C)$. We have the following alternative identification.

\begin{lemma}\label{dual-iso}
	Suppose $\{\sigma_k\}_{1\leq k\leq n_g}$ is a basis for $H^\ast(\cN_g;\C)$ as a vector space over $\C$. 
	Then, the map $\Phi:\mathbb{V}^{3}_{g, d} \to \C[\![t_2, t_3]\!]^{3n_{g}}$ given by
	\begin{equation}\label{Phi-def}
	  \Phi(\zeta):=(\langle \zeta, D_{\Delta_g,\delta_{g,d}+i\Sigma}
	  (S(\sigma_k)e^{t_2D_{(2)}+t_3D_{(3)}}) \rangle)_{1\leq k\leq n_g,\, 0\leq i\leq 2},
	\end{equation}
	defined using the pairing \eqref{eq:pairingfukfloer}, is an isomorphism.
\end{lemma}
\noindent Note that \eqref{pairing-FFH} implies the right side of \eqref{Phi-def} is indeed an element of $\C[\![t_2, t_3]\!]^{3n_{g}}$.

\begin{proof}
	If $\zeta$ is in the kernel of $\Phi$, then its pairing with any element of $\widetilde {\mathbb V}^{3}_{g, d}$ vanishes. Thus \cite[Proposition 3.30]{DX}
	implies that $\zeta=0$. It remains to show that $\Phi$ is surjective. Fix
	\[
	  \bv=\sum_{k=0}^\infty\sum_{i=0}^{k} t_2^it_3^{k-i} v_{i,k-i} \in \C[\![t_2, t_3]\!]^{3n_{g}},
	\]  
	where $v_{i,j}\in \C^{3n_{g}}$. We inductively define an element 
	\[
	  \sigma=\sum_{k=0}^\infty\sum_{i=0}^{k} t_2^it_3^{k-i} \sigma_{i,k-i}
	\]
	with $\sigma_{i,j}\in H^\ast(\cN_g;\C)[\epsilon]/(\epsilon^3-1)$ such that for any integer $n$, in the expression
	\[
	  \Phi\circ \mathbb P(\sum_{k=0}^n\sum_{i=0}^{k} t_2^it_3^{k-i} \sigma_{i,k-i})-\bv\in \C[\![t_2, t_3]\!]^{3n_{g}}
	\]
	only terms of the form $t_2^it_3^j$ with $i+j>n$ appear. In fact, assuming this holds for a given $n$, then we have the following, where $w_{i,j}\in \C^{3n_{g}}$:
	\[
	  \Phi\circ \mathbb P(\sum_{k=0}^n\sum_{i=0}^{k} t_2^it_3^{k-i} \sigma_{i,k-i})-\bv=\sum_{k=n+1}^\infty\sum_{i=0}^{k} t_2^it_3^{k-i} w_{i,k-i}
	\]
	The non-degeneracy of the pairing on $V_{g,d}^3$ implies that for any $0\leq i\leq n+1$, there is a unique 
	$\sigma_{i,n+1-i}\in H^\ast(\cN_g;\C)[\epsilon]/(\epsilon^3-1)$ such that $\phi\circ P(\sigma_{i,n+1-i})=w_{i,n+1-i}$. Here 
	$\phi:V^{3}_{g, d} \to \C^{3n_{g}}$ is defined in a similar way as $\Phi$. It is straightforward to check that 
	we can carry out the induction step with this choice of $\sigma_{i,j}$ when $i+j=n+1$.
\end{proof}

\begin{cor}\label{charVgd3}
	Any relative invariant $D_{\Delta_g,\delta_{g,d}+i\Sigma}(ze^{D_{(2)}+D_{(3)}})$ where $z\in \bA_g^3$ is an element of $\mathbb{V}^{3}_{g, d}$.
	In particular, $\mathbb{V}^{3}_{g, d}$ is the $\C[\![t_2, t_3]\!]$-module generated by such invariants.
\end{cor}
\begin{proof}
	Using Lemma \ref{dual-iso}, there is $\zeta\in \mathbb{V}^{3}_{g, d}$ such that 
	\[
	  \langle \zeta-D_{\Delta_g,\delta_{g,d}+j\Sigma}(z e^{D_{(2)}+D_{(3)}}) ,\;\; D_{\Delta_g,\delta_{g,d}+i\Sigma}(S(\sigma)e^{t_2D_{(2)}+t_3D_{(3)}}) \rangle=0.
	\]
	This implies that the pairing of $\zeta-D_{\Delta_g,\delta_{g,d}+j\Sigma}(z e^{D_{(2)}+D_{(3)}})$ with any element of $\widetilde {\mathbb{V}}^{3}_{g, d}$ 
	is trivial, and hence by \cite[Proposition 3.30]{DX} this element vanishes. In particular, $\zeta=D_{\Delta_g,\delta_{g,d}+j\Sigma}(z e^{D_{(2)}+D_{(3)}})$ belongs to $\C[\![t_2, t_3]\!]$.
\end{proof}

For any integer $i$ and $z \in \bA_g^3$, consider the homomorphism
\[
  \mathbb I([-1,1]\times S^1\times \Sigma_g,[-1,1] \times\gamma_{g,d}+ i\Sigma_g, ze^{t_2C_{(2)}+t_3C_{(3)}}):
  \widetilde {\mathbb{V}}^{3}_{g, d} \to \widetilde {\mathbb{V}}^{3}_{g, d}
\]
where $C=[-1,1]\times S^1\times \{{pt}\}$. Corollary \ref{charVgd3} and functoriality of Fukaya--Floer homology implies that this homomorphism maps ${\mathbb{V}}^{3}_{g, d}$ to itself. This gives $\mathbb{V}^{3}_{g, d}$ the structure of a cyclic module over $\bA_g^3\otimes \C[\![t_2, t_3]\!][\epsilon]/(\epsilon^3-1)$. That is to say, there is an ideal $\mathbb{J}_{g,d}^3$ of $\bA_g^3\otimes \C[\![t_2, t_3]\!][\epsilon]$ containing $\epsilon^3-1$ such that 
\[ 
  \mathbb V_{g,d}^3  = \bA^3_g\otimes \C[\![t_2, t_3]\!] [\varepsilon]/\mathbb J_{g,d}^3.
\]  
The following is another consequence of Lemma \ref{dual-iso}.
\begin{cor}\label{FFH-deform}
	For any element $z$ of the ideal ${J}_{g,d}^3\subset \bA_g^3[\epsilon]$, there is $\bz \in \mathbb{J}_{g,d}^3$ such that 
	\[
	  \bz=z+\sum_{k=1}^\infty\sum_{i=0}^{\infty}t_2^{i}t_3^{k-i}z_{i,k-i}.
	\]
	That is to say, $z$ is the constant term of the power series $\bz$.
\end{cor}
\begin{proof}
	We may regard $z$ as an element of $\bA_g^3\otimes \C[\![t_2, t_3]\!][\epsilon]$ where the coefficient of $t_2^it_3^j$ is zero unless $i=j=0$.
	Thus $z$ determines an element $\zeta$ of $\mathbb V_{g,d}^3$. Since $z\in {J}_{g,d}^3$, $\Phi(\zeta)$ has a trivial constant term and hence we can find 
	$\bv_2,\,\bv_3\in \C[\![t_2, t_3]\!]^{3n_{g}}$ such that 
	\begin{equation}\label{bz-form}
	  \Phi(\zeta)=t_2\bv_2+t_3\bv_3.
	\end{equation}
	By Lemma \ref{dual-iso}, we can find $\eta_2,\, \eta_3\in \mathbb V_{g,d}^3$ such that $\Phi(\eta_i)=\bv_i$. In particular, $\zeta-t_2\eta_2-t_3\eta_3$
	is a trivial element of $\mathbb V_{g,d}^3$. Since $\mathbb{V}^{3}_{g, d}$ is a cyclic module over $\bA_g^3\otimes \C[\![t_2, t_3]\!][\epsilon]/(\epsilon^3-1)$,
	there are $\bz_2,\,\bz_3\in \bA_g^3\otimes \C[\![t_2, t_3]\!][\epsilon]/(\epsilon^3-1)$ such that $\bz_i$ is mapped to $\eta_i$. This implies that 
	$\bz:=z-t_2\bz_2-t_3\bz_3$ is in $\mathbb{J}_{g,d}^3$ and has the form in \eqref{bz-form}.
\end{proof}

Next, we define a Fukaya--Floer analogue of the simple-type ideal \eqref{eq:simpletypeideal}:
\[
\mathbb{S}^{3}_{g, d}:=\ker(\beta^3_2-27)\cap \ker(\beta_3) \cap \bigcap_{\substack{r=2,3\\ 1\leq j\leq 2g}} \ker(\psi^j_r)  \subset \mathbb{V}^{3}_{g, d}.
\]
An important ingredient in the proof of the structure theorem involves an understanding of this $\C[\alpha_2,\alpha_3,\beta_2,\beta_3][\![t_2, t_3]\!]$-module. The following is an adaption of the main argument that proves Theorem \ref{thm:mainev}, given in Section \ref{sec:strategy}. 

\begin{theorem}\label{thm:fukayafloersimpletypeideal}
$\mathbb{S}^{3}_{g, d}$ is a free $\C[\![t_2, t_3]\!]$-module of rank $3(2g-1)^2$. Moreover,
\begin{equation}\label{eq:simpletypeidealdecompfukayafloer}
\mathbb{S}^{3}_{g, d}= \bigoplus_{ \substack{k\in \{0,1,2\}\\ (a,b)\in \mathcal{C}_g}}R_{k, a,b}
\end{equation}
where $R_{k, a,b}$ is the free $\C[\![t_2, t_3]\!]$-module of rank one given by
\begin{equation}\label{Rkab}
R_{k, a,b}=\frac{\C[\![t_2, t_3]\!][\alpha_2, \alpha_3, \beta_2,\beta_3]}{(\alpha_2-(\zeta^k\sqrt{3} a + \zeta^{2k}t_2), \alpha_3-(\zeta^{2k}\sqrt{-3} b - 2 \zeta^{k} t_3), \beta_2-3\zeta^{2k}, \beta_3)}
\end{equation}
The above description also determines $\mathbb{S}^{3}_{g, d}$ as an $\C[\alpha_2,\alpha_3,\beta_2,\beta_3][\![t_2, t_3]\!]$-module.
\end{theorem}

As preparation for the proof, we need the {\it blowup formula} for $U(3)$ polynomial invariants. This will also be an essential ingredient in the proof of the structure theorem. Write $\widehat{X}$ for a blowup of $X$, and denote the exceptional class by $E\in H^2(\widehat{X};\Z)$. The following is essentially due to Culler \cite{culler}, and is stated in \cite[\S 2.5]{DX}. 

\begin{theorem}\label{thm:blowup}
	If $(X,w)$ is $U(3)$ simple type, then for $\Gamma,\Lambda\in H^2(X;\Z)$, we have
	\begin{align*}
		\mathbb{D}_{\widehat{X},w}(t_2E_{(2)} +t_3E_{(3)} &+\Gamma_{(2)}+\Lambda_{(3)}) \\[3mm]
	&= \frac{1}{3}e^{-t_2^2/2+t_3^2}\left(\cosh (\sqrt{3}t_2) + 2 \cos(\sqrt{3}t_3)\right) \mathbb{D}_{X,w}(\Gamma_{(2)}+\Lambda_{(3)}),
	\end{align*}
	\begin{align*}
		\mathbb{D}_{\widehat{X},w+E}( & t_2E_{(2)}+t_3E_{(3)}+\Gamma_{(2)}+\Lambda_{(3)})  \\[3mm]
	&= \frac{1}{3}e^{-t_2^2/2+t_3^2}\left(\cosh (\sqrt{3}t_2) - \cos(\sqrt{3}t_3) - \sqrt{3}\sin(\sqrt{3}t_3)\right) \mathbb{D}_{X,w}(\Gamma_{(2)}+\Lambda_{(3)}).
	\end{align*}
\end{theorem}

Below, we will make use of the identity
\begin{equation}\label{eq:zeta-twisted-rel}
	D_{X,w}((1+\frac{1}{3}\zeta^k x_{(2)}+\frac{1}{9}\zeta^{2k}x_{(2)}^2)e^{z}) = \zeta^{k \, d_w} \mathbb{D}_{X,w}(\zeta^{-k\text{deg}(z)} z)
\end{equation}
where $d_w:=b^+(X)-b^1(X)-w\cdot w+1$ and $z$ is a homogenous element of $\bA^3(X)$. This relation follows from the observation that the mod $3$ dimension of the moduli spaces of $U(3)$ instantons for $(X,w)$ is fixed and equal to $d_w$.

\begin{proof}[Proof of Theorem \ref{thm:fukayafloersimpletypeideal}]
	Suppose $\widetilde {\mathbb{J}}_{g,d}^3$ is the ideal of $\bA_g^3\otimes \C[\![t_2, t_3]\!][\epsilon]$ generated by 
	$\mathbb{J}_{g,d}^3$ and $(\beta_2^3-27,\beta_3,\psi_2^i,\psi_3^i)_{i=1}^{2g}$.
	The pairing $\langle \cdot , \cdot \rangle : \mathbb{V}^{3}_{g, d}\otimes_{\C[\![t_2, t_3]\!]}  \mathbb{V}^{3}_{g, d}\to \C[\![t_2, t_3]\!]$ 
	induces 
	\[
	  \mathbb{S}^{3}_{g, d} \otimes_{\C[\![t_2, t_3]\!]} \mathbb{V}^{3}_{g, d}/(\beta_2^3-27,\beta_3,\psi_2^i,\psi_3^i)_{i=1}^{2g} \to  \C[\![t_2, t_3]\!].
	\]
	The non-degeneracy of the pairing gives
	\[
	  \text{rank}_{\C[\![t_2, t_3]\!]}  (\mathbb{S}^{3}_{g, d}) \leq 
	  \text{rank}_{\C[\![t_2, t_3]\!]}  \(\bA_g^3\otimes \C[\![t_2, t_3]\!][\epsilon]/ \widetilde {\mathbb{J}}_{g,d}^3\).
	\]
	Corollary \ref{FFH-deform} implies that the right hand side of the above inequality is not greater than 
	\[
	   \dim_\C\bA_g^3[\varepsilon]/\widetilde{J}_{g,d}^3 = 3(2g-1)^2,
	\]
	where the latter is established in Section \ref{sec:strategy}.
	Thus
	\begin{equation}\label{eq:rankineqfukfloersimpletype}
		\text{rank}_{\C[\![t_2, t_3]\!]}  \mathbb{S}^{3}_{g, d} \leq 3(2g-1)^2.
	\end{equation}

	We can construct elements of the simple-type ideal using a $K3$ surface. (A similar construction can be done for 
	smooth 4-manifolds of $U(3)$ simple type.) 
	Let $\Sigma'$ be a surface of genus $g$ in a $K3$ surface with $\Sigma'\cdot \Sigma'=2g-2$. For instance, we can construct
	$\Sigma'$ in the following way. The 4-manifold $K3$ admits an elliptic fibration with a section that is a 
	$(-2)$-embedded sphere. The union of this sphere and $g$ regular fibers, after resolving the intersection points, gives a 
	surface with the desired genus and self intersection number. We fix another surface $F$ with $F\cdot F=0$ and 
	$F\cdot \Sigma'=1$. For instance, take $F$ to be a regular fiber. Let also $w$ be the 
	union of $d$ other regular fibers and regard it as a 2-cycle in $K3$ with trivial self-intersection number. 
	Next, let $X$ be the blowup of 
	the $K3$ surface at $2g-2$ points on $\Sigma'$ away from $F$ and $w$, and denote the proper transform of 
	$\Sigma'$ by $\Sigma$. Then $\Sigma$ determines a surface of genus $g$ and self intersection number $0$. 
	We also obtain a surface and a 2-cycle in $X$ induced by $F$ and $w$, which are denoted by the same notation. 

	Removing a regular neighborhood of $\Sigma$ from $X$, $w$ and $F$ determines a 4-manifold $X^\circ$ with boundary 
	$S^1\times \Sigma_g$, a 2-cycle $w^\circ$ that intersects the boundary of $X$ at $S^1\times \{x_1,\dots,x_d\}$ and an
	 embedded surface $F^\circ$ which intersects the boundary at $S^1\times \{y\}$. In particular, for any $z\in \bA^3(X)$, the 
	 following is an element of $\mathbb V_{g,d}^3$:
	 \begin{equation}\label{rel-inv-mbbVgd3}
	   D_{X^\circ, w^\circ}(ze^{t_2F^\circ_{(2)}+t_3F^\circ_{(3)}}).
	 \end{equation}
	 Analogous to \eqref{rel-inv-Sgd3}, one can see the above element of $\mathbb V_{g,d}^3$ belongs to 
	 $\mathbb{S}^{3}_{g, d}$.
	 
	 Next, we define a homomorphism $\Psi:\mathbb{S}^{3}_{g, d} \to \C[\![t_2, t_3]\!]^{3(2g-1)^2}$ 
	 and use the upper bound in \eqref{eq:rankineqfukfloersimpletype} 
	 on the rank of $\mathbb{S}^{3}_{g, d}$ and the elements of $\mathbb{S}^{3}_{g, d}$ constructed in \eqref{rel-inv-mbbVgd3} to
	 show that $\Psi$ is an isomorphism. 
	 For any $\lambda=(a,b,k)$ in $\mathcal C_g\times \{0,1,2\} $, let $P_\lambda\in \C[\![t_2, t_3]\!][w,x,y]$
	 be a polynomial such that for any $\lambda'=(a',b',k')\in \mathcal C_g\times \{0,1,2\} $, the value of $P_\lambda$ at
	 \[
	  (3\zeta^{2k'},\sqrt{3}\zeta^{k'}a'+\zeta^{2k'}t_2, \sqrt{-3}\zeta^{2k'} b'-2\zeta^{k'}t_3)
	\]  
	is $1$ if $\lambda'=\lambda$ and is $0$ if $\lambda'\neq \lambda$. The homomorphism $\Psi$ is defined as
	\[
	  \Psi(\zeta):=\left \{\langle \zeta , D_{\Delta_g,\delta_{g,d}}(P_\lambda(x_{(2)},\Sigma_{(2)},\Sigma_{(3)}) e^{D_{(2)}+D_{(3)}}) 	  \rangle\right \}_{\lambda}.
	\]
	
	To compute $\Psi(\zeta)$ for an element of $\mathbb{S}^{3}_{g, d}$ that is a relative invariant as in \eqref{rel-inv-mbbVgd3}, we can 
	use the pairing formula \eqref{pairing-FFH} to compute 
	\begin{align}\label{pairing-Psi}
	  \langle D_{X^\circ, w^\circ}(ze^{t_2F^\circ_{(2)}+t_3F^\circ_{(3)}}), \;\;
	  D_{\Delta_g,\delta_{g,d}}(&P_\lambda(x_{(2)},\Sigma_{(2)},\Sigma_{(3)}) e^{D_{(2)}+D_{(3)}})\rangle \nonumber\\[3mm]
	  =&D_{X, w}(zP_\lambda(x_{(2)},\Sigma_{(2)},\Sigma_{(3)})e^{t_2F_{(2)}+t_3F_{(3)}}).
	\end{align}
	We consider the special case that 
	\begin{equation}\label{z-choice}
	  z=(1+\frac{1}{3}\zeta^k x_{(2)}+\frac{1}{9}\zeta^{2k}x_{(2)}^2)P(x_{(2)},\Sigma_{(2)},\Sigma_{(3)})
	\end{equation}
	for some $P\in \C[\![t_2, t_3]\!][w,x,y]$.
	 Then the right hand side of  \eqref{pairing-Psi} is given by evaluating the following expression at $s_2=s_3=0$:
	\begin{equation}\label{pairing-series}
	 	R(3\zeta^{2k},\frac{\partial}{\partial s_2},\frac{\partial}{\partial s_3}) \widehat{D}^{\zeta^k}_{X,w}(
	  e^{s_2\Sigma_{(2)} + s_3\Sigma_{(3)}+t_2F_{(2)}+t_3F_{(3)}})
	\end{equation}	
	where $R$ is the element of $\C[\![t_2, t_3]\!][w,x,y]$ given by $P\cdot P_\lambda$.
	Here we use the fact that $K3$ has $U(3)$ simple type \cite{DX}.
	For any cycle $w$ in a $K3$ surface and $\Gamma,\,\Lambda\in H_2(K3)$, the $U(3)$ Donaldson-type invariant is computed in \cite{DX} to be
	\begin{equation}
		\mathbb{D}_{K3,w}(\Gamma_{(2)}+\Lambda_{(3)})=e^{\frac{Q(\Gamma)}{2}-Q(\Lambda)}.  	
	\end{equation}
	This identity, Theorem \ref{thm:blowup} and \eqref{eq:zeta-twisted-rel} can be used to show that 
	\eqref{pairing-series} is equal to 
	\[
	  R(3\zeta^{2k},\frac{\partial}{\partial s_2},\frac{\partial}{\partial s_3}) 
	  \left[\frac{1}{3^{2g-2}}\zeta^ke^{\zeta^{2k}s_2t_2-2\zeta^{k}s_3t_3}(\cosh (\sqrt{3}\zeta^{k}s_2) + 2 \cos(\sqrt{3}\zeta^{2k}s_3))^{2g-2}\right].
	\]
	In particular, \eqref{pairing-series} is equal to
	\[
	  R(3\zeta^{2k},\frac{\partial}{\partial s_2},\frac{\partial}{\partial s_3}) 
	  \left[e^{\zeta^{2k}s_2t_2-2\zeta^{k}s_3t_3} \sum_{(a,b)\in \mathcal C_g} c_{a,b}e^{\sqrt3 \zeta^{k}as_2+\sqrt{-3}\zeta^{2k}b s_3}\right],
	\]
	for some non-zero constants $c_{a,b}$. We may simplify the above expression as 
	\[
	  e^{\zeta^{2k}s_2t_2-2\zeta^{k}s_3t_3} \sum_{(a,b)\in \mathcal C_g} 
	  c_{a,b}R(3\zeta^{2k},\sqrt3 \zeta^{k}a+\zeta^{2k}t_2,\sqrt{-3}\zeta^{2k}b-2\zeta^{k}t_3) 
	  e^{\sqrt3 \zeta^{k}as_2+\sqrt{-3}\zeta^{2k}b s_3}.
	\]
	The assumption $R=P\cdot P_\lambda$ for a fixed $\lambda=(a,b)$ can be used to further \eqref{pairing-series} 
	simplify as
	\[
	  c_{a,b}e^{\zeta^{2k}s_2t_2-2\zeta^{k}s_3t_3}P(3\zeta^{2k},\sqrt3 \zeta^{k}a+\zeta^{2k}t_2,\sqrt{-3}\zeta^{2k}b-2\zeta^{k}t_3) 
	  e^{\sqrt3 \zeta^{k}as_2+\sqrt{-3}\zeta^{2k}b s_3}.
	\]
	
	For a given $(a_0,b_0,k_0)\in \mathcal C_g\times \{0,1,2\} $, we pick $z$ in \eqref{z-choice} with $k=k_0$ and 
	$P=P_{\lambda_0}$ where $\lambda_0=(a_0,b_0)$. Then all components of $\Psi$ applied to this element of 
	$\mathbb{S}^{3}_{g, d}$ are equal to $0$ except the component corresponding to $(a_0,b_0,k_0)$, which is a non-zero 
	real number. This shows that the map $\Psi$ is surjective. This observation and \eqref{eq:rankineqfukfloersimpletype}
	imply that the rank of $\mathbb{S}^{3}_{g, d}$ is $3(2g-1)^2$ and the kernel of $\Psi$ is torsion. However, the kernel
	is a submodule of the free module $\mathbb V_{g,d}^3$, and hence the kernel of $\Psi$ is trivial. Consequently, $\Psi$ gives
	an isomorphism, and is the direct sum of rank $1$ modules given by the above elements as $(a_0,b_0,k_0)$ ranges over all 
	elements of $\mathcal C_g\times \{0,1,2\}$. Furthermore, the computation of the previous paragraph shows that any such rank
	$1$ summand is invariant with respect to the action of the operators $\alpha_2$, $\alpha_3$, $\beta_2$ and $\beta_3$, and it
	is isomorphic $R_{k_0, a_0,b_0}$, defined as in \eqref{Rkab}.
\end{proof}

\subsection{Proof of the structure theorem}\label{subsection:structu-thm}

We now prove the structure theorem. For the convenience of the reader, we recall the statement of Theorem \ref{thm-intro:structure}. Retaining the convention of the previous subsection, we write $D_{X,w}$ for $D_{X,w}^3$, and so forth. Let $\zeta=e^{2\pi i/3}$. 

\begin{theorem}\label{thm:structure}
Suppose $b^+(X) > 1$, and $X$ is $U(3)$ simple type. Then there is a finite set $\{K_i\} \subset  H^2(X; \Z)$ and $c_{i,j}\in\Q[\sqrt{3}]$ such that for any $w\in H^2(X;\Z)$, and $\Gamma,\Lambda\in H_2(X)$:
\begin{equation*}
\mathbb{D}_{X, w}({\Gamma_{(2)}+\Lambda_{(3)}}) = e^{\frac{Q(\Gamma)}{2}-Q(\Lambda)}\sum_{i, j} c_{i, j} \zeta^{w\cdot \left(\frac{K_i-K_j}{2}\right)}e^{\frac{\sqrt{3}}{2}(K_i+K_j)\cdot \Gamma+\frac{\sqrt{-3}}{2}(K_i-K_j)\cdot \Lambda} \label{eq:structurethmformula}
\end{equation*}
Each class $K_i$ is an integral lift of $w_2(X)$, and satisfies the following: if $\Sigma\subset X$ is a smoothly embedded surface of genus $g$ with $\Sigma\cdot \Sigma\geq 0$ and $[\Sigma]$ non-torsion, then 
\begin{equation}\label{eq:structureadjunction}
  2g-2 \geq
|\langle K_i, \Sigma\rangle |+[\Sigma]^2.
\end{equation}
\end{theorem}

\begin{remark}\label{rmk:simpletype}
	The authors suspect a stronger statement holds: namely, that if $(X,w)$ is $U(3)$ simple type for any single $w\in H^2(X;\Z)$, then $X$ is $U(3)$ simple type. However, it appears that to adapt the proof to address such a statement requires equality in Proposition \ref{prop:evinclusion}, and does not follow from the partial description of eigenvalues given in Theorem \ref{thm:mainev}.
\end{remark}

\begin{remark}
Kronheimer and Mrowka used an adjunction inequality in the $U(2)$ setting \cite{km-embedded-i} to prove the Milnor conjecture on the slice genus of torus knots. This motivated the introduction of concordance invariants constructed from versions of $U(2)$ instanton Floer theory for knots \cite{KM:YAFT,km-rasmussen}. The $U(3)$ adjunction inequality in Theorem \ref{thm:structure} implies the Milnor conjecture in a similar way, and it would be interesting to explore whether there are similar concordance invariants that can be defined using $U(3)$ instantons.
\end{remark}

Our proof of Theorem \ref{thm:structure} is largely an adaptation of Mu\~{n}oz's proof in the $N=2$ case \cite{munoz-basic}, which uses $U(2)$ Fukaya--Floer homology. A key ingredient is Theorem \ref{thm:fukayafloersimpletypeideal}, regarding the $U(3)$ Fukaya--Floer analogue of the simple-type ideal \eqref{eq:simpletypeideal}.
We being with the following $N=3$ analogue of Lemma 11 from \cite{munoz-basic}.

\begin{lemma}\label{lemma:structuremainlemma}
Suppose $X$ satisfies $b^+(X)>1$ and is $U(3)$ simple type. Fix $w\in H^2(X;\Z)$. Let $\Sigma\subset X$ be a surface of genus $g$ with $[\Sigma]^2=0$ and $\Sigma\cdot w =d \not\equiv 0\pmod{3}$. Then there are $h_{a,b}\in \C[\![t_2,t_3]\!]$ such that for all $\Gamma,\Lambda\in H_2(X)$ and $l\in \Z$, we have:
\begin{align}
	\mathbb{D}_{X,w+l\Sigma}( & s_2\Sigma_{(2)} + s_3\Sigma_{(3)}+  t_2\Gamma_{(2)}+t_3\Lambda_{(3)}) \label{eq:mainstructurelemma}\\[3mm]
		& =e^{Q(s_2\Sigma + t_2\Gamma)/2 - Q(s_3\Sigma+ t_3\Gamma)} \sum_{(a,b)\in \mathcal{C}_g} \zeta^{lb} h_{a,b} e^{\sqrt{3} a s_2 + \sqrt{-3}bs_3 }\nonumber
\end{align}
\end{lemma}

\begin{proof}
We may suppose $\Gamma$ and $\Lambda$ are represented by surfaces which intersect $\Sigma$ transversely in a single point, and $\Gamma\cdot \Sigma=\Lambda\cdot \Sigma=1$. The general case follows from this case and linearity of the resulting expression. Identify a regular neighborhood of $\Sigma\subset X$ with $D^2\times \Sigma$, and write $X=X^\circ\cup D^2\times \Sigma$. Write $\Gamma=\Gamma^\circ\cup D$ and similarly for $\Lambda$, where $D$ and $\delta$ both denote $D^2\times \{pt\}$ in $\Delta=D^2\times \Sigma$. As in \eqref{pairing-Psi}, the gluing formula \eqref{pairing-FFH} gives
\begin{gather}\label{eq:gluingstep1}
D_{X,w}(ze^{t_2\Gamma_{(2)} + t_3 \Lambda_{(3)}}) = \langle D_{X^\circ,w^\circ}(ze^{t_2\Gamma^\circ_{(2)} + t_3 \Lambda^\circ_{(3)}}), D_{\Delta,\delta}(e^{ t_2 D_{(2)} + t_3 D_{(3)}})\rangle
\end{gather}
for all $z\in\bA^3(X,\Sigma)$, where we define $[\Sigma]^\perp = \{y\in H_2(X) | y\cdot\Sigma=0\}$, and
\[
	\bA^3(X,\Sigma):= \left(\text{Sym}^\ast(H_0(X)\oplus [\Sigma]^\perp)\otimes \Lambda^\ast H_1(X)\right)^{\otimes 2}\subset \bA^3(X).
\] 
The two invariants appearing on the right side of \eqref{eq:gluingstep1} are elements of the Fukaya--Floer homology $\mathbb{V}_{g,d}^3$.
Now let $s_2,s_3$ be formal variables and set 
\begin{equation}\label{eq:zsetinstructurelemma1}
	z=(1+\frac{1}{3} x_{(2)}+\frac{1}{9} x_{(2)}^2)e^{s_2\Sigma_{(2)} + s_3\Sigma_{(3)}}.
\end{equation}
Then the left side of \eqref{eq:gluingstep1} is equal to the left side of \eqref{eq:mainstructurelemma} when $l=0$. By the simple type assumption and the gluing formula, we have \[
	D_{X^\circ,w^\circ}(ze^{t_2\Gamma_{(2)}+t_3\Lambda_{(3)}})\in\mathbb{S}^{3}_{g, d}\otimes_\C \C[\![s_2,s_3]\!].
\]
Furthermore, by Theorem \ref{thm:fukayafloersimpletypeideal} we can write
\[
D_{X^\circ,w^\circ}(ze^{t_2\Gamma_{(2)}+t_3\Lambda_{(3)}}) =  \sum_{ (a,b)\in \mathcal{C}_g} f^w_{a,b}
\]
where $f^w_{a,b}\in R_{0,a,b}\otimes_\C \C[\![s_2,s_3]\!]$. (The presence of $1+x_{(2)}/3+x_{(2)}^2/9$ in $z$ implies $k=0$ in \eqref{eq:simpletypeidealdecompfukayafloer}.) From the description of $R_{0,a,b}$, $ f^w_{a,b}$ is a solution of the differential operator
\begin{equation}\label{eq:diffoperator}
	\left(\frac{\partial}{s_2}-(\sqrt{3} a + t_2)\right)\left(\frac{\partial}{s_3} -(\sqrt{-3} b - 2 t_3)\right).
\end{equation}
By the gluing formula we can then write
\[
	\mathbb{D}_{X,w}(s_2\Sigma_{(2)} + s_3\Sigma_{(3)}+ t_2\Gamma_{(2)}+t_3\Lambda_{(3)}) =  \sum_{ (a,b)\in \mathcal{C}_g} g^w_{a,b}
\]
where $ g^w_{a,b}\in \C[\![s_2,s_3,t_2,t_3]\!]$ is given by the pairing $\langle f^w_{a,b}, D_{\Delta,\delta}(e^{ t_2D_{(2)} + t_3 D_{(3)}})\rangle$. Furthermore, $g^w_{a,b}$ is also a solution of the operator \eqref{eq:diffoperator}. Thus we obtain
\[
	g^w_{a,b} = h^w_{a,b}(t_2,t_3)e^{\sqrt{3} a s_2 + s_2 t_2  + \sqrt{-3} bs_3 -  2s_3t_3 }.
\]
This proves the claim in the case $l=0$.

For the case of general $l$, first note the above argument carries through to show that
\[
	\mathbb{D}_{X,w+l\Sigma}(s_2\Sigma_{(2)} + s_3\Sigma_{(3)}+ t_2\Gamma_{(2)}+t_3\Lambda_{(3)}) =  \sum_{ (a,b)\in \mathcal{C}_g} h^{w+l\Sigma}_{a,b}e^{\sqrt{3} a s_2 + s_2 t_2  + \sqrt{-3} bs_3 -  2s_3t_3 }
\]
for some $h^{w+l\Sigma}_{a,b}\in \C[\![t_2,t_3]\!]$. Next, recall that there is a class $\varepsilon = \varepsilon(\Sigma)$ that acts on $\mathbb{V}_{g,d}^3$ as an operator of degree $-4d\pmod{4N}$, and can also be used via the gluing formula as a class when evaluating $D_{X,w}$. Namely, in the situation at hand, we have the relation
\[
	D_{X,w}(\varepsilon^l z) = D_{X,w+l\Sigma}( z) 
\]
for any $z\in\bA^3(X,\Sigma)$. The operator $\varepsilon$ restricted to $\mathbb{S}_{g,d}^3$ acts as follows:
\[
	\varepsilon:R_{k,a,b}\to R_{k,a,b} \text{ is multiplication by } \zeta^{b+dk}.
\]
This follows from the fact that the eigenvalues in \eqref{eq:epsilon1evs} have $\varepsilon=+1$, combined with the argument of Lemma \ref{lemma:evaction} (using that $\varepsilon$ has degree $-4d \pmod{4N}$). Replace $z$ in \eqref{eq:zsetinstructurelemma1} by
\[
	z=(1+\zeta^{i}\varepsilon +\zeta^{2i}\varepsilon^2)(1+\frac{1}{3} x_{(2)}+\frac{1}{9} x_{(2)}^2)e^{s_2\Sigma_{(2)} + s_3\Sigma_{(3)}}.
\]
From this substitution, carrying the above argument through, we obtain the relation
\begin{equation}\label{eq:diffbundlesstructurelemma}
	\sum_{l=0,1,2} \zeta^{li } h^{w+l\Sigma}_{a,b} = 0\;\; \text{  if  } \;\;  b +i \not  \equiv 0 \pmod{3}   
\end{equation}
for each $i\in \Z$; the key point is that the term $(1+\zeta^{i}\varepsilon +\zeta^{2i}\varepsilon^2)$ places the relative invariants in the $(\zeta^{-i})$-eigenspace of $\varepsilon$. The relations \eqref{eq:diffbundlesstructurelemma} are then used to solve
\[
	h_{a,b}^{w+l\Sigma} = \zeta^{lb}h^w_{a,b}. 
\]
This proves the claimed formula for general $l$, upon setting $h_{a,b}:=h^w_{a,b}$.
\end{proof}

\begin{proof}[Proof of Theorem \ref{thm:structure}]
	The proof runs parallel to Steps 2--5 of the proof of Theorem 2 from \cite{munoz-basic} for the $N=2$ case; we begin with an analogue of Step 2. Fix $X$ a smooth closed oriented $4$-manifold with $b^+(X)>1$ and of $U(3)$ simple type. We first show that for each $w\in H^2(X;\Z)$ there exists a finite set $\{K_i\}_{i\in I}\subset H^2(X;\Z)$ and $c_{i,j}^w\in \Q[\sqrt{3}]$ such that 
	\begin{equation}
\mathbb{D}_{X, w}({\Gamma_{(2)}+\Lambda_{(3)}}) = e^{\frac{Q(\Gamma)}{2}-Q(\Lambda)}\sum_{i, j} c^w_{i, j} e^{\frac{\sqrt{3}}{2}(K_i+K_j)\cdot \Gamma+\frac{\sqrt{-3}}{2}(K_i-K_j)\cdot \Lambda} \label{eq:structurethmformulainproof}
\end{equation}
The blow-up formulas of Theorem \ref{thm:blowup} show that it suffices to prove this for any blowup of $X$. For simplicity we assume $H_2(X;\Z)$ has no torsion (in the general case, mod out by torsion in the argument). After possibly blowing up, we may assume (using $b^+(X)>1$) that the intersection form of $X$ can be diagonalized, $Q=(+1)^r\oplus (-1)^s$. Let $A_1,\ldots,A_r,B_1,\ldots,B_s$ be a corresponding basis, so that $A_k^2=1$, $B_k^2=-1$, $A_k\cdot B_k = 0$. Define
\begin{align*}
	\Sigma^{1} &= A_2-B_1, & \Sigma^{k} &= -A_k-B_1\quad (2\leq k\leq r),\\[1mm]
	\Sigma^{{r+1}} &= A_1-B_2, & \Sigma^{{r+k}} &= \phantom{-}A_1+B_k \quad (2\leq k\leq s),
\end{align*}
and also define $w=A_1+B_1$. Writing $n=r+s$, we obtain a full rank subgroup $H = \langle \Sigma^{1},\ldots, \Sigma^{n}\rangle $ of $H_2(X;\Z)$ such that $2H_2(X;\Z)$ is contained in $H$. We also have
\[
	\Sigma^{k}\cdot \Sigma^{k}=0, \qquad w\cdot \Sigma^{k}= 1.
\]
Represent each $\Sigma_k$ by a connected oriented surface of genus $g_k$. Then iterating the argument of Lemma \ref{lemma:structuremainlemma}, we obtain the following:
	\begin{align*}
	\mathbb{D}_{X,w}( & \sum t_{2,k}\Sigma^{k}_{(2)} + \sum t_{3,k}\Sigma^{k}_{(3)}) \\[4mm]
	& =e^{Q(\sum t_{2,k}\Sigma^{k} )/2 - Q(\sum t_{3,k}\Sigma^{k} )} \sum_{\substack{1\leq k \leq n\\(a_k,b_k)\in \mathcal{C}_{g_k}}} h^w_{a_1,b_1,\ldots,a_n,b_n}  e^{\sqrt{3} \sum a_k t_{2,k} + \sqrt{-3}\sum b_k t_{3,k} }\nonumber
\end{align*}
where each unindexed sum runs from $k=1$ to $k=n$, and where $h^w_{a_1,b_1,\ldots,a_n,b_n}\in \C$. Here $t_{j,k}$ are formal variables, for $j=2,3$ and $1\leq k \leq n$. Now let $\Gamma,\Lambda\in H_2(X)$ be arbitrary. We may write $\Gamma= \sum x_k \Sigma^{k}$ and $\Lambda = \sum y_k \Sigma^{k}$ for some complex numbers $x_k,y_k$. Then specializing each $t_{2,k}$ to $x_kt_2$ and each $t_{3,k}$ to $y_kt_3$ gives the following expression:
	\begin{align*}
	\mathbb{D}_{X,w}( &t_2\Gamma_{(2)} +  t_3\Lambda_{(3)}) \\[4mm]
	& =e^{Q(t_2\Gamma  )/2 - Q(t_3 \Lambda  )} \sum_{\substack{1\leq k \leq n\\(a_k,b_k)\in \mathcal{C}_{g_k}}} h^w_{a_1,b_1,\ldots,a_n,b_n}  e^{t_2 \sqrt{3} \sum a_kx_k + t_3 \sqrt{-3}\sum b_k y_k }\nonumber
\end{align*}
Write $\Sigma^k_\star\in H_2(X)$ for the dual basis of the $\Sigma^k$ under the intersection pairing, so that $\Sigma^k_\star\cdot \Sigma^l = \delta_{kl}$. Note $2\Sigma^k_\star\in H_2(X;\Z)$. Define
\begin{align}
	I & :=\{i = (i_1,\ldots,i_n)\in \Z^n\; \mid\;   |i_k|<g_k \}, \nonumber \\[2mm]
	K_i & := 2\sum_{k=1}^n i_k\Sigma^k_\star \in H_2(X;\Z) \;\; \; \text{ for } \quad i\in I. \label{eq:defnofki}
\end{align}
Then we obtain \eqref{eq:structurethmformulainproof} (setting $t_2=t_3=1$) by letting $i,j$ range over $I$ and setting
\[
	c_{i,j}^w = h^w_{a_1,b_1,\ldots,a_n,b_n}
\]
where $i,j\in I$ are uniquely determined by $(a_1,b_1,\ldots,a_n,b_n)$, and conversely, through the relations $i_k+j_k = a_k$ and $i_k-j_k = b_k$ for all $1\leq k \leq n$.

Note that the argument of Lemma \ref{lemma:structuremainlemma} shows that the classes $K_i$ obtained above do not depend on $w$. Alternatively, without appealing to this point, one can take the union of the classes obtained for each $w$, where one ranges over one $w$ for each class in $H^2(X;\Z/3)$, to eliminate any a priori dependency. 

Next, we argue that the $K_i$ are integral lifts of $w_2(X)$, which amounts to showing $K_i\cdot x \equiv x^2\pmod{2}$ for all $x\in H_2(X;\Z)$. This is an adaptation of Step 3 in the proof of Theorem 2 from \cite{munoz-basic}. It is clear from the definition of $K_i$ in \eqref{eq:defnofki} that $\Sigma^k\cdot K_i\equiv 0\pmod{2}$, which agrees mod $2$ with $\Sigma^k\cdot \Sigma^k=0$, and this verifies the claim on $H\subset H_2(X;\Z)$. The general property used here in fact essentially follows from Lemma \ref{lemma:structuremainlemma}: if $\Sigma\subset X$ satisfies $\Sigma\cdot w\not\equiv 0$ and $\Sigma\cdot \Sigma=0$, then $\Sigma\cdot K_i$ is even for any of the $K_i$.

Now suppose $x\in H_2(X;\Z)\setminus H$. Then there is some $k$ for which $x\cdot \Sigma^k\neq 0$. We can find $m\in \Z$ such that $x':= x+ m\Sigma^k$ satisfies
\[
	N:= (x')^2 \geq 0, \qquad w\cdot x' \not\equiv 0 \pmod{3}.
\]
(Here the property $w\cdot \Sigma^k=1$ is used to obtain the second condition.) Now let $\widetilde{X}$ be $X$ blown up at $N$ points, and denote by $E_1,\ldots,E_N$ the associated exceptional divisors. It follows from the blowup formulas of Theorem \ref{thm:blowup} that if $\{K_i\}$ are the classes in \eqref{eq:structurethmformulainproof} for $X$, then classes associated to $\widetilde X$ are given by
\begin{equation}\label{eq:blopupclasses}
	K_i + \sum_{l=1}^N \varepsilon_l E_l , \qquad \varepsilon_l \in \{1,-1\}.
\end{equation}
Consider $y=x'-E_1-\cdots -E_N$. This satisfies $y^2=0$ and $y\cdot w\not\equiv 0\pmod{3}$. By the previous paragraph, we have that $y\cdot (K_i + \sum_{l=1}^N \varepsilon_l E_l)$ is even. On the other hand,
\[
	y\cdot (K_i + \sum_{l=1}^N \varepsilon_l E_l) \equiv x\cdot K_i + N \pmod{2}.  
\]
Since $x^2 \equiv (x')^2 = N \pmod{2}$, this proves the claim for $x$, and shows that each $K_i$ is indeed characteristic.

We next consider the analogue of Step 4 in the proof of Theorem 2 from \cite{munoz-basic}. The goal is to show, upon setting $c_{i,j}=c_{i,j}^0$, that we have the relation 
\begin{equation}\label{eq:coeffmatchlemma0}
	c^w_{i, j} = \zeta^{w\cdot\left(\frac{K_i - K_j}{2}\right) }c_{i,j}.
\end{equation}
We now suppose $w^2>0$, as the invariants only depend on the mod $3$ reduction of $w$, and every non-zero class in $H^2(X;\Z)$ is mod $3$ congruent to one with positive square. Consider again $\widetilde X$, the blowup of $X$ at $N:=w^2$ points, with exceptional divisors $E_1,\ldots,E_N$. By the blowup formula, the classes associated to $\widetilde X$ are as in \eqref{eq:blopupclasses}. Write
\[
	\widetilde K_i = K_i + \sum_{l=1}^N \varepsilon^i_l E_l, \qquad \widetilde K_j = K_j + \sum_{l=1}^N \varepsilon^j_l E_l
\]
for two such classes. Then the blowup formula gives
\begin{equation}\label{eq:coeffmatchlemma1}
	c^{E_1}_{\widetilde K_i , \, \widetilde K_j } = \frac{ q_1p_2\cdots p_N}{3^N2} c_{K_i,K_j}
\end{equation}
where the numbers $p_k$ and $q_k$ are defined as follows:
\[
	p_k = \begin{cases} 1/2, & \varepsilon_k^j = \varepsilon^i_k\\
								1, &  \varepsilon_k^j \neq  \varepsilon^i_k
								 \end{cases}
								 \hspace{1cm}
 q_k = \begin{cases} 1/2, & \varepsilon_k^j = \varepsilon^i_k\\
								\zeta^{-(\varepsilon_k^i-\varepsilon^j_k)/2}, &  \varepsilon_k^j \neq  \varepsilon^i_k
								 \end{cases}
\]
Consider $x= w-E_1-\cdots-E_N$. Note that $x^2=0$ and $x\cdot E_1\not\equiv 0\pmod{3}$. In this situation Lemma \ref{lemma:structuremainlemma} provides the relationship
\begin{equation}\label{eq:coeffmatchlemma2}
	c^{w-E_2-\cdots-E_N}_{\widetilde K_i , \, \widetilde K_j} = \zeta^{\frac{1}{2}x\cdot \left(\widetilde K_i  - \widetilde K_j \right)}c^{E_1}_{\widetilde K_i , \, \widetilde K_j} = \zeta^{\frac{1}{2}w\cdot (K_i-K_j) + \frac{1}{2}\sum_{l=2}^N (\varepsilon_l^i-\varepsilon_l^j)}c^{E_1}_{\widetilde K_i , \, \widetilde K_j}.
\end{equation}
On the other hand, another application of the blowup formula yields
\begin{equation}\label{eq:coeffmatchlemma3}
	c^{w-E_2-\cdots-E_N}_{\widetilde K_i , \, \widetilde K_j} = \frac{p_1\overline{q}_2\cdots \overline{q}_N}{3^N} c^w_{K_i,K_j}.
\end{equation}
Combining \eqref{eq:coeffmatchlemma1}--\eqref{eq:coeffmatchlemma3}, we obtain the desired relation \eqref{eq:coeffmatchlemma0}. 

The claim that $c_{i,j}\in\Q[\zeta]$ follows from \eqref{eq:structurethmformulainproof} and the fact that the invariants $D_{X,w}$ output rational values. These same observations also imply that $c_{i,j}$ is the complex conjugate of $c_{j,i}$. Furthermore, there is the general property
\begin{equation}\label{eq:conjugationsymmetry}
	D_{X,-w}(z) = D_{X,w}(\tau(z))
\end{equation}
where $z\in \bA^3(X)$ and $\tau:\bA^3(X)\to \bA^3(X)$ is the algebra homomorphism which maps $\alpha_{(r)}$ to $(-1)^r\alpha_{(r)}$; see \cite[2.10]{DX}. Taking $w=0$, relations \eqref{eq:conjugationsymmetry} and \eqref{eq:structurethmformulainproof} yield $c_{i,j}=c_{j,i}$. We conclude that $c_{i,j}$ is real and hence $c_{i,j}\in\Q[\sqrt{3}]$.

Finally we consider the adjunction inequality \eqref{eq:structureadjunction}. The proof of Step 5 in the proof of Theorem 2 from \cite{munoz-basic} carries over nearly verbatim. An argument in \cite{km-structure} reduces the proof to the case in which $N:=\Sigma\cdot \Sigma>0$. Consider again $\widetilde X$, the blowup of $X$ at $N$ points, and the proper transform $\widetilde \Sigma\subset \widetilde X$ of $\Sigma$, which represents the class $\Sigma-E_1-\cdots -E_N$. As $\widetilde \Sigma\cdot \widetilde \Sigma=0$ and $\widetilde\Sigma\cdot w\not\equiv 0\pmod{3}$ for $w=E_1$, Lemma \ref{lemma:structuremainlemma} yields 
\[
	2g-2 \geq | \left(K_i + \sum \varepsilon_l E_l \right)\cdot \left( \Sigma-E_1-\cdots -E_N\right) |
\]
for all of the associated classes $K_i$ of $X$, and all $\varepsilon\in \{1,-1\}^N$. This implies the desired inequality \eqref{eq:structureadjunction}, and completes the proof of the theorem.
\end{proof}


\section{A non-vanishing theorem for symplectic 4-manifolds}\label{non-vanishing-symp} 

In this section, we prove Theorem \ref{thm-intronon-vanishing-symplectic} of the introduction, which we restate here:

\begin{theorem}\label{non-vanishing-symplectic}
Let $X$ be a closed symplectic 4-manifold with $b^+(X) > 1$. Then the invariant $D_{X,w}^3$ is non-trivial for all $w\in H^2(X;\Z)$.
\end{theorem}

As a consequence of Theorem \ref{non-vanishing-symplectic}, we have the following non-vanishing result for admissible bundles. In the same way that we deduce Theorem \ref{thm:intro2} from Corollary \ref{non-vanishing-3-man}, the following non-vanishing result can be used to give another proof of Theorem \ref{thm:intro2}.

\begin{cor}\label{non-vanishing-admissible-pair}
	Let $(Y,\gamma)$ be an admissible pair such that $Y$ is irreducible. Then the instanton Floer homology group 
	$I^3_*(Y, \gamma)$ is non-trivial. 
\end{cor}
\begin{proof}
	The corollary is a consequence of Theorem \ref{non-vanishing-symplectic} and a result about embeddings of 3-manifolds  
	into symplectic manifolds. Since $(Y,\gamma)$ is admissible, we have $b_1(Y)>0$. In particular, $Y$ can be embedded in a symplectic manifold 
	$X$ as a separating submanifold such that the map $H^2(X;\Z)\to H^2(Y;\Z)$ is surjective and the two components $X_1$ and $X_2$
	obtained by cutting $X$ along $Y$ have $b^+>0$. This follows from Gabai's theorem about the existence of taut foliations on 
	3-manifolds with $b_1>0$ \cite{Gab:fol-sut} and \cite[Proposition 15]{KM:prop-P}. 
	The latter is obtained by combining various earlier results \cite{El:symp-filling,Et:symp-filling,ET:confoliations,KM:prop-P} (see also
	\cite[Section 41.3]{km:monopole}). Our control on $H^2(X;\Z)$ implies that there is a $2$-cycle $w$ on $X$ whose intersection with $Y$ is homologous to 
	$\gamma$. Using Theorem \ref{non-vanishing-symplectic}, we know
	\[
	  \mathbb{D}^3_{X, w}({\Gamma_{(2)}+\Lambda_{(3)}})
	\]
	is non-trivial for some $\Gamma,\Lambda\in H_2(X;\Z)$ where $\Gamma$ and $\Lambda$ are represented by embedded surfaces whose intersection with 
	$Y$ are respectively equal to $c$ and $l$. Now we can use the pairing formula \eqref{pairing-FFH} to see that the Fukaya--Floer homology group 
	$\mathbb I^3_*(Y, \gamma,L)$ with $L$ given by the homology classes of $c$ and $l$ is non-trivial. The non-vanishing of this Fukaya--Floer homology group implies that $\mathbb I^{3,j}_*(Y, \gamma,L)$ is non-zero for some $j$. The spectral sequence from $I_\ast(Y,\gamma)\otimes \C^{2^{2j}}$ to $\mathbb I^{3,j}_*(Y, \gamma,L)$ implies that $I^3_*(Y, \gamma)$
	is non-zero.
\end{proof}

\begin{proof}[Proof of Theorem \ref{non-vanishing-symplectic}]
	After possibly perturbing the symplectic form of $X$ and then rescaling, we can assume that the symplectic form $\omega$ of $X$ represents an integral 
	cohomology class. Now \cite[Theorem 2]{Don-Lefs-pen-symp} implies that $X$ admits a (topological) Lefschetz pencil such that the 	
	fibers are symplectic subvarieties representing the Ponicar\'e dual of $k[\omega]$ where $k$ is a large enough integer. 
	In particular, the base locus of this Lefschetz pencil is given by a non-empty set of points $\{x_1,\cdots,x_m\}$, and by blowing up
	$X$ at these points, we obtain $\widehat X$, which is a Lefschetz fibration over $S^2$
	where a generic fiber $F$ (obtained as the 
	proper transform of a fiber of the Lefschetz pencil) represents the cohomology class
	\[
	  k[\omega]-E_1-\cdots -E_m
	\]  
	with $E_i$ the exceptional classes. Taking $k$ large enough, we may also assume that the genus of $F$ is as 
	large as we wish and all fibers of the Lefschetz fibration are irreducible. The latter claim is \cite[Theorem 3.10]{Smith:Lef-pen} and 
	the former follows from adjunction formula. (See, for example, (3.9) in \cite{Smith:Lef-pen}.)
	The 2-cycle $w$ in $X$ induces a cycle in $\widehat X$ and if necessary we add $\pm E_1$
	 to this cycle to guarantee that the resulting cycle $\widehat w$ in $\widehat X$ satisfies $\widehat w \cdot F \equiv 1$ mod $3$.
	 By Theorem \ref{thm:blowup}, if we show that $D^3_{\widehat X,\widehat w}$ is non-trivial, then $D^3_{X,w}$
	 is also non-trivial. 

	We decompose the base $S^2$ of the Lefschetz fibration structure on $\widehat X$ as a union
	$D_- \cup A_1 \cup\cdots A_l \cup D_{+}$ such that $D_{\pm}$ are discs and the $A_i$ are annuli, and the Lefschetz fibration 
	has no critical point over the discs and exactly one critical point over each annulus. This induces a decomposition of $\widehat X$
	as follows:
	\[
	  \widehat{X}=D_-\times F \cup W_1 \cup \cdots \cup W_l \cup D_+\times F
	\]
	where $W_i:Y_{i-1}\to Y_i$ is a cobordism that admits a Lefschetz fibration over the annulus $A_i$ with one singular fiber. In particular, $Y_i$ fibers 
	over $S^1$ with fiber $F$. We regard $D_-\times F$ as a cobordism from the empty set to $Y_0=S^1\times F$ and 
	$D_+\times F$ as a cobordism from $Y_l=S^1\times F$ to the empty set. Without loss of generality, we can 
	assume that the intersection of $\widehat w$ and $Y_i$ is transversal and we write $\gamma_i$ for the induced $1$-cycle on $Y_i$.
	We also denote the intersection of $\widehat w$ with $D_{\pm}\times F$ and $W_i$ by $w_\pm$ and $w_i$.
	Our assumption on $\widehat w \cdot F$ implies that $(Y_i,\gamma_i)$ is $3$-admissible. (In the special case of
	$\gamma_{0}$ and $\gamma_{l}$, they are given by circle fibers of $Y_0$ and $Y_l$.) The 3-manifolds $Y_i$ can be regarded 
	as a closure of the product sutured manifold. In particular, we can use Theorem \ref{thm:intro-SHI-invariance} to see that 
	$I_*^3(Y_i,\gamma_i\vert F)=\C$. 
	
	As the Floer homology of $(S^1\times F,\gamma_1)$ is generated by relative invariants of $D^2\times F$, there exist polynomials $p_{\pm}\in \C[x,y,z]$ such that 
	\begin{align}
	  D_{D_-\times F, w_-}^3(p_-(x_{(2)},F_{(2)},F_{(3)}))&\in I_*^3(Y_0,\gamma_0\vert F), \label{eq:cobinvtsympprf1}\\[2mm]
	  D_{D_+\times F,w_+}^3(p_+(x_{(2)},F_{(2)},F_{(3)}))&:I_*^3(Y_l,\gamma_l\vert F)\to \C \label{eq:cobinvtsympprf2}
	\end{align}
	are non-trivial. The gluing formula expresses the invariant
	\[
		D_{\widehat{X},\widehat{w}}^3(p_+p_-(x_{(2)},F_{(2)},F_{(3)}))
	\]
	in terms of a composition of the two quantities \eqref{eq:cobinvtsympprf1}, \eqref{eq:cobinvtsympprf2} and the maps
		\begin{align}
	  I_*^3(W_l, w_l) &: I_*^3(Y_{l-1},\gamma_{l-1}\vert F)\to I_*^3(Y_l,\gamma_l\vert F). \label{eq:sympproofmiddlemap}
	\end{align}
	Thus to prove our claim, it suffices to show that \eqref{eq:sympproofmiddlemap} is non-zero.
	
	The cobordism $W_l:Y_{l-1}\to Y_l$ can be decomposed further as
	the composition of the following two 4-dimensional cobordisms:
	\[
	  \mathcal L_l: \emptyset \to Y_{\phi_l}, \hspace{1cm}
	  \mathcal P_l: Y_{l-1} \sqcup Y_{\phi_l} \to Y_{l},
	\]
	where $Y_{\delta_l}$ is the mapping torus of a positive Dehn twist along a non-separating simple closed curve in $F$. 
	Here we are using the fact that 
	the fibers of our Lefschetz fibration are irreducible. The positive Dehn twist $\delta_l$ is determined by the property that if  $Y_{i}$ is the 
	mapping torus of the diffeomorphism $\phi_{i}:F\to F$, then $\phi_l=\delta_l\circ \phi_{l-1}$. The cycle $w_l$ induces the cycles $c_l$
	and $c_{l}'$ on $\mathcal L_l$ and $\mathcal P_l$. We also write $\epsilon_l$ for the induced cycle on $Y_{\delta_l}$.
	The excision theorem of \cite{DX} implies that the cobordism map 
	\begin{equation}\label{excision-iso}
	  I_*^3(\mathcal P_l,c_l'):I_*^3(Y_{l-1},\gamma_{l-1}\vert F)\otimes I_*^3(Y_{\delta_l},\epsilon_l \vert F) \to I_*^3(Y_l,\gamma_l\vert F)
	\end{equation}
	is an isomorphism of $1$-dimensional vector spaces. The following lemma and the non-vanishing of \eqref{excision-iso} implies the non-vanishing of \eqref{eq:sympproofmiddlemap}, completing the proof.
\end{proof}

\begin{figure}[t]
\centering
\includegraphics[scale=1.05]{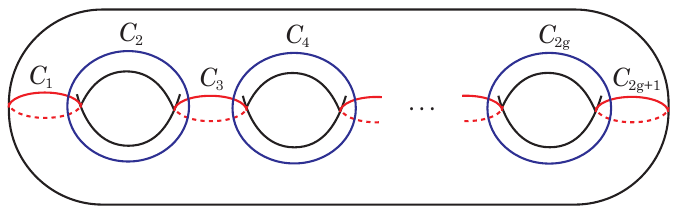}
\caption{The curves $C_i$ for $1\leq i \leq 2g+1$ on a surface of genus $g$.}
\label{fig:dehn-curves}
\end{figure}

 \begin{lemma}
	For an oriented closed surface $F$, let $\mathcal L$ be a Lefschetz fibration over the 2-dimensional disc with one irreducible
	singular fiber. Let $c$ be a 2-cycle on $\mathcal L$ such that $c\cdot F \equiv 1$ mod $3$. 
	Then $I_*^3({\mathcal L, c})$ has a non-trivial component in $I_*^3(Y,\gamma\vert F)$ 
	where $(Y,\gamma)$ is the boundary of $(\mathcal L, c)$.
 \end{lemma}
 
 \begin{proof}
	The 4-manifold $\mathcal L$ can be embedded into any closed 4-manifold $X$ together with the structure of a genus $g=g(F)$
	Lefschetz fibration, which has at least one irreducible singular fiber. For instance, we can take the elliptic surface $X=E(g+1)$, 
	which has a genus $g$ Lefschetz fibration in addition to its standard elliptic fibration.
	In fact, we may take a Lefschetz fibration with the singular fibers given by the monodromies
	\[
	  (\psi_{1},\psi_2,\cdots,\psi_{2g-1}, \psi_{2g-1}, \cdots, \psi_{2},\psi_{1})^{4}
	\]
	where $\psi_i$ denotes the Dehn twist along the simple closed curve $C_i$ in Figure \ref{fig:dehn-curves}. 
	See \cite[Chapter 8]{GS:kirby-calc}. We still denote the fibers of this Lefschetz fibration by $F$. Then the algebraic intersection of $F$
	and a fiber $f$ of the elliptic fibration of $E(g+1)$ is equal to $2$.
	
	The 4-manifold $\mathcal L$ can be identified with a regular neighborhood of the singular fiber corresponding to $\psi_1$.
	The complement of $\mathcal L$ determines another Lefschetz fibration $\mathcal Z$ over a disc. 
	This manifold has the homotopy type of $F$ where we glue 
	2-cells to it along simple close curves corresponding to the Dehn twists involved in the Lefschetz fibration structure of $\mathcal Z$.
	In particular, $H_1(\mathcal Z)$ is trivial, and hence there is a 2-cycle $c'$ on $\mathcal Z$ whose restriction to $Y$ is $\gamma$. 
	The 2-cycles $c$ and $c'$ can be glued to each other to form a 2-cycle $\tilde c$ on $E(g+1)$.
	We may further split $(\mathcal Z,c)$ as the composition of $(Z_0,c_0)$ and $(Z_1,c_1)$ such that $Z_1$ is a regular neighborhood of 
	regular fiber of the Lefschetz fibration of $\mathcal Z$ and $Z_1$ is the complement. In particular, $Z_1$ is diffeomorphic to $D^2\times F$
	and we can assume that $c_1=D^2\times \{x\}$ for $x\in F$. The pair $(Z_0,c_0)$ can be regarded as a cobordism from $(Y,\gamma)$
	to $(S^1\times F,\gamma_1)$. To prove the claim, it suffices to show that there is a polynomial $Q\in \C[x,y]$ such that 
	\begin{equation}\label{relative-inv}
	  \widehat{D}^3_{Z_0,c_0}(Q(F_{(2)},F_{(3)})) \circ I_*^3(\mathcal L, c) 
	\end{equation}
	is a non-zero element of $I_*^3(S^1\times F,\gamma_1|F)$. Then functoriality implies that 
	$I_*^3(\mathcal L, c) $ has a non-trivial component in $I_*^3(Y,\gamma\vert F)$.
	
	The polynomial $p$ can be constructed as in \cite[Proposition 5.7]{DX} using the calculation of 
	$U(3)$ invariants of elliptic surfaces in \cite{DX}. For the pair $(E(g+1), w)$, we have 
	  \begin{equation*}\label{U(3)-series-Elliptic}
		\mathbb{D}^3_{E(g+1), w}({t_2F_{(2)}+t_3F_{(3)}}) = \(\frac{2}{3}\cosh(2\sqrt 3 t_2)-
		\frac{2}{3} \cosh(-\frac{2\pi i}{3}w \cdot f+2\sqrt 3 i t_3 )\)^{g-1}.
	\end{equation*}	
	In particular, it is a power series of the form 
	\begin{equation}\label{power-series-rewrite}
		\mathbb{D}^3_{E(g+1), w}({t_2F_{(2)}+t_3F_{(3)}})=\sum_{a,b} d^w_{a,b} e^{2\sqrt 3 at_2+2\sqrt 3 i bt_3}
	\end{equation}
	where $d^w_{a,b}$ are constant coefficients, $a+b$ has the same parity as $g-1$ and $|a|+|b|\leq g-1$. For instance, we have $d^w_{g-1,0}=(2/3)^{g-1}$.
	We may use the above identities to compute $\widehat{D}_{X,w}(P(F_{(2)},F_{(3)}))$ for any polynomial $P$. In fact, using the notation in 
	\eqref{power-series-rewrite} we have
	\begin{equation}\label{power-poly-rewrite}
		\widehat{D}_{E(g+1),w}^3(P(F_{(2)},F_{(3)}))=\sum_{a,b} d^w_{a,b} P(a,b).
	\end{equation}	
	Now let $Q$ be a polynomial such that $Q(g-1,0)=1$, and $Q(a,b)=0$ for any other $(a,b)$ as above. Using \eqref{power-poly-rewrite} 
	and the fact that $E(n)$ with $n\geq 2$ is $U(3)$ simple type, we have
	\begin{equation}\label{power-poly-rewrite}
		\widehat{D}_{E(g+1),w}^3(R(x_{(2)},x_{(3)},F_{(2)},F_{(3)})Q(F_{(2)},F_{(3)}))=\left(\frac{2}{3}\right)^{g-1}R(3,0,g-1,0)
	\end{equation}	
	for any polynomial $R\in \C[v,w,x,y]$. We claim that this polynomial $Q$ satisfies the required property for \eqref{relative-inv}.

	Gluing $(Z_0,c_0)$ and $(\mathcal L, c)$ produces a pair $(\widetilde Z_0,\widetilde c_0)$ with boundary $(S^1\times F,\gamma_1)$, and by functoriality \eqref{relative-inv} is equal to 	the following:
	\begin{equation}\label{relative-inv-2}
	  \widehat{D}^3_{\widetilde Z_0,\widetilde c_0}(Q(F_{(2)},F_{(3)})).
	\end{equation}	
	Using the functoriality of instanton Floer homology again, we see that 
	\begin{align*}
	\langle \widehat D^3_{\widetilde Z_0,\widetilde c_0}( Q(F_{(2)},F_{(3)})), \;\;
	& D^3_{Z_1,c_1+l\cdot F}(R(x_{(2)},x_{(3)},F_{(2)},F_{(3)}))\rangle \hspace{3cm}\\[2mm]
	=& \widehat{D}^3_{E(g+1),\widetilde c+l\cdot F}(R(x_{(2)},x_{(3)},F_{(2)},F_{(3)})Q(F_{(2)},F_{(3)})).
	\end{align*}
	Since instanton Floer homology of $(S^1\times F,\gamma_1)$ is generated by elements of the form 
	$D^3_{Z_1,c_1+l\cdot F}(R(x_{(2)},x_{(3)},F_{(2)},F_{(3)}))$, we may combine the above pairing formula and \eqref{power-poly-rewrite}
	to see that \eqref{relative-inv-2} is non-zero and belongs to the eigenspace $I_*^3(S^1\times F,\gamma_1|F)$.
\end{proof}

	
\section{$U(N)$ framed instanton homology}	\label{sec:framed}

	In this section we study the $U(N)$ framed instanton homology for $3$-manifolds. These groups were essentially introduced by Kronheimer and Mrowka \cite{KM:YAFT}, and have been extensively studied in the $N=2$ case, see for example \cite{KM:unknot,scadutothesis,bs-lspace}. After establishing some basic properties of $U(N)$ framed instanton homology, we compute its Euler characteristic and state a connected sum theorem. In the final subsection we discuss a decomposition result for cobordism maps in the $N=3$ case, which follows from an adaptation of the $U(3)$ Structure Theorem in this setting.

\subsection{Definition}
	
	Let $Y$ be a closed, oriented, connected $3$-manifold. Delete a small embedded open $3$-ball from $Y$ to obtain $M$, which has a $2$-sphere boundary. Let $\alpha$ be any simple closed curve on the boundary of $M$. Define the {\emph{$U(3)$ framed instanton homology of $Y$}} as follows:
	\[
		I^{\#,3}_\ast(Y) = SHI_\ast^3(M,\alpha).
	\]
	More concretely, the $U(3)$ framed instanton homology is given as
	\[
		I^{\#,3}_\ast(Y) = I^3_\ast(Y\# T^3, \gamma | R )
	\]
	where $\gamma$ is the $1$-cycle $S^1\times \{x\}$ in $T^3=S^1\times T^2$ and $R=\{y\}\times T^2$, where $x\in T^2$ and $y \in S^1$. As $\mu_2(R)=\mu_3(R)=\beta_3=0$ on the group $I^3_\ast(Y\# T^3)$, it follows that $I^{\#,3}_\ast(Y)$ is defined as the $(3)$-eigenspace of $\beta_2$ acting on $I^3_\ast(Y\# T^3)$. Note that we have already encountered these groups at the end of Section \ref{sec:sutured}.
	
	More generally, we define the $U(N)$ framed instanton homology for $N\geq 2$ as
		\begin{equation}\label{eq:framedun}
		I^{\#,N}_\ast(Y) = I^N_\ast(Y\# T^3, \gamma | R)
	\end{equation}
	where the notation on the right-side denotes the $(N)$-eigenspace of the operator $\beta_2$ acting on $I^N_\ast(Y\# T^3, \gamma )$.
	This construction is due to Kronheimer and Mrowka, see \cite[\S 4.1]{KM:YAFT}. The $N=2$ version has been studied in various settings (see e.g. \cite{scadutothesis}), often motivated by Kronheimer and Mrowka's conjecture \cite[\S 7.9]{km-sutures} that $I^{\#,2}(Y)$ is isomorphic to Ozsv\'{a}th and Szab\'{o}'s Heegaard Floer group $\widehat{HF}(Y)$ with complex coefficients.

\begin{remark}
	Note that definition \eqref{eq:framedun} allows one to use any coefficient ring when defining $U(N)$ framed instanton homology. In what follows, we will continue to assume that complex coefficients are used.
\end{remark}
	
	As $I^N_\ast(Y\# T^3, \gamma )$ is relatively $\Z/4N$-graded and $\beta_2$ has degree $4$, the group \eqref{eq:framedun} inherits a relative $\Z/4$-grading. This can be lifted to an absolute $\Z/4$-grading, just as in the $N=2$ case; the discussion in \cite[\S 7.3]{scadutothesis} adapts in a straightforward manner. 
	
	For our purposes, we only need to specify an absolute $\Z/2$-grading on $I^{\#,N}(Y)$. To do this, it suffices to define an absolute $\Z/2$-grading on $I^N(Y,\omega)$ for any $N$-admissible pair $(Y,\omega)$. For a critical point $\rho$ which is a generator of the complex defining $I^N_\ast(Y,\omega )$, set
	\begin{equation}\label{eq:mod2grdef}
		\text{gr}(\rho) = \text{ind}(A) + (N^2-1)(b_1(X)-b^+(X)+b_1(Y)-1) \pmod{2}
	\end{equation}
	where $\text{ind}(A)$ is the index of $D_A$, the (perturbed) ASD operator associated to a $PU(N)$-connection $A$ over a $4$-manifold $X$ with cylindrical end $Y\times [0,\infty)$, with $A$ restricting to the pullback of $\rho$ over $Y\times [0,\infty)$. This is well-defined by an argument analogous to the one given in \cite[\S 5.6]{donaldson-book},  using the index formulas found in \cite{kronheimer-higher}, for example. 
	
	Note from the construction of framed instanton homology that
	\[
		\dim I^{\#,N}(Y) = \frac{1}{N}\dim I^N_\ast(Y\# T^3, \gamma ).
	\]	
	Consider $Y=S^3$. As $I^N_\ast( T^3, \gamma )$ is of dimension $N$, generated by $N$ non-degenerate flat connections (see \cite{kronheimer-higher,KM:YAFT}), the dimension of $I^{\#,N}(S^3)$ is $1$.
	
	Note that the $(3)$-eigenspace of $\beta_2$ acting on $I_\ast(Y\# T^3)$ agrees with the $(1)$-eigenspace of the operator $\varepsilon$. The action of $\varepsilon$ can also be viewed as the action of a certain $PU(N)$-gauge transformation supported on the $T^3$-factor. From this viewpoint, which will be adapted below, $I^\#(Y,\gamma)$ is the Morse homology of a (perturbed) Chern--Simons functional on a configuration space of $PU(N)$-connections which is quotiented by a slightly larger gauge group, and its critical set (in the unpertubed case) is homeomorphic to
	\[
		R^N(Y):= \text{Hom}(\pi_1(Y),SU(N)),
	\]
	the $SU(N)$ representation space of $\pi_1(Y)$. Note that the quotient of $R^N(Y)$ by the action of conjugation, denoted $\mathfrak{X}^N(Y)$, is the $SU(N)$ character variety of $Y$.
	
\subsection{Euler characteristic}	

	In \cite{scadutothesis} it was shown that $I^{\#,2}(Y)$ has Euler characteristic equal to $|H_1(Y;\Z)|$ if $b_1(Y)=0$, and is otherwise zero. This fact generalizes as follows.
	
	\begin{theorem}\label{thm:eulercharframed}
		For any $N\geq 2$ and any closed, oriented, connected $3$-manifold $Y$:
		\begin{equation}\label{eq:framedeulercharthm}
			\chi\left( I^{\#,N}(Y) \right) = \begin{cases} |H_1(Y;\Z)|^{N-1}, & b_1(Y)=0 \\ 0, & b_1(Y)>0\end{cases}
		\end{equation}
	\end{theorem}
	
	\begin{proof}
		We first explain the proof under the assumption that $b_1(Y)=0$, $\mathfrak{X}^N(Y)$ is a finite set of non-degenerate points. In particular, $R^N(Y)$ is Morse--Bott nondegenerate for the Chern--Simons functional on $(Y\#T^3,\gamma)$. In particular, we have a homeomorphism of spaces
		\begin{equation}\label{eq:sunrepvarietyinproof}
			R^N(Y) \cong \bigsqcup_{[\rho]\in \mathfrak{X}^N(Y)} SU(N)/\Gamma_{\rho}
		\end{equation}
		where $\Gamma_\rho\subset SU(N)$ denotes the stabilizer of $\rho$ under the conjugation action. The stabilizer $\Gamma_\rho$ is isomorphic to a group of the form
		\[
			S(U(n_1)\times U(n_2)\times \cdots \times U(n_k))
		\]
		where $\sum n_i \leq N$. Let us say that $\rho$ is {\emph{abelian}} if $\sum n_i = N$. This terminology is justified by the fact that $\rho$ is abelian if and only if it factors through the abelianization $H_1(Y;\Z)$; an equivalent condition is that the stabilizer $\Gamma_\rho$ has the same rank as $SU(N)$. (Recall that the rank of a compact Lie group is the dimension of a maximal torus.)
		
		A abelian $SU(N)$ representation $\rho$ may be constructed by taking $N-1$ homomorphisms $\rho_i:H_1(Y;\Z)\to U(1)$ for $i=1,\ldots,N-1$ and composing 
		\[
			\rho_1\oplus \cdots \oplus \rho_{N-1}\oplus (\rho_1\cdots \rho_{N-1})^{-1}:H_1(Y;\Z)\to SU(N)
	\]
	 with the natural surjection $\pi_1(Y)\to H_1(Y;\Z)$. This constructs $|H_1(Y;\Z)|^{N-1}$ abelian representations; call these {\emph{standard}}. Every abelian representation is conjugate to a standard one, but some standard abelian representations are equivalent by conjugation. Conjugation induces an action of the Weyl group of $SU(N)$, the symmetric group $S_N$, on the set of standard abelian representations. The orbit-stabilizer formula for this $S_N$-action yields
	 \begin{equation}\label{eq:countstandardcompletelyreducibles}
	 	|H_1(Y;\Z)|^{N-1} = \sum_{[\rho]\in \mathfrak{R}^N(Y)} |W_{SU(N)}|/|W_{\Gamma_\rho}|
	 \end{equation}
	 where $\mathfrak{R}^N(Y)\subset \mathfrak{X}^N(Y)$ is the subset of abelian classes. The notation $W_G$ denotes the Weyl group of $G$. In writing this formula we have identified the stabilizer of $\rho$ under the $S_N=W_{SU(N)}$ action with the Weyl group of ${\Gamma_\rho}$. 		
		A result of Hopf and Samelson \cite{hopf-samelson} says that a connected homogeneous space $G/H$, where $G$ is a compact Lie group and $H$ is a closed subgroup, has Euler characteristic
		\begin{equation}\label{eq:hopfsamelson}
			\chi(G/H) = \begin{cases} |W_G|/|W_H|, & \text{rank}(G) = \text{rank}(H)\\[2mm]
														0, &  \text{rank}(G) > \text{rank}(H)\end{cases}
		\end{equation}
		Combining \eqref{eq:hopfsamelson}, \eqref{eq:sunrepvarietyinproof}, \eqref{eq:countstandardcompletelyreducibles}, and the earlier observation that $\rho$ is abelian if and only if the rank of $\Gamma_\rho$ is that of $SU(N)$, we obtain
		\[
			\chi(R^N(Y)) = |H_1(Y;\Z)|^{N-1}.
		\]
		A small perturbation used in defining $I^{\#,N}_\ast(Y\# T^3, \gamma )$ can be chosen so that the orbit of $\rho$ appearing in \eqref{eq:sunrepvarietyinproof} is replaced by the set of critical points $\{\alpha_i\}$ of a Morse function on $SU(N)/\Gamma_\rho$. The mod $2$ grading of $\alpha_i$ is given by its Morse index plus the mod $2$ grading of $\rho$ as defined by equation \eqref{eq:mod2grdef}. Thus the relation
			\begin{equation}\label{eq:claimmidproofeulerchar}
			\chi\left( I^{\#,N}(Y) \right) = \chi(R^N(Y)) = |H_1(Y;\Z)|^{N-1}
		\end{equation}
		will hold in the case at hand once it is shown that the mod $2$ grading of each abelian critical point is even. In what follows, we represent $\rho\in R^N(Y)$ by a connection $\alpha\#\beta$ on $Y\# T^3$, where $\alpha$ is a flat $SU(N)$ connection on $Y$ and $\beta$ is one of the $N$ flat non-degenerate $PU(N)$ connections on $T^3$ compatible with the bundle data $\gamma$. 
		
		Let $X$ be a $4$-manifold with boundary $Y$. Denote by $W$ the cobordism from $Y$ to $Y\# T^3$ which is topologically the boundary sum of $Y\times I$ with $D^2\times T^2$. Write $X'$ for the union of $X$ and $W$ along $Y$. Let $A_X$ be a $PU(N)$ connection on $X$ with a cylindrical end attached, restricting to the pullback of $\alpha$ over the end, and $\text{ind}^-(A_X)$ the index of the associated ASD operator with exponential decay (see for example \cite[\S 3.3.1]{donaldson-book}). Let $A_W$ be a $PU(N)$ connection on $W$ with cylindrical ends attached, equal to the pullback of $\alpha$ on the incoming end and that of $\alpha\#\beta$ on the outgoing end. Then by index additivity we have
		\begin{align}
			\text{gr}(\alpha\# \beta)  & \equiv \text{ind}(A)+ (N^2-1)(b_1(X')-b^+(X')+b_1(Y\# T^3)-1) \label{eq:indexcompproof} \\[2mm]
					& \equiv \text{ind}^{-}(A_X) + \text{ind}^{+-}(A_W) + (N^2-1)(b_1(X)-b^+(X))\pmod{2} \nonumber
		\end{align}	
		Here $\text{ind}^{+-}(A_W)$ is the index of the ASD operator associated to $A_W$ with exponential growth at the incoming end and exponential decay at the outgoing end; see \cite[\S 2.2]{DX} for this setup. By the Atiyah--Patodi--Singer theorem \cite{apsi} (see also \cite[Eq. 2.16]{DX}),
		\begin{align*}
			 \text{ind}^{+-}&(A_W) = 4N\kappa(A_W) -   \frac{N^2-1}{2}(\chi(W) + \sigma(W)) \\[2mm]
			 	&  + \frac{1}{2}\left( h^0(\alpha)+h^1(\alpha) -  h^0(\alpha\#\beta)-h^1(\alpha\# \beta) -\rho_{\text{ad}\alpha}(Y) + \rho_{\text{ad}(\alpha\# \beta)}(Y\# T^3)  \right).
		\end{align*}
		Here $h^i(\alpha)$ is the dimension of $H^i(Y;\text{ad}\alpha)$, and so forth. By our current assumptions, we have $h^0(\alpha\# \beta)=h^1(\alpha)=0$ and $h^0(\alpha)=\dim \Gamma_\rho$, while $h^1(\alpha\#\beta)=\dim  SU(N)/\Gamma_\rho$. We also choose $A_W$ to be a flat connection, obtained by gluing a flat connection extending $\beta$ over $D^2\times T^2$ to the product flat connection induced by $\alpha$ on $Y\times I$ and extending by translation to cylindrical ends. (That $\beta$ extends to a flat connection over $D^2\times T^2$ is easily verified by the description in \cite[\S 4.1]{KM:YAFT}.) By $\kappa(A_W)=0$, $\chi(W)=-1$, $\sigma(W)=0$:
		 		\begin{align*}
			 \text{ind}^{+-}&(A_W) =   \dim \Gamma_\rho  - \frac{1}{2}(\rho_{\text{ad}\alpha}(Y) - \rho_{\text{ad}(\alpha\# \beta)}(Y\# T^3) ).
		\end{align*}
		On the other hand, by Atiyah--Patodi--Singer's result \cite[Thm. 2.4]{apsii}, we have
		\begin{equation}\label{eq:apsuseno2}
			\rho_{\text{ad}\alpha}(Y) - \rho_{\text{ad}(\alpha\# \beta)}(Y\# T^3) = (N^2-1)\sigma(W) - \sigma_{\text{ad} A_W}(W)
		\end{equation}
		Here we use that the adjoint bundle has rank $N^2-1$, see \eqref{eq:adjointbundlered}. A computation using the Mayer--Vietoris sequence with local coefficients shows $H^2(W;\text{ad}A_W)=0$ and hence the right side of \eqref{eq:apsuseno2} vanishes. Thus $\text{ind}^{+-}(A_W) =   \dim \Gamma_\rho$. Plugging into \eqref{eq:indexcompproof} yields
				\begin{align}
			\text{gr}(\alpha\# \beta) \equiv  \text{ind}^{-}(A_X) + \dim \Gamma_\rho + (N^2-1)(b_1(X)-b^+(X))\pmod{2} \label{eq:grmidproofline}
		\end{align}
		Now suppose $\rho$ is abelian. Then $\alpha$ is compatible with a splitting $L_1\oplus\cdots\oplus L_N$ where each $L_i$ is a complex line bundle. The associated adjoint bundle is isomorphic to
		\begin{equation}\label{eq:adjointbundlered}
			 \bigoplus_{i<j} L_i\otimes L_j^{-1} \oplus \underline{\R}^{N-1}
		\end{equation}
		We may choose $X$ such that $H^2(X;\Z)\to H^2(Y;\Z)$ is a surjection; then we may choose line bundles $\widetilde L_i$ over $X$ which extend the $L_i$. Further, choose $A_X$ so that $\text{ad}A_X$ splits as $\oplus_{i<j} A_{ij}\oplus \Theta$ where $\Theta$ is a trivial connection on $\underline{\R}^{N-1}$ and $A_{ij}$ is a $U(1)$ connection on $\widetilde L_i\otimes \widetilde L_j^{-1}$. With these choices, we compute
		\begin{equation}\label{eq:indexmod2forfillingx}
			\text{ind}^-(A_X) \equiv -(N^2-1)(1-b_1(X)+b^+(X)) \pmod{2}
		\end{equation}
		Indeed, the index of $A_X$ splits into a sum; the indices associated to the $A_{ij}$ are even, because the relevant operators are complex linear, and the index associated to $\Theta$ is $(N^2-1)$ times the index of the standard ASD operator on $X$. We then obtain from \eqref{eq:grmidproofline}:
					\begin{align*}
			\text{gr}(\alpha\# \beta) \equiv \dim \Gamma_\rho -(N^2-1) \equiv 0 \pmod{2} 
		\end{align*}
		Here we have used that $\dim \Gamma_\rho =  \sum n_i^2-1$ for some non-negative integers $n_i$ which satisfy $\sum n_i=N$. This completes the proof of claim \eqref{eq:claimmidproofeulerchar} under the given assumptions.
		
		In the general case for $b_1(Y)=0$, a holonomy perturbation must be used. When $N=2$, it is explained in \cite[Theorem 3.6]{mme} that there are small holonomy perturbations for $Y$ such that the critical set of the Chern--Simons functional is discrete and non-degenerate, and the corresponding orbits on the framed configuration space are Morse--Bott non-degenerate. In our case, we use such a perturbation on $Y\# T^3$ which is supported on $Y$. For the above argument to adapt, it is important that for a small enough such perturbation, the number of abelian critical points and their stabilizer-types remain the same; this is true because these reducibles are cut out transversely within the subspace of the configuration space consisting of abelian connections.
		
				When $b_1(Y)>0$, the abelian representations in $R^N(Y)$ form a disjoint union of tori, each of dimension $b_1(Y)$. In the simplified version of the above argument, these tori now contribute zero to the Euler characteristic. Adapting the above argument, with similar remarks regarding perturbations, gives $\chi(I^{\#,N}(Y))=0$ in this case.
	\end{proof}
	
	\begin{remark}
		The absolute $\Z/2$-grading used here agrees with that of \cite[Proposition 6.20]{cdx}. In that reference, the $\Z/2$-grading is determined by the conditions that (i) for any cobordism $(W,c):(Y,\omega)\to (Y',\omega')$ between $N$-admissible pairs, the degree of the corresponding cobordism map is the parity of
		\[
			\frac{N^2-1}{2}(\chi(W) + \sigma(W) +b_0(Y') - b_0(Y) + b_1(Y') - b_1(Y));
		\]
		and (ii) the generator of $I^N_\ast(\emptyset)$, which is by convention $1$-dimensional, is supported in even degree. Condition (i) follows from the definition \eqref{eq:mod2grdef} along the same lines as \cite[Prop. 7.1]{scadutothesis}. Furthermore, the normalization condition (ii) is equivalent, assuming (i), to the condition that $I^{\#,N}(S^3)$ is supported in even degree. 
	\end{remark}

The framed instanton homology can also be defined for any $3$-manifold $Y$ with a $1$-cycle $\omega\subset Y$. In this case the $U(N)$ framed instanton homology is denoted
\begin{equation}\label{eq:unframedinstantonhomologydef}
	I^{\#,N}_\ast(Y,\omega)
\end{equation}
and is defined as the $(N)$-eigenspace of the operator $\beta_2$ acting on $I^N_\ast(Y\# T^3,\omega\cup \gamma )$. The isomorphism type of this group only depends on $Y$ and $[\omega]\in H_1(Y;\Z/N)$, and sometimes we conflate $\omega$ with its homology class, or its Poincar\'{e} dual. The argument of Theorem \ref{thm:eulercharframed} can be adapted to show that the Euler characteristic of \eqref{eq:unframedinstantonhomologydef} is also given by the right side of \eqref{eq:framedeulercharthm}, and is in particular independent of $\omega$.
	
	\begin{remark}
		Theorem \ref{thm:eulercharframed} (and its extension to \eqref{eq:unframedinstantonhomologydef} mentioned above) is compatible with the surgery exact $(N+1)$-gons of \cite{cdx}, which were proven for $N\leq 4$. In particular, it satisfies the Euler characteristic relation \cite[Cor. 1.9]{cdx}, giving evidence for the existence of surgery exact $(N+1)$-gons for $N>4$.
	\end{remark}
	
	It follows from a result of Borel \cite{borel} that if $G$ is a compact connected Lie group and $H$ is a closed connected subgroup with rank equal to that of $G$, then the cohomology of $G/H$ with complex coefficients is supported in even degrees. In particular, in the case that $b_1(Y)=0$, if $R^N(Y)$ consists entirely of abelian representations and is Morse--Bott non-degenerate for the Chern--Simons functional, then 
		\begin{equation}\label{eq:lspacecondition}
		\dim I^{\#,N}(Y) = |H_1(Y;\Z)|^{N-1}.
	\end{equation}
	This occurs in the case that $Y$ is a lens space. The condition \eqref{eq:lspacecondition}, that the dimension of the $U(N)$ framed instanton homology is equal to its Euler characteristic, is a natural generalization of the $U(2)$ instanton $L$-space condition \cite{bs-lspace}. A natural question is whether the class of $3$-manifolds defined by the condition \eqref{eq:lspacecondition} depends on $N$. In the case that $Y$ is an $U(2)$ instanton $L$-space and also satisfies \eqref{eq:lspacecondition} for some $N>2$, we have
			\begin{equation}\label{eq:frq}
		I^{\#,N}(Y) \cong I^{\#,2}(Y)^{\otimes (N-1)}.
	\end{equation}
	Thus we are led to ask: is there an example of a $3$-manifold $Y$ for which \eqref{eq:frq} does not hold for some $N>2$? 

\begin{remark}
	In the above discussion, one may also include the case $N=1$. The $U(1)$ framed instanton homology of $Y$ is a special case of the plane Floer homology of the first author \cite{daemi-plane}, and is isomorphic to the homology of the Jacobian torus of $Y$.
\end{remark}

\subsection{A product formula for connected sums}

The $U(N)$ framed instanton homology behaves in a simple way with respect to connected sums. The following generalizes a known result in the $N=2$ case.

\begin{theorem}
	Let $(Y,\omega)$ and $(Y',\omega')$ be connected $3$-manifolds with $1$-cycles. Then
\begin{equation}\label{eq:framedkunneth}
	I^{\#,N}(Y\#Y',\omega\cup \omega') \cong I^{\#,N}(Y,\omega)\otimes I^{\#,N}(Y',\omega')
\end{equation}
\end{theorem}

The proof of this result is analogous to the proof in the $N=2$ case, which is a slight variation of the argument given in \cite[Cor. 5.9]{KM:unknot}. The main device used is genus $1$ excision, which holds for all $N$, just as in the $N=2$ case, by the simple description of the Floer homology of $(T^3,\gamma)$ as given in \cite{kronheimer-higher, KM:YAFT}. With genus $1$ excision, the key observation in proving \eqref{eq:framedkunneth} is that one may cut open $Y\# T^3$ and $Y'\# T^3$ along the  copies of $2$-tori labelled $R$ in each $3$-torus, and reglue the resulting boundary components so as to form a connected $3$-manifold diffeomorphic to $Y\# Y'\# T^3$. From the proof it is also easy to see that, just as in the $N=2$ case, the isomorphism \eqref{eq:framedkunneth} preserves $\Z/2$-gradings, and is natural with respect to ``split'' cobordisms.

\subsection{A decomposition result for cobordism maps}

Despite the various properties of $U(N)$ framed instanton homology discussed above, for $N>2$ very few computations of these groups are currently known. In the $N=2$ case, many computations have been aided by the use of Floer's exact triangle \cite{floer-dehn,scadutothesis}. The surgery exact $(N+1)$-gons of \cite{cdx}, proved in the cases $N\leq 4$, provide a generalization that may be useful for computations in the higher rank cases. Another tool that has been very useful in the $N=2$ case, as is illustrated in the work of Baldwin and Sivek \cite{bs-lspace}, is a decomposition result for cobordism maps. In this subsection we explain how to obtain an analogous decomposition result in the $N=3$ case. This is essentially an adaptation of Theorem \ref{thm:structure} to the setting of cobordism maps in $U(3)$ framed instanton homology. 

Let $(X, w):(Y,\omega)\to (Y',\omega')$ be a cobordism of pairs, where $w$ is a $2$-cycle restricting to $\omega$ and $\omega'$ at the boundary components. Here and below we suppose $X$, $Y$ and $Y'$ are connected. There is an associated cobordism map of framed instanton homology groups
\[
	I^{\#,N}(X,w): I^{\#,N}(Y,\omega)\to I^{\#,N}(Y',\omega')
\]
obtained by choosing a path $c$ embedded in $X\setminus w$ and splicing $I\times T^3$ onto $X$ along this path, to obtain a cobordism $X^\#$ from $Y\# T^3$ to $Y'\#T^3$; see \cite[\S 7.1]{scadutothesis}. The choice of path $c$ is suppressed from the notation.

In the case $N=3$, following the decomposition \eqref{es-decomp} we may write
\begin{equation}\label{eq:frameddecomp3mfld}
	I^{\#,3}(Y,\omega) = \bigoplus_{s} I^{\#,3}(Y,\omega;s)
\end{equation}
where the direct sum is over homomorphisms $s:H_2(Y;\Z)\to \Gamma \subset \Z\oplus \Z$, with $\Gamma$ being the sublattice of pairs $(a,b)$ with $a$ and $b$ of the same parity. 

In what follows, we write $s=(s_2,s_3)$ for any such homomorphism, where $s_2$ and $s_3$ are the projections to the two $\Z$ factors. Further, if $s:H_2(X;\Z)\to \Z\oplus \Z$ is a homomorphism where $X$ is manifold with a submanifold $Y\subset X$, we write $s|_Y$ for the composition of $s$ with the inclusion-induced homomorphism $H_2(Y;\Z)\to H_2(X;\Z)$. 

The decomposition result for cobordism maps in this setting is as follows.

\begin{theorem}\label{thm:cobordismmapdecomposition}
	Let $(X, w):(Y,\omega)\to (Y',\omega')$ be a cobordism of pairs with $b_1(X)=0$ and $b^+(X)>0$. Then there is a natural decomposition of the cobordism map
	\[
		I^{\#,3}(X,w) = \sum_{s} I^{\#,3}(X,w; s),
	\]
	\[
		I^{\#,3}(X,w; s): I^{\#,3}(Y,\omega;s|_Y) \to  I^{\#,3}(Y',\omega';s|_{Y'})
	\]
	where the sum is over homomorphisms $s:H_2(X;\Z)\to \Z\oplus\Z$. 	These maps satisfy:
	\begin{enumerate}
		\item[{\emph{(i)}}] $I^{\#,3}(X,w; s)=0$ for all but finitely many $s$.
		\item[{\emph{(ii)}}] If $I^{\#,3}(X,w; s)\neq 0$, then $s_2(x)+s_3(x)+x\cdot x\equiv 0\pmod{2}$ for all $x\in H_2(X;\Z)$, and for any smoothly embedded, connected, orientable surface $\Sigma\subset X$ with non-negative self-intersection and having $[\Sigma]$ non-torsion we have
		\[
			|s_2([\Sigma]) \pm s_3([\Sigma]) | + [\Sigma]\cdot [\Sigma] \leq 2g(\Sigma)-2.
		\]
		
		\item[{\emph{(iii)}}] If $(X,w)$ is a composition of two cobordisms $(X'',w''):(Y,\omega)\to (Y'',\omega'')$ and $(X',w'):(Y'',\omega'')\to (Y',\omega')$ each with $b_1=0$ and $b^+>0$, then
		\[
			I^{\#,3}(X',w'; s')\circ I^{\#,3}(X'',w''; s'') = \sum I^{\#,3}(X,w; s)
		\]
		where the sum is over $s:H_2(X;\Z)\to \Z\oplus \Z$ such that $s|_{X'}=s'$ and $s|_{X''}=s''$.
		\item[{\emph{(iv)}}] Let $\widehat{X}=X\# \overline{\mathbb{C}\mathbb{P}}^2$ denote the blowup of $X$, with $e$ the exceptional sphere and $E$ its Poincar\'{e} dual. Let $\zeta= e^{2\pi i/3}$. Then for all $s:H_2(X;\Z)\to \Z$ and $l,k\in \Z$:
		\begin{align*}
			I^{\#,3}(\widehat{X},w; s + l E_2 + k E_3 ) &= \begin{cases} \frac{1}{6}I^{\#,3}(\widehat{X},w; s),  & l=\pm 1, k=0 \\ 
			 \frac{1}{3}I^{\#,3}(\widehat{X},w; s),  & l=0, k=\pm 1\\
			 0, & \text{ otherwise }
			\end{cases} \\[2mm]
			I^{\#,3}(\widehat{X},w+e; s + l E_2 + k E_3 ) &= \begin{cases} \frac{1}{6}I^{\#,3}(\widehat{X},w; s),  & l=\pm 1, k=0 \\ 
			 \frac{1}{3}\zeta^{k} I^{\#,3}(\widehat{X},w; s),  & l=0, k=\pm 1\\
			 0, & \text{ otherwise }
			\end{cases} 
		\end{align*}
		Here $E_2$ (resp. $E_3$) is the homomorphism $H_2(\widehat{X};\Z)\to\Z$ which is pairing with $E$ composed with the inclusion of $\Z$ into the first factor (resp. second factor) of $\Z\oplus \Z$. 
		\item[{\emph{(v)}}] $I^{\#,3}(X,w+a; s)= \zeta^{s_3(a)}I^{\#,3}(X,w; s)$ for any $a\in H_2(X;\Z)$.
	\end{enumerate}
\end{theorem}

This result should be compared to a similar decomposition result in the case of $U(2)$ framed instanton homology, given as Theorem 1.16 in \cite{bs-lspace}. The proof of Theorem \ref{thm:cobordismmapdecomposition} is largely a consequence of a straightforward adaptation of Theorem \ref{thm:structure} to the case of cobordisms. This adaptation is carried out in the $U(2)$ case in \cite{bs-lspace}. We only mention some essential points. First, one defines a formal power series in $\Gamma,\Lambda\in H_2(X;\R)$ by
\begin{equation}\label{eq:frameddonaldsonseries}
	\mathbb{D}^\#_{X, w}({\Gamma_{(2)}+\Lambda_{(3)}}) = I^{\#,3}(X,w ,(1+\frac{1}{3}x_{(2)}+\frac{1}{9}x_{(2)}^2)e^{{\Gamma_{(2)}+\Lambda_{(3)}}})
\end{equation}
where the notation $I^{\#,3}(X,w,z)$ for $z\in \mathbf{A}^3(X)$ is the cobordism map defined by cutting down via the divisor associated to $z$. The coefficients of this power series in $\Gamma,\Lambda$ are linear maps from $I^{\#,3}(Y,\omega)$ to $I^{\#,3}(Y',\omega')$. The proof of Theorem \ref{thm:structure} adapts to show that
	\begin{equation}
\mathbb{D}^\#_{X, w}({\Gamma_{(2)}+\Lambda_{(3)}}) =  e^{\frac{Q(\Gamma)}{2}-Q(\Lambda)}\sum_{i, j} c_{i, j}  \zeta^{w\cdot \left(\frac{K_i-K_j}{2}\right)}e^{\frac{\sqrt{3}}{2}(K_i+K_j)\cdot \Gamma+\frac{\sqrt{-3}}{2}(K_i-K_j)\cdot \Lambda}  \label{eq:structurethmformulacobordism}
\end{equation}
The new feature here is that the $c_{i,j}$ are no longer constants, but are instead linear maps $I^{\#,3}(Y,\omega)\to I^{\#,3}(Y',\omega')$. Furthermore, $c_{i,j}$ has coefficients in $\Q[\sqrt{3}]$ with respect to rational bases of the framed instanton groups. Here we view $K_i:H_2(X;\Z)\to\Z$; these are characteristic, just as before, and the adjunction inequality is also as stated in \eqref{eq:structureadjunction}.

The presence of $T^3$ with its non-trivial bundle in the formation of $X^\#$, and the assumption $b_1(X)=0$, guarantee that $X^\#$ has the corresponding $U(3)$ simple type condition for all choices of $w\subset X$ (with bundle data over the $I\times T^3$ part in $X^\#$ fixed). By the definition of framed instanton homology, $\beta_2=\mu_2(x)$ acts as $3$, and so the right side of \eqref{eq:frameddonaldsonseries} is
\[
	3 I^{\#,3}(X,w , e^{{\Gamma_{(2)}+\Lambda_{(3)}}}).
\]
Now for a homomorphism $s=(s_2,s_3):H_2(X;\Z)\to \Z\oplus \Z$ we define
\[
	I^{\#,3}(X,w; s) = \begin{cases} \frac{1}{3} \zeta^{w\cdot \left(\frac{K_i-K_j}{2}\right)} c_{i, j} & \text{ if }s_2=\frac{1}{2}(K_i+K_j),\;\, s_3=\frac{1}{2}(K_i-K_j)\\
	0 & \text{ otherwise} \end{cases}
\]
The properties listed in Theorem \ref{thm:cobordismmapdecomposition} are then proved in much the same way as in the $U(2)$ case, using the formula \eqref{eq:structurethmformulacobordism} and its properties related to the structure theorem; see \cite{bs-lspace} for details. Note that property (iv) follows from a straightforward adaptation of Theorem \ref{thm:blowup}, the $N=3$ blowup formula, to the invariants \eqref{eq:frameddonaldsonseries}.

\begin{remark}
	In the $U(2)$ decomposition result of \cite{bs-lspace}, the assumption $b^+(X)>0$ is removed using a trick that involves the trace cobordism of $1$-surgery on the $(2,5)$ torus knot. We expect that the assumption $b^+(X)>0$ can also be removed from Theorem \ref{thm:cobordismmapdecomposition}.
\end{remark}


\section{$N=3$ knot homology and the Alexander polynomial}\label{sec:alexander}

In this final section, we describe a conjectural relationship between the $U(3)$ instanton knot homology group $KHI_*^3(Y,K)$ introduced in Section \ref{sec:sutured} and the Alexander polynomial. Throughout this section, $K$ is a knot in an integer homology $3$-sphere $Y$.

There is a $\Z/2$-grading on $KHI_*^3(Y,K)$ defined analogously as in the $U(2)$ case. More generally, there is a relative $\Z/2$-grading on the $U(3)$ sutured instanton homology of any balanced sutured $3$-manifold. This is because of the following: $U(3)$ instanton homology of an admissible bundle has a relative $\Z/2$-grading; the operators from which the simultaneous eigenspaces are defined all have even degree; and the excision maps used in the proof of invariance are also homogeneously $\Z/2$-graded. To define an absolute $\Z/2$-grading on $KHI_*^3(Y,K)$, one can use \eqref{eq:mod2grdef} for a particular closure of the knot complement. However, the specific choice of convention will not concern us for what follows. 

Using the $\Z/2$-grading on $KHI_*^3(Y,K)$ and the decomposition
\begin{equation}\label{eq:knothomologyu3graded2}
  KHI_*^3(Y,K)=\bigoplus_{(a,b)\in \mathcal C_{g+1}}KHI_*^3(Y,K;a,b),
\end{equation} 
from Section \ref{sec:sutured} (where $g$ is the Seifert genus of $K$), which is compatible with the $\Z/2$-grading, we define a Laurent polynomial in two-variables $t_2$ and $t_3$ as follows:
	\[
	  \Delta^3_{(Y,K)}(t_2,t_3):=\sum_{(a,b)\in \mathcal{C}_{g+1}} \chi(KHI_*^3(Y,K;a,b))t_2^at_3^b
	\]
	The authors expect that this polynomial is determined by the symmetrized Alexander polynomial $\Delta_{(Y,K)}(t)$ through the following formula:
	\begin{equation}\label{eq:conjecturalalexanderrelation}
		\Delta^3_{(Y,K)}(t_2,t_3) = \pm \Delta_{(Y,K)}(t_2t_3)\Delta_{(Y,K)}(t_2t^{-1}_3).
	\end{equation}
	We present an argument for this relation that is based on some hypotheses which will be made clear momentarily. Let $Z=Z(K)$ be the closed 3-manifold which is the following closure of the sutured manifold associated to the knot:
	\[
		Z = Y\setminus N(K) \cup S^1\times (F_{1,1}\setminus D^2)
	\]
	Let $c,c'$ be two closed simple curves in the interior of $F_{1,1}$ which generate $H_1(F_{1,1};\Z)$. Then there are tori $\Sigma_1=S^1\times c$, $\Sigma_2=S^1\times c'$ in $Z$, and a surface $\Sigma_0\subset Z$ formed by gluing a Seifert surface for $K$ to $\{pt\}\times \partial F_{1,1}$. (The surface $\Sigma_0$ is denoted by $\overline S$ at the end of Section \ref{sec:sutured}.) Recall from \eqref{eq:knothomologyu3} that we have:
	\begin{equation}\label{eq:knothomologyu3-2}
  KHI_*^3(Y,K)=I_*^3(Z ,c'\vert T).
\end{equation}
	
	Consider the $4$-manifold $X=S^1\times Z$. We have tori $\Sigma_3=S^1\times c$, $\Sigma_4=S^1\times c'$, $\Sigma_5=S^1\times \mu$ (in each case the $S^1$ is external to $Z$), where $\mu$ is a meridian for $K$. Then 
	\[
		H_2(X;\Z) = \bigoplus_{i=0}^5 \Z\cdot [\Sigma_i]
	\]
	Furthermore, $\Sigma_i\cdot \Sigma_i=0$ for all $i$. Note the signature of $X$ is zero. The $4$-manifold $X$ is {\emph{$U(2)$ strong simple type}} in the sense of Mu\~{n}oz \cite{munoz-basic}; this means that
\begin{equation}\label{eq:u2strongsimpletype}
	D_{X,w}^2(x^2 z) = 4 D_{X,w}^2(z),  \qquad D_{X,w}^3(\delta z)=0
\end{equation}
for all $z\in \bA^2(X)=\text{Sym}^\ast(H_0(X)\otimes H_2(X))\otimes \Lambda^\ast H_1(X)$, $\delta\in H_1(X)$, and $w\in H^2(X;\Z)$, where $x$ is a point class. To see that $X$ is $U(2)$ strong simple type, one first shows that $(X,w)$ is $U(2)$ strong simple type for a certain $w\in H^2(X;\Z)$. Choose 
	\[
		w = \text{P.D.}\left( \Sigma_0 \cup \Sigma_1 \cup \Sigma_2\right).
	\]
	Then $w$ has odd pairing with $\Sigma_i$ for $3\leq i\leq 5$. The strong simple type relations are obtained for $D^2_{X,w}$ through gluing formulas along the Fukaya--Floer homology of $S^1\times \Sigma_i$ for $3\leq i\leq 5$. Here it is key that the surfaces $\Sigma_i$ for $3\leq i \leq 5$ are genus $1$, and the Fukaya--Floer homology in the genus $1$ case is particularly simple; in particular, the action of a 1-cycle in $S^1\times \Sigma_i$ on its Fukaya--Floer homology is trivial, and the operator associated to the point class squares to $4$ times the identity. As each element of $H_1(X)$ comes from some element of $H_1(\Sigma_i)$ for $3\leq i\leq 5$, one obtains the relations
	\[
		D_{X,w}^2(\delta z)=0
\]
for all $\delta \in H_1(X)$ and $z\in \bA^2(X)=\text{Sym}^\ast(H_0(X)\otimes H_2(X))\otimes \Lambda^\ast H_1(X) $. Then, strong simple type for one $w\in H^2(X;\Z)$ implies the result for all $w$ (see Mu\~{n}oz \cite{munoz-basic}).
	
	Next, as $X$ is $U(2)$ strong simple type, we can apply the structure theorem in this case (again, see Mu\~{n}oz \cite{munoz-basic}). Let $K\in H^2(X;\Z)$ be a $U(2)$ basic class for $X$. Then 
	\[
		2g(\Sigma_i)-2 \geq |K\cdot \Sigma_i | + \Sigma_i\cdot \Sigma_i = |K\cdot \Sigma_i |
	\]
	by the adjunction inequality. For $1\leq i\leq 5$, we have $g(\Sigma_i)=1$, and we obtain $K\cdot \Sigma_i=0$. For $i=0$, we obtain $|K\cdot \Sigma_0|\leq 2g$ where $g=g(\Sigma_0)-1$ is the Seifert genus of $K$. Define
	\[
		K_r = 2r \text{P.D.}[\Sigma_5].
	\]
	We conclude that the possible $U(2)$ basic classes of $X$ are the $K_r$ for which $r\in \Z$, $|r|\leq g$. Now let $w=\text{P.D.}[\Sigma_4]$. The structure theorem in the $U(2)$ case then reads
	\[
		\widehat{D}^2_{X,w}(e^{\Sigma_0}) =D^2_{X,w}((1+\frac{x}{2})e^{\Sigma_0}) = e^{Q(\Sigma_0)/2} \sum_{r=-g}^g (-1)^{(w^2+ K_r\cdot w)/2}\beta_r e^{K_r\cdot \Sigma_0}.
	\]
	Witten's conjecture adapted to this case gives $\beta_r = 2^{2 + \frac{1}{4}(7\chi(X) + 11\sigma(X))}\text{SW}(K_r)= 4 \text{SW}_X(K_r)$. Using this, and $Q(\Sigma_0)=0$, $w\cdot w=0$, $K_r\cdot w=0$, we obtain
	\[
		\widehat{D}^2_{X,w}(e^{\Sigma_0}) = 4 \sum_{r=-g}^g \text{SW}_X(K_r)e^{2r}
	\]
	This invariant can also be expressed as the super trace of a combination of maps induced on the $U(2)$ instanton knot homology $KHI^2_\ast(Y,K)$ by the cobordism $X'=[0,1]\times Z$ with bundle cyle $w'=[0,1]\times c'$, adorned with the operators $(1+\frac{x}{2})\mu(\Sigma_0)^i$. First, we recall that 
	\[
		KHI^2_\ast(Y,K) = \bigoplus_{j=-g}^g KHI^2_\ast(Y,K;j)
	\]
	where $KHI_\ast(Y,K)$ is the generalized eigenspace of $\mu(pt)$ acting on $I_\ast^2(Z,c')$ with eigenvalue $2$, and $KHI^2_\ast(Y,K;j)$ is the simultaneous generalized eigenspace of $(\mu(pt),\mu(\Sigma_0))$ acting on $I_\ast^2(Z,c')$ with eigenvalues $(2,2j)$. We compute:
	\begin{align*}
		\widehat{D}^2_{X,w}(e^{\Sigma_0}) &= 2 \sum_{i=0}^\infty \frac{1}{i!}\text{tr}_s \left( I^2(X',w', (1+\frac{x}{2})\mu(\Sigma_0)^i)\right) \\
		& = 2 \sum_{j=-g}^g\sum_{i=0}^\infty \frac{1}{i!}\chi( KHI_\ast(Y,K;j) ) 2 (2j)^i   = 4 \sum_{j=-g}^g A_j e^{2j}
	\end{align*}
	where $A_j$ is the coefficient of $t^j$ in $\Delta_{(Y,K)}(t)$. Here, in the last equality, we have used that the graded Euler characteristic of $KHI_\ast(Y,K)$ is the Alexander polynomial \cite{KM:alexander,lim}. The ``2'' appearing outside the summands in the middle expressions comes from a gluing factor (see the discussion \cite[\S 5.2]{KM:unknot}). In particular, we obtain
	\begin{equation}\label{eq:swinalexstep}
		\text{SW}_X(K_r) = A_r.
	\end{equation}
	
	We then repeat this analysis in the $U(3)$ setting. Assume $X$ is $U(3)$ simple type (we predict this is true, but we cannot adapt the above argument in the $U(2)$ case; see Remark \ref{rmk:simpletype}). Mari\~{n}o and Moore \cite{marino-moore} conjecture that the basic classes in the $U(3)$ structure theorem are the same as the $U(2)$ basic classes (see also \cite[Conjecture 7.2]{DX}). The $U(3)$ structure theorem then gives the following:
	\begin{equation}\label{eq:u3casealexanderstep}
\widehat{D}_{X, w}^3(e^{s_2 (\Sigma_0)_{(2)}+s_3(\Sigma_0)_{(3)}}) =\sum_{i, j} c_{i, j} e^{s_2\sqrt{3}(i+j) + s_3\sqrt{-3}(i-j)}.
\end{equation}
The conjecture also states that the constants $c_{i,j}$ are given by
\begin{align}
	c_{i,j} &= 2^{\chi(X)+\frac{3}{2}\sigma(X)+\frac{1}{2}K_i\cdot K_j} 3^{2 + \frac{7}{4}\chi(X) + \frac{11}{4}\sigma(X)}\text{SW}_X(K_i)\text{SW}_X(K_j) \nonumber \\[2mm] &= 9\text{SW}_X(K_i)\text{SW}_X(K_j). \label{eq:cijinalexanderstep}
\end{align}
We compute that the left side of \eqref{eq:u3casealexanderstep} is also equal to the following (where the ``$3$'' on the outside of the first sum comes from a gluing factor analogous to the $N=2$ case):
	\begin{align*}
		 3 \sum_{k,l=0}^\infty \frac{1}{k!l!}\text{tr}_s & \left( I^3(X',w', (1+\frac{x_{(2)}}{3} + \frac{x_{(2)}^2}{9})(s_2\mu_2(\Sigma_0))^{k}(s_3\mu_3(\Sigma_0))^l)\right) \\
		& = 3 \sum_{(a,b)\in \mathcal{C}_{g+1}}\sum_{k,l=0}^\infty \frac{1}{k!l!}\chi( KHI^3_\ast(Y,K;a,b) ) 3 (\sqrt{3}a)^k  (\sqrt{-3}b)^l \\
		&   = 9 \sum_{(a,b)\in \mathcal{C}_{g+1}}^g A_{a,b} e^{\sqrt{3} a s_2 + \sqrt{-3}b s_3}
	\end{align*}
	where $A_{a,b}$ is the coefficient of $t_2^at_3^b$ in $\Delta_{(Y,K)}^3(t_3,t_3)$. By \eqref{eq:swinalexstep}, \eqref{eq:u3casealexanderstep} and \eqref{eq:cijinalexanderstep}, we get:
	\[
		A_{a,b} = c_{(a+b)/2, (a-b)/2} = A_{(a+b)/2}A_{(a-b)/2}.
	\]
	This establishes \eqref{eq:conjecturalalexanderrelation} under the assumptions stated above, apart from an overall sign $\pm$, which is determined by conventions that are not discussed here.
	
	A question arises regarding the extent to which the relationship \eqref{eq:conjecturalalexanderrelation}, which is at the level of graded Euler characteristics, might hold at the level of Floer homologies. For example, in the special case that the $U(2)$ and $U(3)$ instanton knot homology groups of $(Y,K)$ are supported in even gradings, \eqref{eq:conjecturalalexanderrelation} implies a vector space isomorphism
	\begin{equation}
		KHI_\ast^3(Y,K) \cong KHI_\ast^2(Y,K)\otimes KHI_\ast^2(Y,K). \label{eq:conjknotu2tou3}
	\end{equation}
	Thus, along the same lines following \eqref{eq:frq}, we are led to ask: is there a knot for which the isomorphism \eqref{eq:conjknotu2tou3} does not hold? 
	
	\begin{remark}
		The $4$-manifold $X=S^1\times Z(K)$ is an instance of Fintushel--Stern's knot surgery on a standard $T^2$ inside the $4$-torus $T^4$, and the relationship \eqref{eq:swinalexstep} is the same kind that is established in \cite{fs-knot-surgery} between Seiberg--Witten invariants and the Alexander polynomial (see \cite{ni} for a more general statement, relevant to our case). One approach to proving \eqref{eq:conjecturalalexanderrelation} is to establish Fintushel--Stern knot surgery formulas in the setting of Donaldson invariants, of type $U(2)$ and $U(3)$, removing the dependency of the above argument on conjectural relationships to Seiberg--Witten theory.
	\end{remark}
	
	\begin{remark}
	The $U(2)$ instanton knot homology group $KHI_\ast^2(Y,K)$ is isomorphic to a version of singular instanton homology where one takes the connected sum of $(Y,K)$ with the Hopf link in $S^3$, uses a bundle associated to an arc connecting the two resulting link components, and the singular condition is that the holonomy of a connection along shrinking meridians limits to an element in $U(2)$ conjugate to $\text{diag}(i,-i)$. See \cite[\S 5.4]{KM:unknot}.
	
		The $U(3)$ instanton knot homology $KHI^3_\ast(Y,K)$ is isomorphic to a version of $U(3)$ singular instanton homology as developed in \cite{KM:YAFT}, using a similar description as above, but with the singular condition that the holonomy of a connection along shrinking meridians limits to an element in $U(3)$ conjugate to $\text{diag}(1,\zeta,\zeta^2)$. A similar application of excision as in \cite{KM:unknot} can be used to give this isomorphism.
		
		A natural question is whether the relationship between Khovanov homology and the instanton group $KHI^2_\ast(S^3,K)$ established in \cite{KM:unknot} has a counterpart in the setting of $U(3)$ (or more generally, $U(N)$), and if so, what quantum knot homology theory plays the role of Khovanov homology in this setting.
	\end{remark}

\bibliography{references}

\begin{bibdiv}
\begin{biblist}

\bib{AB}{article}{
      author={Atiyah, M.~F.},
      author={Bott, R.},
       title={The {Y}ang-{M}ills equations over {R}iemann surfaces},
        date={1983},
        ISSN={0080-4614},
     journal={Philos. Trans. Roy. Soc. London Ser. A},
      volume={308},
      number={1505},
       pages={523\ndash 615},
         url={https://doi.org/10.1098/rsta.1983.0017},
      review={\MR{702806}},
}

\bib{apsi}{article}{
      author={Atiyah, M.~F.},
      author={Patodi, V.~K.},
      author={Singer, I.~M.},
       title={Spectral asymmetry and {R}iemannian geometry. {I}},
        date={1975},
        ISSN={0305-0041},
     journal={Math. Proc. Cambridge Philos. Soc.},
      volume={77},
       pages={43\ndash 69},
         url={https://doi.org/10.1017/S0305004100049410},
      review={\MR{397797}},
}

\bib{apsii}{article}{
      author={Atiyah, M.~F.},
      author={Patodi, V.~K.},
      author={Singer, I.~M.},
       title={Spectral asymmetry and {R}iemannian geometry. {II}},
        date={1975},
        ISSN={0305-0041},
     journal={Math. Proc. Cambridge Philos. Soc.},
      volume={78},
      number={3},
       pages={405\ndash 432},
         url={https://doi.org/10.1017/S0305004100051872},
      review={\MR{397798}},
}

\bib{baranovskii}{article}{
      author={Baranovski\u{\i}, V.~Yu.},
       title={The cohomology ring of the moduli space of stable bundles with
  odd determinant},
        date={1994},
        ISSN={1607-0046},
     journal={Izv. Ross. Akad. Nauk Ser. Mat.},
      volume={58},
      number={4},
       pages={204\ndash 210},
         url={https://doi.org/10.1070/IM1995v045n01ABEH001629},
      review={\MR{1307064}},
}

\bib{DB:FF}{incollection}{
      author={Braam, Peter},
      author={Donaldson, Simon},
       title={Fukaya-{F}loer homology and gluing formulae for polynomial
  invariants},
        date={1995},
   booktitle={The {F}loer memorial volume},
      series={Progr. Math.},
      volume={133},
   publisher={Birkh\"auser, Basel},
       pages={257\ndash 281},
      review={\MR{1362830}},
}

\bib{borel}{article}{
      author={Borel, Armand},
       title={Sur la cohomologie des espaces fibr\'{e}s principaux et des
  espaces homog\`enes de groupes de {L}ie compacts},
        date={1953},
        ISSN={0003-486X},
     journal={Ann. of Math. (2)},
      volume={57},
       pages={115\ndash 207},
         url={https://doi.org/10.2307/1969728},
      review={\MR{51508}},
}

\bib{bs-lspace}{article}{
      author={Baldwin, John~A.},
      author={Sivek, Steven},
       title={Instantons and {L}-space surgeries},
        date={2023},
        ISSN={1435-9855},
     journal={J. Eur. Math. Soc. (JEMS)},
      volume={25},
      number={10},
       pages={4033\ndash 4122},
         url={https://doi.org/10.4171/jems/1280},
      review={\MR{4634689}},
}

\bib{cdx}{article}{
      author={Culler, Lucas},
      author={{Daemi}, Aliakbar},
      author={Xie, Yi},
       title={{Surgery, {P}olygons and $SU(N)$-{F}loer {H}omology}},
        date={2017Dec},
     journal={arXiv e-prints},
       pages={arXiv:1712.09910},
      eprint={1712.09910},
}

\bib{culler}{book}{
      author={Culler, Lucas~Howard},
       title={The {B}lowup {F}ormula for {H}igher {R}ank {D}onaldson
  {I}nvariants},
   publisher={ProQuest LLC, Ann Arbor, MI},
        date={2014},
  url={http://gateway.proquest.com/openurl?url_ver=Z39.88-2004&rft_val_fmt=info:ofi/fmt:kev:mtx:dissertation&res_dat=xri:pqm&rft_dat=xri:pqdiss:0830159},
        note={Thesis (Ph.D.)--Massachusetts Institute of Technology},
      review={\MR{3279020}},
}

\bib{daemi-plane}{article}{
      author={{Daemi}, Aliakbar},
       title={{{A}belian {G}auge {T}heory, {K}nots and {O}dd {K}hovanov
  {H}omology}},
      eprint={1508.07650},
}

\bib{donaldson-book}{book}{
      author={Donaldson, S.~K.},
       title={Floer homology groups in {Y}ang-{M}ills theory},
      series={Cambridge Tracts in Mathematics},
   publisher={Cambridge University Press, Cambridge},
        date={2002},
      volume={147},
        ISBN={0-521-80803-0},
         url={https://doi.org/10.1017/CBO9780511543098},
        note={With the assistance of M. Furuta and D. Kotschick},
      review={\MR{1883043}},
}

\bib{Don-Lefs-pen-symp}{article}{
      author={Donaldson, S.~K.},
       title={Lefschetz pencils on symplectic manifolds},
        date={1999},
        ISSN={0022-040X,1945-743X},
     journal={J. Differential Geom.},
      volume={53},
      number={2},
       pages={205\ndash 236},
         url={http://projecteuclid.org/euclid.jdg/1214425535},
      review={\MR{1802722}},
}

\bib{DX}{article}{
      author={Daemi, Aliakbar},
      author={Xie, Yi},
       title={Sutured manifolds and polynomial invariants from higher rank
  bundles},
        date={2020},
        ISSN={1465-3060},
     journal={Geom. Topol.},
      volume={24},
      number={1},
       pages={49\ndash 178},
         url={https://doi.org/10.2140/gt.2020.24.49},
      review={\MR{4080482}},
}

\bib{earl}{article}{
      author={Earl, Richard},
       title={The {M}umford relations and the moduli of rank three stable
  bundles},
        date={1997},
        ISSN={0010-437X},
     journal={Compositio Math.},
      volume={109},
      number={1},
       pages={13\ndash 48},
         url={https://doi.org/10.1023/A:1000101030261},
      review={\MR{1473604}},
}

\bib{mme}{article}{
      author={Eismeier, Mike~Miller},
       title={Equivariant instanton homology},
        date={2023},
      eprint={arXiv:1907.01091},
}

\bib{earl-kirwan}{article}{
      author={Earl, Richard},
      author={Kirwan, Frances},
       title={Complete sets of relations in the cohomology rings of moduli
  spaces of holomorphic bundles and parabolic bundles over a {R}iemann
  surface},
        date={2004},
        ISSN={0024-6115},
     journal={Proc. London Math. Soc. (3)},
      volume={89},
      number={3},
       pages={570\ndash 622},
         url={https://doi.org/10.1112/S0024611504014832},
      review={\MR{2107008}},
}

\bib{El:symp-filling}{article}{
      author={Eliashberg, Yakov},
       title={A few remarks about symplectic filling},
        date={2004},
        ISSN={1465-3060,1364-0380},
     journal={Geom. Topol.},
      volume={8},
       pages={277\ndash 293},
         url={https://doi.org/10.2140/gt.2004.8.277},
      review={\MR{2023279}},
}

\bib{ET:confoliations}{book}{
      author={Eliashberg, Yakov~M.},
      author={Thurston, William~P.},
       title={Confoliations},
      series={University Lecture Series},
   publisher={American Mathematical Society, Providence, RI},
        date={1998},
      volume={13},
        ISBN={0-8218-0776-5},
         url={https://doi.org/10.1090/ulect/013},
      review={\MR{1483314}},
}

\bib{Et:symp-filling}{article}{
      author={Etnyre, John~B.},
       title={On symplectic fillings},
        date={2004},
        ISSN={1472-2747,1472-2739},
     journal={Algebr. Geom. Topol.},
      volume={4},
       pages={73\ndash 80},
         url={https://doi.org/10.2140/agt.2004.4.73},
      review={\MR{2023278}},
}

\bib{floer-dehn}{incollection}{
      author={Floer, A.},
       title={Instanton homology and {D}ehn surgery},
        date={1995},
   booktitle={The {F}loer memorial volume},
      series={Progr. Math.},
      volume={133},
   publisher={Birkh\"{a}user, Basel},
       pages={77\ndash 97},
      review={\MR{1362823}},
}

\bib{fs-knot-surgery}{article}{
      author={Fintushel, Ronald},
      author={Stern, Ronald~J.},
       title={Knots, links, and {$4$}-manifolds},
        date={1998},
        ISSN={0020-9910},
     journal={Invent. Math.},
      volume={134},
      number={2},
       pages={363\ndash 400},
         url={https://doi.org/10.1007/s002220050268},
      review={\MR{1650308}},
}

\bib{Fuk:FF}{incollection}{
      author={Fukaya, Kenji},
       title={Floer homology for oriented {$3$}-manifolds},
        date={1992},
   booktitle={Aspects of low-dimensional manifolds},
      series={Adv. Stud. Pure Math.},
      volume={20},
   publisher={Kinokuniya, Tokyo},
       pages={1\ndash 92},
      review={\MR{1208307}},
}

\bib{Gab:fol-sut}{article}{
      author={Gabai, David},
       title={Foliations and the topology of {$3$}-manifolds},
        date={1983},
        ISSN={0022-040X,1945-743X},
     journal={J. Differential Geom.},
      volume={18},
      number={3},
       pages={445\ndash 503},
         url={http://projecteuclid.org/euclid.jdg/1214437784},
      review={\MR{723813}},
}

\bib{GS:kirby-calc}{book}{
      author={Gompf, Robert~E.},
      author={Stipsicz, Andr\'{a}s~I.},
       title={{$4$}-manifolds and {K}irby calculus},
      series={Graduate Studies in Mathematics},
   publisher={American Mathematical Society, Providence, RI},
        date={1999},
      volume={20},
        ISBN={0-8218-0994-6},
         url={https://doi.org/10.1090/gsm/020},
      review={\MR{1707327}},
}

\bib{hopf-samelson}{article}{
      author={Hopf, H.},
      author={Samelson, H.},
       title={Ein {S}atz \"{u}ber die {W}irkungsr\"{a}ume geschlossener
  {L}iescher {G}ruppen},
        date={1941},
        ISSN={0010-2571},
     journal={Comment. Math. Helv.},
      volume={13},
       pages={240\ndash 251},
         url={https://doi.org/10.1007/BF01378063},
      review={\MR{6546}},
}

\bib{juhasz}{article}{
      author={Juh{\'a}sz, Andr{\'a}s},
       title={Holomorphic discs and sutured manifolds},
        date={2006},
        ISSN={1472-2747},
     journal={Algebr. Geom. Topol.},
      volume={6},
       pages={1429\ndash 1457},
         url={http://dx.doi.org/10.2140/agt.2006.6.1429},
      review={\MR{2253454}},
}

\bib{J:surface-decomp}{article}{
      author={Juh\'{a}sz, Andr\'{a}s},
       title={Floer homology and surface decompositions},
        date={2008},
        ISSN={1465-3060,1364-0380},
     journal={Geom. Topol.},
      volume={12},
      number={1},
       pages={299\ndash 350},
         url={https://doi.org/10.2140/gt.2008.12.299},
      review={\MR{2390347}},
}

\bib{kirbylist}{inproceedings}{
      author={Kirby, Rob},
       title={Problems in low dimensional manifold theory},
        date={1978},
   booktitle={Algebraic and geometric topology ({P}roc. {S}ympos. {P}ure
  {M}ath., {S}tanford {U}niv., {S}tanford, {C}alif., 1976), {P}art 2},
      series={Proc. Sympos. Pure Math., XXXII},
   publisher={Amer. Math. Soc., Providence, RI},
       pages={273\ndash 312},
      review={\MR{520548}},
}

\bib{kirwan}{article}{
      author={Kirwan, Frances},
       title={The cohomology rings of moduli spaces of bundles over {R}iemann
  surfaces},
        date={1992},
        ISSN={0894-0347},
     journal={J. Amer. Math. Soc.},
      volume={5},
      number={4},
       pages={853\ndash 906},
         url={https://doi.org/10.2307/2152712},
      review={\MR{1145826}},
}

\bib{KM:prop-P}{article}{
      author={Kronheimer, P.~B.},
      author={Mrowka, T.~S.},
       title={Witten's conjecture and property {P}},
        date={2004},
        ISSN={1465-3060,1364-0380},
     journal={Geom. Topol.},
      volume={8},
       pages={295\ndash 310},
         url={https://doi.org/10.2140/gt.2004.8.295},
      review={\MR{2023280}},
}

\bib{km:monopole}{book}{
      author={Kronheimer, Peter},
      author={Mrowka, Tomasz},
       title={Monopoles and three-manifolds},
      series={New Mathematical Monographs},
   publisher={Cambridge University Press, Cambridge},
        date={2007},
      volume={10},
        ISBN={978-0-521-88022-0},
         url={https://doi.org/10.1017/CBO9780511543111},
      review={\MR{2388043}},
}

\bib{KM:alexander}{article}{
      author={Kronheimer, Peter},
      author={Mrowka, Tom},
       title={Instanton {F}loer homology and the {A}lexander polynomial},
        date={2010},
        ISSN={1472-2747},
     journal={Algebr. Geom. Topol.},
      volume={10},
      number={3},
       pages={1715\ndash 1738},
         url={https://doi.org/10.2140/agt.2010.10.1715},
      review={\MR{2683750}},
}

\bib{km-sutures}{article}{
      author={Kronheimer, Peter},
      author={Mrowka, Tomasz},
       title={Knots, sutures, and excision},
        date={2010},
        ISSN={0022-040X},
     journal={J. Differential Geom.},
      volume={84},
      number={2},
       pages={301\ndash 364},
         url={http://projecteuclid.org/euclid.jdg/1274707316},
      review={\MR{2652464}},
}

\bib{KM:unknot}{article}{
      author={Kronheimer, P.~B.},
      author={Mrowka, T.~S.},
       title={Khovanov homology is an unknot-detector},
        date={2011},
        ISSN={0073-8301},
     journal={Publ. Math. Inst. Hautes \'Etudes Sci.},
      number={113},
       pages={97\ndash 208},
         url={https://doi.org/10.1007/s10240-010-0030-y},
      review={\MR{2805599}},
}

\bib{KM:YAFT}{article}{
      author={Kronheimer, Peter},
      author={Mrowka, Tomasz},
       title={Knot homology groups from instantons},
        date={2011},
        ISSN={1753-8416},
     journal={J. Topol.},
      volume={4},
      number={4},
       pages={835\ndash 918},
         url={http://dx.doi.org/10.1112/jtopol/jtr024},
      review={\MR{2860345 (2012m:57060)}},
}

\bib{km-rasmussen}{article}{
      author={Kronheimer, P.~B.},
      author={Mrowka, T.~S.},
       title={Gauge theory and {R}asmussen's invariant},
        date={2013},
        ISSN={1753-8416},
     journal={J. Topol.},
      volume={6},
      number={3},
       pages={659\ndash 674},
         url={https://doi.org/10.1112/jtopol/jtt008},
      review={\MR{3100886}},
}

\bib{km-icm}{inproceedings}{
      author={Kronheimer, Peter~B.},
      author={Mrowka, Tomasz~S.},
       title={Knots, three-manifolds and instantons},
        date={2018},
   booktitle={Proceedings of the {I}nternational {C}ongress of
  {M}athematicians---{R}io de {J}aneiro 2018. {V}ol. {I}. {P}lenary lectures},
   publisher={World Sci. Publ., Hackensack, NJ},
       pages={607\ndash 634},
      review={\MR{3966740}},
}

\bib{km-embedded-i}{article}{
      author={Kronheimer, P.~B.},
      author={Mrowka, T.~S.},
       title={Gauge theory for embedded surfaces. {I}},
        date={1993},
        ISSN={0040-9383},
     journal={Topology},
      volume={32},
      number={4},
       pages={773\ndash 826},
         url={https://doi.org/10.1016/0040-9383(93)90051-V},
      review={\MR{1241873}},
}

\bib{km-structure}{article}{
      author={Kronheimer, P.~B.},
      author={Mrowka, T.~S.},
       title={Embedded surfaces and the structure of {D}onaldson's polynomial
  invariants},
        date={1995},
        ISSN={0022-040X},
     journal={J. Differential Geom.},
      volume={41},
      number={3},
       pages={573\ndash 734},
         url={http://projecteuclid.org/euclid.jdg/1214456482},
      review={\MR{1338483}},
}

\bib{king-newstead}{article}{
      author={King, A.~D.},
      author={Newstead, P.~E.},
       title={On the cohomology ring of the moduli space of rank {$2$} vector
  bundles on a curve},
        date={1998},
        ISSN={0040-9383},
     journal={Topology},
      volume={37},
      number={2},
       pages={407\ndash 418},
         url={https://doi.org/10.1016/S0040-9383(97)00041-4},
      review={\MR{1489212}},
}

\bib{kronheimer-higher}{article}{
      author={Kronheimer, P.~B.},
       title={Four-manifold invariants from higher-rank bundles},
        date={2005},
        ISSN={0022-040X},
     journal={J. Differential Geom.},
      volume={70},
      number={1},
       pages={59\ndash 112},
         url={http://projecteuclid.org/euclid.jdg/1143572014},
      review={\MR{2192061}},
}

\bib{lim}{article}{
      author={Lim, Yuhan},
       title={Instanton homology and the {A}lexander polynomial},
        date={2010},
        ISSN={0002-9939},
     journal={Proc. Amer. Math. Soc.},
      volume={138},
      number={10},
       pages={3759\ndash 3768},
         url={https://doi.org/10.1090/S0002-9939-2010-10412-4},
      review={\MR{2661575}},
}

\bib{munoz-basic}{article}{
      author={Mu\~{n}oz, Vicente},
       title={Basic classes for four-manifolds not of simple type},
        date={2000},
        ISSN={1019-8385},
     journal={Comm. Anal. Geom.},
      volume={8},
      number={3},
       pages={653\ndash 670},
         url={https://doi.org/10.4310/CAG.2000.v8.n3.a9},
      review={\MR{1775702}},
}

\bib{munoz}{article}{
      author={Mu\~{n}oz, Vicente},
       title={Ring structure of the {F}loer cohomology of {$\Sigma\times{\bf
  S}^1$}},
        date={1999},
        ISSN={0040-9383},
     journal={Topology},
      volume={38},
      number={3},
       pages={517\ndash 528},
         url={https://doi.org/10.1016/S0040-9383(98)00028-7},
      review={\MR{1670396}},
}

\bib{marino-moore}{article}{
      author={Mari\~{n}o, Marcos},
      author={Moore, Gregory},
       title={The {D}onaldson-{W}itten function for gauge groups of rank larger
  than one},
        date={1998},
        ISSN={0010-3616},
     journal={Comm. Math. Phys.},
      volume={199},
      number={1},
       pages={25\ndash 69},
         url={https://doi.org/10.1007/s002200050494},
      review={\MR{1660219}},
}

\bib{morganbass}{incollection}{
      author={Morgan, John~W.},
       title={The {S}mith conjecture},
        date={1984},
   booktitle={The {S}mith conjecture ({N}ew {Y}ork, 1979)},
      series={Pure Appl. Math.},
      volume={112},
   publisher={Academic Press, Orlando, FL},
       pages={3\ndash 6},
         url={https://doi.org/10.1016/S0079-8169(08)61632-3},
      review={\MR{758460}},
}

\bib{ni}{article}{
      author={Ni, Yi},
       title={Fintushel-{S}tern knot surgery in torus bundles},
        date={2017},
        ISSN={1753-8416},
     journal={J. Topol.},
      volume={10},
      number={1},
       pages={164\ndash 177},
         url={https://doi.org/10.1112/topo.12002},
      review={\MR{3653065}},
}

\bib{os-symplectic}{article}{
      author={Ozsv\'{a}th, Peter},
      author={Szab\'{o}, Zolt\'{a}n},
       title={Holomorphic triangle invariants and the topology of symplectic
  four-manifolds},
        date={2004},
        ISSN={0012-7094},
     journal={Duke Math. J.},
      volume={121},
      number={1},
       pages={1\ndash 34},
         url={https://doi.org/10.1215/S0012-7094-04-12111-6},
      review={\MR{2031164}},
}

\bib{scadutothesis}{article}{
      author={Scaduto, Christopher~W.},
       title={Instantons and odd {K}hovanov homology},
        date={2015},
        ISSN={1753-8416},
     journal={J. Topol.},
      volume={8},
      number={3},
       pages={744\ndash 810},
         url={https://doi.org/10.1112/jtopol/jtv012},
      review={\MR{3394316}},
}

\bib{sivek}{article}{
      author={Sivek, Steven},
       title={Donaldson invariants of symplectic manifolds},
        date={2015},
        ISSN={1073-7928},
     journal={Int. Math. Res. Not. IMRN},
      number={6},
       pages={1688\ndash 1716},
         url={https://doi.org/10.1093/imrn/rnt345},
      review={\MR{3340371}},
}

\bib{Smith:Lef-pen}{article}{
      author={Smith, Ivan},
       title={Lefschetz pencils and divisors in moduli space},
        date={2001},
        ISSN={1465-3060,1364-0380},
     journal={Geom. Topol.},
      volume={5},
       pages={579\ndash 608},
         url={https://doi.org/10.2140/gt.2001.5.579},
      review={\MR{1833754}},
}

\bib{siebert-tian}{article}{
      author={Siebert, Bernd},
      author={Tian, Gang},
       title={Recursive relations for the cohomology ring of moduli spaces of
  stable bundles},
        date={1995},
        ISSN={1300-0098},
     journal={Turkish J. Math.},
      volume={19},
      number={2},
       pages={131\ndash 144},
      review={\MR{1349566}},
}

\bib{zagier}{incollection}{
      author={Zagier, Don},
       title={On the cohomology of moduli spaces of rank two vector bundles
  over curves},
        date={1995},
   booktitle={The moduli space of curves ({T}exel {I}sland, 1994)},
      series={Progr. Math.},
      volume={129},
   publisher={Birkh\"{a}user Boston, Boston, MA},
       pages={533\ndash 563},
         url={https://doi.org/10.1007/978-1-4612-4264-2_20},
      review={\MR{1363070}},
}

\end{biblist}
\end{bibdiv}
\bibliographystyle{alpha.bst}
\Addresses

\end{document}